\documentclass[11pt, reqno, a4paper]{amsart}
\usepackage[margin=1.2in]{geometry}
\numberwithin{equation}{section}
\usepackage{amssymb,amsfonts,amsthm}
\usepackage[numbers,sort&compress]{natbib}
\usepackage[utf8]{inputenc}
\usepackage{listings}
\usepackage{bm}
\usepackage[hyperpageref]{backref}
\usepackage{esint}
\usepackage{color}
\usepackage{siunitx}
\usepackage{bigints}
\usepackage{tikz}
\usepackage{graphicx}
\usepackage{float}
\usetikzlibrary{arrows.meta,calc,decorations.pathreplacing,patterns}
\usepackage{subcaption}
\usepackage{mathtools}
\usepackage[
    colorlinks=true,
    linkcolor=blue,
    citecolor=red,
    urlcolor=blue,
    pdfborder={0 0 0}
]{hyperref}

\allowdisplaybreaks[4]

\newtheoremstyle{myremark}{10pt}{10pt}{}{}{\bfseries}{.}{.5em}{}

\newtheorem{theorem}{Theorem}[section]

\newtheorem{lemma}[theorem]{Lemma}
\newtheorem{proposition}[theorem]{Proposition}

\theoremstyle{definition}
\newtheorem{definition}[theorem]{Definition}
\newtheorem{remark}{Remark}[section]

\usepackage{hyperref}

\allowdisplaybreaks[4]

\begin{document}

\title[Quantitative Stability for weighted Hardy inequalities]{Quantitative Stability for Weighted Hardy Inequalities: Local and nonlocal Cases}

\author[V{.} Sahu]{Vivek Sahu}

\address{ Theoretical Statistics and Mathematics Unit,
Indian Statistical Institute, Delhi Centre, S.J. Sansanwal Marg, New Delhi, Delhi 110016, India}
\email{vivek@isid.ac.in,  viiveksahu@gmail.com}

\subjclass[2020]{26D10; 46E35; 35A23}

\keywords{Quantitative stability, Weighted Hardy inequalities, Weighted spaces}

\date{}

\dedicatory{}

\begin{abstract}
We establish quantitative stability estimates for weighted Hardy inequalities in both the local and the fractional setting, showing that the Hardy deficit of a normalized admissible function controls its distance to the associated family of virtual extremizers. In the local case, for $p>1$, $\alpha\in[0,1)$ with $0<p-\alpha p<N$, and $u\in C_c^1(\mathbb{R}^N)$ normalized so that $\int_{\mathbb{R}^N}|u|^p|x|^{\alpha p-p}\,dx=1$, we prove
\begin{equation*}
\int_{\mathbb{R}^N}|\nabla u(x)|^p|x|^{\alpha p}\,dx-\Big(\frac{N-p+\alpha p}{p}\Big)^p\int_{\mathbb{R}^N}\frac{|u(x)|^p}{|x|^{p-\alpha p}}\,dx
\geq C\,\operatorname{dist}_{\alpha}(u,\mathcal{M}_\alpha)^{\max\{4,2p\}},  
\end{equation*}
where $\mathcal{M}_\alpha$ is the family generated by the virtual extremizer and $\operatorname{dist}_{\alpha}(u,\mathcal{M}_\alpha)$ is a scale-invariant Lorentz-type distance. We extend this to the weighted fractional Hardy inequality, for $s\in(0,1)$, $\alpha=\alpha_1+\alpha_2\ge0$ with $\alpha_1p,\alpha_2p\in(-N,sp)$, $0<sp-\alpha p<N$, and $u$ normalized analogously,
\begin{equation*}
    \int_{\mathbb{R}^{N}} \int_{\mathbb{R}^{N}}\frac{|u(x)-u(y)|^p}{|x-y|^{N+sp}}|x|^{\alpha_1p}|y|^{\alpha_2p}\,dx\,dy- \mathcal{C} \int_{\mathbb{R}^N}\frac{|u(x)|^p}{|x|^{sp-\alpha p}}\,dx
\geq C\,\operatorname{dist}_{s,\alpha}(u, \mathcal{M}_{s, \alpha})^{\max\{4,2p\}},
\end{equation*}
with $\operatorname{dist}_{s,\alpha}(u, \mathcal{M})$ the analogous distance to $\mathcal{M}_{s,\alpha}$, measured on $u$ for $p\ge2$ and on a power-type transformation of $u$ for $1<p<2$. Both estimates recover the known unweighted results when the weight parameters vanish. We also establish weighted Sobolev and weighted fractional Sobolev extension theorems on bounded smooth domains avoiding the origin. In addition, we establish weighted Sobolev and fractional Sobolev inequalities on bounded smooth domains away from the origin. The proof of the quantitative stability results is rearrangement-free, combining these inequalities with scale-invariant Poincaré--Sobolev estimates on annuli, a telescoping oscillation decomposition, and weighted Lorentz-space embeddings. The dichotomy between $p\ge2$ and $1<p<2$ arises from the convexity estimate used to extract a coercive remainder from the deficit.
\end{abstract}

\maketitle

\tableofcontents


\section{Introduction}\label{Section 1}

Hardy inequalities are fundamental functional inequalities in analysis and
partial differential equations. A notable feature of the sharp Hardy
inequality is the non-attainment of its optimal constant in the natural
Sobolev space. This naturally raises the question of how close a function
with a small deficit is to the corresponding extremal profile. Quantitative
stability results address this question by relating the deficit to a
suitable distance from the set of extremals, or from virtual extremals when
the optimal constant is not attained. Such estimates provide a quantitative
description of near-extremizers and have attracted considerable interest in
the study of functional inequalities.

\subsection{Classical Hardy inequalities}
Let $N\geq 2$ and $1<p<N$. The classical Hardy inequality on
$\mathbb{R}^{N}$ states that
\begin{equation}\label{eq:classical-hardy}
    \int_{\mathbb{R}^{N}} |\nabla u(x)|^{p}\,dx
    \geq
    \left(\frac{N-p}{p}\right)^{p}
    \int_{\mathbb{R}^{N}}
    \frac{|u(x)|^{p}}{|x|^{p}}\,dx,
\end{equation}
for every $u\in\mathcal{D}^{1,p}(\mathbb{R}^{N})$, where
$\mathcal{D}^{1,p}(\mathbb{R}^{N})$ denotes the completion of
$C_{c}^{\infty}(\mathbb{R}^{N})$ with respect to
$\|\nabla u\|_{L^{p}(\mathbb{R}^{N})}$. The constant in
\eqref{eq:classical-hardy} is optimal but is not attained by any nonzero
function in $\mathcal{D}^{1,p}(\mathbb{R}^{N})$. The Euler--Lagrange equation associated with the above Hardy inequality \eqref{eq:classical-hardy} is satisfied by the function
\begin{equation}\label{eq:subcritical-virtual-extremals}
    v_a(x)
    =
    a \, |x|^{-\frac{N-p}{p}},
    \qquad a\in\mathbb{R}\setminus\{0\},
\end{equation}
which does not belong to $\mathcal{D}^{1,p}(\mathbb{R}^{N})$. Therefore, we refer to $v_a$ as a \textit{virtual extremizer}. To quantify the stability of \eqref{eq:classical-hardy}, we consider the
Hardy deficit
\begin{equation*}
   \delta_{p}(u)
    :=
    \int_{\mathbb{R}^{N}}|\nabla u(x)|^{p}\,dx
    -
    \left(\frac{N-p}{p}\right)^{p}
    \int_{\mathbb{R}^{N}}
    \frac{|u(x)|^{p}}{|x|^{p}}\,dx
\end{equation*}
and the family of virtual extremals $\mathcal{Z}_{p}:=    \left\{a\, |x|^{-\frac{N-p} {p}}:a\in\mathbb{R}\right\}$.

In this direction, Cianchi and Ferone
\cite[Theorem~$1.1$]{CianchiFerone} established a quantitative stability
estimate for the classical Hardy inequality. Since the virtual extremals
$v_a$ belong to the weak Lorentz space $L^{p^{*},\infty}(\mathbb{R}^{N})$, but $v_{a} \notin L^{p^{*},p}(\mathbb{R}^{N})$ (see Definition \ref{defi-lorentz} for the definition of Lorentz space), where $p^{*} = \frac{Np}{N-p}$, they introduced the normalized distance
\begin{equation*}
    d_{p}(u)
    :=
    \inf_{a\in\mathbb{R}}
    \frac{
        \|u-v_a\|_{L^{p^{*},\infty}(\mathbb{R}^{N})}
    }{
        \|u\|_{L^{p^{*},p}(\mathbb{R}^{N})}
    }.
\end{equation*}
They proved that there exists a constant $C=C(N,p)>0$ such that
\begin{equation}\label{eq:Cianchi-Ferone-subcritical}
    \delta_{p}(u)
    \geq
    C
    \left(
        \int_{\mathbb{R}^{N}}
        \frac{|u(x)|^{p}}{|x|^{p}}\,dx
    \right)
    d_{p}(u)^{2p^{*}}.
\end{equation}
Thus, a small Hardy deficit quantitatively forces $u$ to be close to the
family of virtual extremals in the weak Lorentz space. 

\smallskip

In a recent work, the author with Banerjee, and Ganguly \cite{BanerjeeGangulySahu} obtained a quantitative stability result for the fractional Hardy inequality. Their method further yielded an improvement of the quantitative stability estimate for the classical Hardy inequality, specifically through a sharper exponent in the distance from the family of virtual extremals. More precisely, the exponent $2p^{*}$ appearing in the result of Cianchi--Ferone was replaced by $\max\{4,2p\}$. Consequently, for $N\geq2$ and $1<p<N$, there exists a constant $C=C(N,p)>0$ such that
\begin{equation*}
    \delta_{p}(u)
    \geq
    C\left(
    \int_{\mathbb{R}^{N}}
    \frac{|u(x)|^{p}}{|x|^{p}}\,dx
    \right)
    d_{p}(u)^{\max\{4,2p\}}.
\end{equation*}

\subsubsection{\textbf{Main results}} In this article, we first consider weighted versions of the Hardy inequalities and establish corresponding quantitative stability estimates. More precisely, we study the weighted local energy associated with power weights of the form $|x|^{\alpha p}$ and examine its relation to the corresponding weighted Hardy term. The presence of power weights in the gradient term naturally leads to a weighted functional framework that is compatible with the scaling properties of the inequality. 

\smallskip

For $p\geq 1$ and $\alpha\in\mathbb{R}$, we introduce the weighted Lebesgue space
\begin{equation*}
L^{p,\alpha}(\mathbb{R}^{N})
:=
\left\{
u:\mathbb{R}^{N}\to\mathbb{R}\ \text{measurable}:
\int_{\mathbb{R}^{N}}
|u(x)|^{p}|x|^{\alpha p}\,dx<\infty
\right\},
\end{equation*}
endowed with the norm
\begin{equation*}
\|u\|_{L^{p,\alpha}(\mathbb{R}^{N})}
:=
\left(
\int_{\mathbb{R}^{N}}
|u(x)|^{p}|x|^{\alpha p}\,dx
\right)^{\frac1p}.
\end{equation*}
For $u\in L^{p,\alpha}(\mathbb{R}^{N})$, we define the weighted gradient seminorm by
\begin{equation*}
    \| \nabla u \|_{L^{p, \alpha}(\mathbb{R}^{N})} := \left( \int_{\mathbb{R}^{N}} |\nabla u(x)|^{p}|x|^{\alpha p} \, dx \right)^{\frac{1}{p}}.
\end{equation*}
Accordingly, we define the weighted Sobolev norm by
\begin{equation*}
    \| u\|_{W^{1,p, \alpha}(\mathbb{R}^{N})}:= \left( \|u\|^{p}_{L^{p,\alpha}(\mathbb{R}^{N})} +  \| \nabla u \|^{p}_{L^{p, \alpha}(\mathbb{R}^{N})} \right)^{\frac{1}{p}}.
\end{equation*}
If $\alpha p>-N$, then the weight $|x|^{\alpha p}$ is locally integrable in $\mathbb{R}^{N}$. Consequently, for every $u\in C_c^1(\mathbb{R}^{N})$, the above weighted Sobolev norm is finite. We therefore assume $\alpha \geq 0$ and define $W^{1,p, \alpha}(\mathbb{R}^{N})$ as the completion of $C^{1}_{c}(\mathbb{R}^{N})$ with respect to the norm $\| \cdot \|_{W^{1,p, \alpha}(\mathbb{R}^{N})}$.

\smallskip

The weighted version of the classical Hardy inequality was established by Secchi, Smets, and Willem \cite{Secchi}. In \cite[Theorem $2.1$]{Secchi}, they proved that, for $1<p<\infty$ and $0<p- \alpha p<N$, every $u\in W^{1,p, \alpha}(\mathbb{R}^N)$ satisfies
\begin{equation}\label{Weighted Hardy inequality}
  \int_{\mathbb{R}^N}
|\nabla u(x)|^p |x|^{\alpha p}\,dx \geq \left( \frac{N-p+ \alpha p}{p} \right)^{p} \int_{\mathbb{R}^N} \frac{|u(x)|^{p}}{|x|^{p- \alpha p}} \, dx, 
\end{equation}
where the constant $\left( \frac{N-p+ \alpha p}{p} \right)^{p}$ is optimal, but it is never attained in $W^{1,p, \alpha}(\mathbb{R}^{N})$. The function
\begin{equation}\label{omega alpha}
    v_{\alpha}(x)
:=
|x|^{-\frac{N-p+\alpha p}{p}},
\end{equation}
which satisfies the corresponding Euler--Lagrange equation associated with the weighted Hardy inequality \eqref{Weighted Hardy inequality}, i.e.,
\begin{equation}\label{Euler-Lagrange weighted Hardy}
    -\operatorname{div}\left(
|x|^{\alpha p}|\nabla v_\alpha|^{p-2}
\nabla v_\alpha
\right)
=
\left(\frac{N-p+\alpha p} {p}\right)^p \frac{v_\alpha^{p-1}}{|x|^{p- \alpha p}}
\qquad\text{in }\mathbb{R}^N\setminus\{0\}.
\end{equation}
Although $v_{\alpha}$ satisfies the Euler--Lagrange equation associated
with \eqref{Weighted Hardy inequality}, it does not
belong to the natural energy space $W^{1,p,\alpha}(\mathbb{R}^{N})$. Indeed,
its weighted seminorm satisfies
\begin{equation*}
\int_{\mathbb{R}^N}
|\nabla v_{\alpha}(x)|^p |x|^{\alpha p}\,dx \sim\int_{\mathbb{R}^{N}}\frac{dx}{|x|^{N}}
    =\infty.
\end{equation*}
Thus, $v_{\alpha}$ is not an admissible function in the natural
energy space, despite satisfying the corresponding Euler--Lagrange equation.
For this reason, we refer to $v_{\alpha}$ as the \emph{virtual
extremizer} associated with the weighted fractional Hardy inequality
\eqref{Weighted Hardy inequality}.

The non-attainment of the optimal constant in \eqref{Weighted Hardy inequality} naturally raises the question of quantitative stability. In particular, it is desirable to measure how close a function is to the family of virtual extremizers in
terms of the deficit of the inequality. Accordingly, we define the weighted Hardy deficit by
\begin{equation}
    \delta_{p, \alpha}(u) := \int_{\mathbb{R}^N}
|\nabla u(x)|^p |x|^{\alpha p}\,dx - \left( \frac{N-p+ \alpha p}{p} \right)^{p} \int_{\mathbb{R}^N} \frac{|u(x)|^{p}}{|x|^{p- \alpha p}} \, dx,
\end{equation}
and the family of virtual extremals $\mathcal{M}_{\alpha}:=    \left\{a \, v_{\alpha} :a\in\mathbb{R}\right\}$.

We have shown in Lemma \ref{Marcinkiewicz} with $s=1$, that the virtual extremizer
$v_{\alpha}$ belongs to the weak Lebesgue space
$L^{p^{*}_{\alpha},\infty}(\mathbb{R}^{N})$, where
\begin{equation*}
 p^{*}_{\alpha} := \frac{Np}{N-p+\alpha p}  .
\end{equation*}
This weak integrability property allows us to measure the distance of an admissible function from
the one-dimensional family generated by the virtual extremizer. Accordingly,
we define the following scale-invariant distance functional
\begin{equation}\label{Defn: d_p, alpha}
d_{p,\alpha}(u)
    :=
    \inf_{a\in\mathbb{R}}
    \frac{
        \|u-a \, v_{\alpha}\|_{L^{p^{*}_{\alpha}, \infty}(\mathbb{R}^{N})}
    }{
\|u\|_{L^{p^{*}_{\alpha},p}(\mathbb{R}^{N})}
    }.
\end{equation}
Here, the numerator measures the deviation of $u$ from the family of virtual extremizers, while the denominator provides the natural normalization compatible with the scaling of the weighted fractional Hardy
inequality.

We now state the main quantitative stability result of this article in the
range $p>1, \, \alpha \in [0,1)$ and $0<p-\alpha p<N$. It shows that the weighted Hardy deficit controls
the distance of an admissible function from the virtual extremizer family.
More precisely, we have the following theorem.
\begin{theorem}\label{Theorem: Quantitative Stabiliy Weighted Hardy inequality}
Let $p>1$, $\alpha \in [0, 1) $ be such that $0< p-\alpha p <N$. Then, for every
$u\in C^{1}_{c}(\mathbb{R}^{N})$, there exists a constant $C=C(N,p,\alpha)>0$ such that
\begin{equation}\label{Eq: Quantitative stability weighted Hardy}
    \delta_{p,\alpha}(u)    \geq C\left(\int_{\mathbb{R}^{N}}
        \frac{|u(x)|^{p}}
        {|x|^{p-\alpha p}}
        \,dx
    \right)
    d_{p,\alpha}(u)^{\max \{ 4, 2p \}}.
\end{equation}
\end{theorem}

The theorem above reduces to the quantitative stability result for the Hardy inequality, \eqref{eq:classical-hardy}, by taking $\alpha=0$. Therefore, the above theorem establishes quantitative stability for the weighted Hardy inequality, which also includes the unweighted Hardy inequality.

\begin{remark}
The results stated in Theorem \ref{Theorem: Quantitative Stabiliy Weighted Hardy inequality} yield quantitative
stability estimates for the weighted Hardy inequality. Indeed, \eqref{Eq: Quantitative stability weighted Hardy} establish that
the deficit is nonnegative and, moreover, provides a quantitative measure
of the distance between $u$ and the family of virtual extremals
$\{a\, v_{\alpha}\}_{a \in \mathbb{R}\setminus\{0\}}$.
Hence, any near-extremizing sequence associated with
\eqref{Weighted Hardy inequality} approaches this
family in the corresponding Lorentz-type sense. We can normalize $u$ by assuming that
$
\int_{\mathbb{R}^N} \frac{|u(x)|^p}{|x|^{p-\alpha p}} \, dx = 1$.
With this normalization, the conclusions of the above theorem can
be expressed in the following form
\begin{equation*}
  \delta_{p,\alpha}(u) \geq d_{p, \alpha} (u)^{\max \{ 4, 2p \}}.
\end{equation*}
\end{remark}

The classical weighted interpolation inequality established by Caffarelli, Kohn, and Nirenberg \cite{Caffarelli1984} is now known as the Caffarelli--Kohn--Nirenberg inequality. In Section \ref{Section 3}, we establish a weighted extension theorem for bounded $C^{1}$ domain  $\Omega\subset\mathbb{R}^N$ such that $0\notin \overline{\Omega}$. In particular, we show that functions in the weighted Sobolev space can be extended to the whole space with the extension norm controlled by the original norm (see Theorem \ref{Theorem: Weighted Local Extension theorem}). We combine this extension theorem with the Caffarelli--Kohn--Nirenberg inequality to derive the weighted Sobolev inequality on bounded $C^{1}$ domains $\Omega\subset\mathbb{R}^N$ satisfying $0\notin \overline{\Omega}$ (see Theorem \ref{Theorem: Weighted Sobolev inequality on bounded domain}).

More precisely, under the assumptions $p >1, \, \alpha \in [0,1)$, $p-\alpha p <N$, and $p^{*}_{\alpha} = \frac{Np}{N-p+\alpha p}$
we establish the following weighted Sobolev inequality
\begin{equation*}
\| u\|_{L^{p^{*}_{\alpha}}(\Omega)} \leq  C \|u\|_{W^{1,p,\alpha}(\Omega)},
\qquad \forall \, u\in W^{1,p,\alpha}(\Omega).
\end{equation*}
This inequality provides the weighted Sobolev estimate needed in the subsequent proof of the main results of this paper.

\smallskip

For further developments on classical Hardy inequalities, including geometric improvements, sharp constants, we refer to 
\cite{chaudhuri2002improved, Adimurthi2005, Adimurthi2009, akutagawa2013geometric, Barbatis2003, Barbatis2004, berchio2017sharp, berchio2020optimal, Brezis1997-2, Brezis2000, Brezis1997, carron1997inegalites, chen2023sharp, chen2024stability, Chen2026MathAnn, delPino2010, Filippas2002, sahu2026improvedquantitativestabilitycritical, Opic1990}. For quantitative stability results in various Sobolev and Hardy-Sobolev inequalities, we refer to \cite{Debdip2026, Bianchi1991, Cianchi2006, Cianchi2009, Dolbeault2025, Figalli2019, Figalli2022,  Fusco2007, Fusco2008, Neumayer2020}

\subsection{Fractional Hardy inequalities}

The classical Hardy inequality has a natural nonlocal counterpart, namely the fractional Hardy inequality. For $0<s<1$ and $1\leq p<\frac{N}{s}$, Frank and Seiringer \cite{frank2008} established the sharp inequality
\begin{equation}\label{eq:fractional-hardy}
    \int_{\mathbb{R}^{N}}\int_{\mathbb{R}^{N}}
    \frac{|u(x)-u(y)|^{p}}{|x-y|^{N+sp}}\,dx\,dy
    \geq
    C_{N,s,p}
    \int_{\mathbb{R}^{N}}
    \frac{|u(x)|^{p}}{|x|^{sp}}\,dx,
    \qquad
    \forall\,u\in W^{s,p}(\mathbb{R}^{N}),
\end{equation}
where $C_{N,s,p}>0$ is the sharp constant. As in the classical case, the sharp constant is not attained in the natural fractional Sobolev space. The Euler--Lagrange equation associated with the above Hardy inequality \eqref{eq:fractional-hardy} is satisfied by the function
\begin{equation*}
    \omega_a(x)
    =
    a\,|x|^{-\frac{N-sp}{p}},
    \qquad a\in\mathbb{R}\setminus\{0\},
\end{equation*}
which does not belong to $W^{s,p}(\mathbb{R}^{N})$. Therefore, we refer to $\omega_a$ as a \textit{virtual extremizer}. This naturally leads to the question of whether the deficit in \eqref{eq:fractional-hardy} can quantitatively control the distance of a function from the family of virtual extremals.

Recently, the author with Banerjee, and Ganguly \cite{BanerjeeGangulySahu} established quantitative stability estimates for \eqref{eq:fractional-hardy} by developing an entirely rearrangement-free approach. The corresponding virtual extremals
\begin{equation*}
    \omega_a(x)
    =
    a\,|x|^{-\frac{N-sp}{p}},
    \qquad a\in\mathbb{R}\setminus\{0\},
\end{equation*}
belong to the weak Lebesgue space $L^{p_s^*,\infty}(\mathbb{R}^{N})$, but not to $L^{p_s^*, p}(\mathbb{R}^{N})$, where $p_s^*=\frac{Np}{N-sp}$ (see Definition \ref{defi-lorentz} for the definitions of weak Lebesgue and Lorentz spaces).

Define the fractional Hardy deficit by
\begin{equation*}
    \delta_{s,p}(u)
    :=
    \int_{\mathbb{R}^{N}}\int_{\mathbb{R}^{N}}
    \frac{|u(x)-u(y)|^{p}}{|x-y|^{N+sp}}\,dx\,dy
    -
    C_{N,s,p}
    \int_{\mathbb{R}^{N}}
    \frac{|u(x)|^{p}}{|x|^{sp}}\,dx.
\end{equation*}
For $p\geq2$, the distance from $u$ to the family of virtual extremals is defined by
\begin{equation*}
    d_{s,p}(u)
    :=
    \inf_{a\in\mathbb{R}}
    \frac{
        \|u-\omega_a\|_{L^{p_s^*,\infty}(\mathbb{R}^{N})}
    }{
        \|u\|_{L^{p_s^*,p}(\mathbb{R}^{N})}
    },
\end{equation*}
where $\|\cdot\|_{L^{p_s^*,p}}$ denotes the Lorentz norm, while
$\|\cdot\|_{L^{p_s^*,\infty}}$ denotes the weak-$L^{p_s^*}$ (Marcinkiewicz) norm (see Definition \ref{defi-lorentz}). The approach is based on a scale-invariant Poincar\'e--Sobolev inequality, a decomposition of $\mathbb{R}^{N}$ into concentric annuli, and suitable Lorentz-space estimates. In particular, for $p\geq2$, they proved
\begin{equation}\label{quantitative frac Hardy 1}
    \delta_{s,p}(u)
    \geq
    C
    \left(
        \int_{\mathbb{R}^{N}}
        \frac{|u(x)|^{p}}{|x|^{sp}}\,dx
    \right)
    d_{s,p}(u)^{2p}.
\end{equation}

For the range $1<p<2$, the above distance is no longer suitable for their argument. Setting $ q=\frac{2N}{N-sp}$, they introduced the modified distance
\begin{equation*}
    \widetilde{d}_{s,p}(u)
    :=
    \inf_{a\in\mathbb{R}}
    \frac{
        \left\|
        u^{\langle \frac{p}{2} \rangle}
        -
        \omega^{\langle \frac{p}{2} \rangle}_a
        \right\|_{L^{q,\infty}(\mathbb{R}^{N})}
    }{
        \|u\|_{L^{p_s^*,p}(\mathbb{R}^{N})}^{\frac{p}{2}}
    }, \quad a^{\langle \frac{p}{2} \rangle} = |a|^{\frac{p}{2}} \operatorname{sgn} (a).
\end{equation*}
They then proved that, for $1<p<2$,
\begin{equation}\label{quantitative frac Hardy 2}
    \delta_{s,p}(u)
    \geq
    C \left( \int_{\mathbb{R}^{N}}
    \frac{|u(x)|^{p}}{|x|^{sp}}\,dx \right)
    \widetilde{d}_{s,p}(u)^{4}  .  
\end{equation}
Thus, combining the two regimes, the fractional Hardy deficit controls the distance to the family of virtual extremals with exponent $\alpha=\max\{4,2p\}$, with the distance understood as $d_{s,p}$ for $p\geq2$ and as $\widetilde{d}_{s,p}$ for $1<p<2$.

Hence, a small fractional Hardy deficit quantitatively forces $u$ to be close to the family of virtual extremals in an appropriate Lorentz-type sense, providing a quantitative description of near-extremizers.

\subsubsection{\textbf{Main results}}

In this article, we focus on establishing quantitative stability estimates for a weighted version of the fractional Hardy inequality. More precisely, we investigate how the deficit in the weighted fractional Hardy inequality controls the deviation from equality. Our approach is motivated by the rearrangement-free framework developed in \cite{BanerjeeGangulySahu} for the fractional Hardy inequality. We extend this framework to the weighted setting by combining scale-invariant estimates with suitable weighted Lorentz-space inequalities and a careful analysis of the weighted nonlocal energy. In particular, the presence of two different power weights in the fractional Gagliardo seminorm requires a suitable choice of the underlying weighted functional setting.

\smallskip

For $p\geq 1$, $s\in(0,1)$, and $\alpha,\alpha_{1},\alpha_{2}\in\mathbb{R}$ satisfying $\alpha=\alpha_{1}+\alpha_{2}$, we first introduce the weighted Lebesgue space
\begin{equation*}
\mathcal{L}^{p,\alpha}(\mathbb{R}^{N})
:=
\left\{
u:\mathbb{R}^{N}\to\mathbb{R}\ \text{measurable}:
\int_{\mathbb{R}^{N}}
|u(x)|^{p}
\left(
|x|^{\alpha_{1}p}+|x|^{\alpha_{2}p}
\right)\,dx<\infty
\right\},
\end{equation*}
equipped with the norm
\begin{equation*}
\|u\|_{\mathcal{L}^{p,\alpha}(\mathbb{R}^{N})}
:=
\left(
\int_{\mathbb{R}^{N}}
|u(x)|^{p}
\left(
|x|^{\alpha_{1}p}+|x|^{\alpha_{2}p}
\right)\,dx
\right)^{\frac1p}.
\end{equation*}
For $u\in \mathcal{L}^{p,\alpha}(\mathbb{R}^{N})$, we define the weighted Gagliardo seminorm by
\begin{equation}\label{Weighted Gagliardo Seminorm}
[u]_{W^{s,p,\alpha}(\mathbb{R}^{N})}
:=
\left(
\int_{\mathbb{R}^{N}}\int_{\mathbb{R}^{N}}
\frac{|u(x)-u(y)|^{p}}
{|x-y|^{N+sp}}
|x|^{\alpha_{1}p}
|y|^{\alpha_{2}p}
\,dx\,dy
\right)^{\frac1p}.
\end{equation}
Dipierro and Valdinoci \cite[Lemma $2.1$]{Valdinoci2015} proved that if $\alpha_{1}p, \, \alpha_{2}p \in (-N, sp)$ and $\alpha_{1}p+ \alpha_{2}p > -N$, then for every $u \in C^{1}_{c}(\mathbb{R}^{N})$, the seminorm $[u]_{W^{s,p, \alpha}(\mathbb{R}^{N})} < \infty$. We therefore define the weighted fractional Sobolev space $W^{s,p,\alpha}(\mathbb{R}^{N})$ as the completion of $C^{1}_{c}(\mathbb{R}^{N})$ with respect to the norm
\begin{equation}\label{Weighted norm definition}
\|u\|_{W^{s,p,\alpha}(\mathbb{R}^{N})}
:=
\left(
\|u\|_{\mathcal{L}^{p,\alpha}(\mathbb{R}^{N})}^{p}
+
[u]_{W^{s,p,\alpha}(\mathbb{R}^{N})}^{p}
\right)^{\frac1p}.
\end{equation}
For recent developments in  weighted Gagliardo seminorm and their geometric improvements, we refer to \cite{Ao2022, KijaczkoSNS, Kijaczko2025AMPA,  VivekCCM2026,  VivekJMAA2025}.

\smallskip

The weighted version of the fractional Hardy inequality with point singularity was established by Dyda and Kijaczko \cite{dyda2024}. Using a ground state representation, they obtained the weighted Hardy inequality with an optimal constant. More precisely, let $\alpha=\alpha_{1}+ \alpha_{2}$, $s\in(0,1)$, $p\geq1$, and  $\alpha_{1} p , \, \alpha_{2} p, \, \alpha p \in (-N, sp)$. For all $u \in C_{c}(\mathbb{R}^{N})$ when $sp - \alpha p < N$, and for all $u \in C_{c}(\mathbb{R}^{N} \setminus \{0\})$ when $sp - \alpha p > N$, the following inequality holds with an optimal constant $\mathcal{C} > 0$:
\begin{equation}\label{Weighted fractional Hardy : point singularity}
\int_{\mathbb{R}^{N}}\int_{\mathbb{R}^{N}} \frac{|u(x)-u(y)|^{p}}{|x-y|^{N+sp}} |x|^{\alpha_{1} p} |y|^{\alpha_{2} p} \, dx \, dy \geq  \mathcal{C}\int_{\mathbb{R}^{N}} \frac{|u(x)|^{p}}{|x|^{sp-\alpha p}} \, dx ,
\end{equation}
where
\begin{equation*}\label{The value C 1}
\mathcal{C} = \mathcal{C}(N,s,p, \alpha_{1}, \alpha_{2}) = \int_{0}^{1} r^{sp-1} (r^{-\alpha_{1} p} +r^{-\alpha_{2} p
}) \left| 1-r^{(N+\alpha_{1} p+\alpha_{2} p-sp)/p} \right|^{p} \Phi_{N,s,p}(r) \, dr.
\end{equation*}
Here, 
\begin{equation*}
\Phi_{N,s,p}(r) = \begin{dcases}
        |\mathbb{S}^{N-2}| \int_{-1}^{1} \frac{(1-t^{2})^{\frac{N-3}{2}}}{(1-2tr+r^{2})^{\frac{N+sp}{2}}} \, dt, & N\geq 2 \\ 
        (1-r)^{-1-sp}+ (1+r)^{-1-sp}, & N = 1.
    \end{dcases}
\end{equation*}

Furthermore, under the above assumptions, they showed that the function
\begin{equation*}
    \omega_{s,\alpha}(x)
    :=
    |x|^{-\frac{N-sp+\alpha p}{p}},
    \qquad
    \alpha=\alpha_{1}+\alpha_{2},
\end{equation*}
satisfies the Euler--Lagrange equation associated with
\eqref{Weighted fractional Hardy : point singularity}. More precisely,
\begin{equation*}
2\lim_{\varepsilon\to0}
\int_{\lvert |x|-|y|\rvert>\varepsilon}
(\omega_{s, \alpha}(x)-\omega_{s, \alpha}(y))
|\omega_{s, \alpha}(x)-\omega_{s, \alpha}(y)|^{p-2}
k(x,y)\,dy
=
\mathcal{C}
\frac{\omega_{s, \alpha}(x)^{p-1}}
{|x|^{sp-\alpha p}},
\end{equation*}
uniformly on compact subsets of
$\mathbb{R}^{N}\setminus\{0\}$, where
\begin{equation*}
k(x,y) := \frac{1}{2}|x-y|^{-N-sp} \left(|x|^{\alpha_{1}p}|y|^{\alpha_{2}p}+|x|^{\alpha_{2}p}|y|^{\alpha_{1}p}\right).
\end{equation*}
Although $\omega_{s,\alpha}$ satisfies the Euler--Lagrange equation associated
with \eqref{Weighted fractional Hardy : point singularity}, it does not
belong to the natural energy space $W^{s,p,\alpha}(\mathbb{R}^{N})$. Indeed,
its weighted fractional Gagliardo seminorm satisfies
\begin{equation*}
\int_{\mathbb{R}^{N}}\int_{\mathbb{R}^{N}}
    \frac{|\omega_{s,\alpha}(x)-\omega_{s,\alpha}(y)|^{p}}
    {|x-y|^{N+sp}}|x|^{\alpha_{1}p}|y|^{\alpha_{2}p}\,dx\,dy \sim\int_{\mathbb{R}^{N}}\frac{dx}{|x|^{N}}
    =\infty.
\end{equation*}
Thus, $\omega_{s,\alpha}$ is not an admissible function in the natural
energy space, despite satisfying the corresponding Euler--Lagrange equation.
For this reason, we refer to $\omega_{s,\alpha}$ as the \emph{virtual
extremizer} associated with the weighted fractional Hardy inequality
\eqref{Weighted fractional Hardy : point singularity}.

The non-attainment of the optimal constant in \eqref{Weighted fractional Hardy : point singularity} naturally raises the question of quantitative stability. In particular, it is desirable to measure how close a function is to the family of virtual extremizers in
terms of the deficit of the inequality. Accordingly, we define the weighted fractional Hardy deficit by
\begin{equation}\label{Weighted Frac Hardy deficit}
\delta_{s,p,\alpha}(u)
:=
\int_{\mathbb{R}^{N}}\int_{\mathbb{R}^{N}}
\frac{|u(x)-u(y)|^{p}}
{|x-y|^{N+sp}}
|x|^{\alpha_{1}p}|y|^{\alpha_{2}p}
\,dx\,dy
-
\mathcal{C}
\int_{\mathbb{R}^{N}}
\frac{|u(x)|^{p}}
{|x|^{sp-\alpha p}}
\,dx,
\end{equation}
and the family of virtual extremals $\mathcal{M}_{s,\alpha}:=    \left\{a\, \omega_{s, \alpha}:a\in\mathbb{R}\right\}$.

We have shown in Lemma \ref{Marcinkiewicz} that the virtual extremizer
$\omega_{s,\alpha}$ belongs to the weak Lebesgue space
$L^{p^{*}_{s,\alpha},\infty}(\mathbb{R}^{N})$, where
\begin{equation*}
   p^{*}_{s,\alpha} := \frac{Np}{N-sp+ \alpha p}, \quad 0<sp-\alpha p<N. 
\end{equation*}
This weak integrability property allows us to measure the distance of an admissible function from
the one-dimensional family generated by the virtual extremizer. Accordingly,
we define the following scale-invariant distance functional
\begin{equation*}
    \mathcal{D}_{s,p,\alpha}(u)
    :=
    \inf_{a\in\mathbb{R}}
    \frac{
        \|u-a \, \omega_{s,\alpha}\|_{L^{p^{*}_{s,\alpha}, \infty}(\mathbb{R}^{N})}
    }{
        \|u\|_{L^{p^{*}_{s,\alpha},p}(\mathbb{R}^{N})}
    }.
\end{equation*}
Here, the numerator measures the deviation of $u$ from the family of virtual extremizers, while the denominator provides the natural normalization compatible with the scaling of the weighted fractional Hardy
inequality.

We now state the main quantitative stability result for the weighted fractional Hardy inequality in the range $p\geq2$. It shows that the weighted fractional Hardy deficit controls
the distance of an admissible function from the virtual extremizer family.
More precisely, we have the following theorem.

\begin{theorem}\label{Theorem: Quantitative frac hardy p geq 2}
Let $p\geq2$, $s\in(0,1)$, and let $\alpha_{1},\alpha_{2}\in\mathbb{R}$ with
$\alpha=\alpha_{1}+\alpha_{2} \geq 0$ satisfy $\alpha_{1}p,\,\alpha_{2}p\in(-N,sp)$ and $ 0<sp-\alpha p<N$. Then, for every
$u\in C^{1}_{c}(\mathbb{R}^{N})$, there exists a constant $C=C(N,p,s,\alpha_{1},\alpha_{2})>0$ such that
\begin{equation}\label{Stability ineq: p>2}
    \delta_{s,p,\alpha}(u)
    \geq
    C
    \left(
        \int_{\mathbb{R}^{N}}
        \frac{|u(x)|^{p}}
        {|x|^{sp-\alpha p}}
        \,dx
    \right)
    \mathcal{D}_{s,p,\alpha}(u)^{2p}.
\end{equation}
\end{theorem}

The quantitative stability estimate established above for the weighted fractional Hardy inequality is valid for $p\geq 2$, where the corresponding remainder term has a different structure. We now consider the case $1<p<2$, for which a different form of the remainder term is required. Set
\begin{equation*}
    q_{s, \alpha}:=\frac{2N}{N-sp+\alpha p}.
\end{equation*}
For $1<p<2$, we introduce the following distance to the family of extremal functions:
\begin{equation*} \widetilde{\mathcal{D}}_{s,p,\alpha}(u):=
    \inf_{a\in\mathbb{R}}
    \frac{
        \left\|u^{\left\langle \frac{p}{2} \right\rangle}
        -a\,\omega^{\left\langle \frac{p}{2} \right\rangle}_{s,\alpha}
        \right\|_{L^{q_{s, \alpha},\infty}(\mathbb{R}^{N})}
    }{
        \|u\|^{\frac{p}{2}}_{L^{p^{*}_{s, \alpha},p}(\mathbb{R}^{N})}
    }.
\end{equation*}
In the above distance function, we used the fact that $\omega^{\frac{p}{2}}_{s,\alpha} \in L^{q_{s,\alpha}(\mathbb{R}^{N})}$, which follows from (see Remark \ref{Remark on Lorentz norm})
\begin{equation*}
\| \omega_{s, \alpha} \|_{L^{p^{*}_{s, \alpha}}(\mathbb{R}^{N})}  =  \left\| \left( \omega_{s, \alpha}^{\frac{p}{2}} \right)^{\frac{2}{p}}  \right\|_{L^{p^{*}_{s, \alpha}}(\mathbb{R}^{N})}=  \left\|\omega^{\frac{p}{2}}_{s,\alpha}
        \right\|^{\frac{2}{p}}_{L^{q_{s, \alpha},\infty}(\mathbb{R}^{N})}, \quad q_{s, \alpha} = \frac{2N}{N-sp+\alpha p}.
\end{equation*}
The following theorem establishes the quantitative stability estimate for the weighted fractional Hardy inequality in the range $1<p<2$.

\begin{theorem}\label{Theorem: Quantitative frac hardy 1 < p < 2}
Let $1<p<2$, $s\in(0,1)$, and let $\alpha_{1},\alpha_{2}\in\mathbb{R}$ with
$\alpha=\alpha_{1}+\alpha_{2} \geq 0$ satisfy $\alpha_{1}p, \, \alpha_{2}p\in(-N,sp)$ and $ 0<sp-\alpha p<N$. Then, for every
$u\in C^{1}_{c}(\mathbb{R}^{N})$, there exists a constant $C=C(N,p,s,\alpha_{1},\alpha_{2})>0$ such that
\begin{equation}\label{Stability ineq: p<2}
    \delta_{s,p,\alpha}(u)
    \geq
    C
    \left(
        \int_{\mathbb{R}^{N}}
        \frac{|u(x)|^{p}}
        {|x|^{sp-\alpha p}}
        \,dx
    \right)
    \widetilde{\mathcal{D}}_{s,p,\alpha}(u)^{4}.
\end{equation}
\end{theorem}

\begin{remark}
The results stated in Theorems \ref{Theorem: Quantitative frac hardy p geq 2}
and \ref{Theorem: Quantitative frac hardy 1 < p < 2} yield quantitative
stability estimates for the fractional Hardy inequality. Indeed,
\eqref{Stability ineq: p>2} and \eqref{Stability ineq: p<2} establish that
the deficit is nonnegative and, moreover, provides a quantitative measure
of the distance between $u$ and the family of virtual extremals
$\{a\, \omega_{s, \alpha}\}_{a \in \mathbb{R}\setminus\{0\}}$.
Hence, any near-extremizing sequence associated with
\eqref{Weighted fractional Hardy : point singularity} approaches this
family in the corresponding Lorentz-type sense. We
can normalize $u$ by assuming that
$
\int_{\mathbb{R}^N} \frac{|u(x)|^p}{|x|^{sp-\alpha p}} \, dx = 1$.
With this normalization, the conclusions of the above two theorems can
be expressed in the following form
\begin{equation*}
  \delta_{s,p,\alpha}(u) \geq 
  \begin{cases}
 \mathcal{D}_{s,p,\alpha}(u)^{2p} , \quad p \geq 2, \\[2mm]
 \widetilde{\mathcal{D}}_{s,p,\alpha}(u)^{4}, \quad 1<p<2.
  \end{cases}
\end{equation*}
\end{remark}

The two theorems above reduce to the quantitative stability results for the fractional Hardy inequality, \eqref{quantitative frac Hardy 1} and \eqref{quantitative frac Hardy 2}, by taking $\alpha_{1}=\alpha_{2}=0$.

\smallskip

To the best of the author knowledge, a weighted fractional extension theorem and the corresponding weighted fractional Sobolev inequality on bounded Lipschitz domains, under the weighted setting considered here, are not available in the existing literature. These results constitute a key ingredient in the proof of the main theorems of this paper. In Section \ref{Section : 5}, we establish a weighted fractional extension theorem for
bounded Lipschitz domains $\Omega\subset\mathbb{R}^N$ such that
$0\notin\partial\Omega$. In particular, we show that functions in the weighted fractional Sobolev space can be extended to the whole space with the extension norm controlled by the original norm (see Theorem \ref{Theorem: fractional extension theorem}). In Section \ref{Section : 6}, we combine this extension theorem with the fractional Caffarelli--Kohn--Nirenberg inequality of Nguyen and Squassina \cite{Squassina2018} to derive the corresponding weighted fractional Sobolev inequality on bounded Lipschitz domains $\Omega\subset\mathbb{R}^N$ satisfying
$0\notin\partial\Omega$ (see Theorem \ref{Theorem: Weighted frac Sobolev inequality}). More precisely, setting
\begin{equation*}
p^{*}_{s,\alpha}=\frac{Np}{N-sp+\alpha p},
\qquad \alpha=\alpha_{1}+\alpha_{2} \geq 0,  \quad 0<sp-\alpha<N   
\end{equation*}
and assuming $\alpha_{1}p,\alpha_{2}p\in(-N,sp)$, we establish the following weighted fractional Sobolev inequality:
\begin{equation*}
    \|u\|_{L^{p^{*}_{s, \alpha}}(\Omega)}
        \leq C \|u\|_{W^{s,p, \alpha}(\Omega)}, \quad \forall \, u \in W^{s,p, \alpha}(\Omega) .
\end{equation*}
This inequality provides the scale-invariant Sobolev estimate needed in the subsequent proof of the main quantitative stability results.

\smallskip

The rearrangement-free framework provides a unified approach to the local and fractional Hardy inequalities. For recent developments in fractional Hardy inequalities, we refer to \cite{ADICCM2026, AdiJFA2026, AdiCVPDE2026, Adi2025, Bal2022, Bianchi2024a, Bianchi2024, Bianchi2026, Dyda2011, Bogdan2022, Brasco2018, Chen2003, Cinti2024, Dyda2004, Dyda2023,  Frank2010, Gyula2026, Loss2010, Lizaveta2026}, without any claim of completeness.

\subsection{Organization of the paper}

The structure of the paper is as follows.

\begin{itemize}
 \item[Section \ref{Section 1}:] We briefly review the classical Hardy inequality and its fractional counterpart. We then introduce the corresponding weighted Hardy inequalities in the local and fractional settings and state the main quantitative stability results, together with the associated virtual extremizers and distance functionals.

\item[Section \ref{Section 2 : Lorentz Spaces and Weak-Lp Estimates}:]  Lorentz spaces and weak-$L^p$ estimates. This section
introduces decreasing rearrangements and Lorentz quasi-norms, recalls the
Hardy--Littlewood rearrangement inequality, establishes the weak Lebesgue
integrability of the virtual extremizer, and proves a Lorentz-space estimate
for the weighted Hardy potential.

\item[Section \ref{Section 3}:]  Weighted Sobolev and Hardy--Sobolev inequalities in the
local case. This section establishes the weighted Sobolev inequality,
including the relevant extension result for bounded $C^1$ domains away from
the origin, and develops a weighted Hardy inequality with a remainder term.

\item[Section \ref{Section 4}:]  Proof of quantitative stability for the weighted Hardy
inequality. In this section, we prove Theorem \ref{Theorem: Quantitative Stabiliy Weighted Hardy inequality} using the weighted Sobolev
and Hardy--Sobolev estimates established in the preceding sections.

\item[Section \ref{Section : 5}:]  Extension theory for weighted fractional Sobolev spaces.
This section establishes the zero-extension, reflection, and truncation
properties and proves the weighted fractional extension theorem for bounded
Lipschitz domains whose boundary does not contain the origin.

\item[Section \ref{Section : 6}:]  Weighted fractional Sobolev inequalities on bounded
Lipschitz domains. This section develops the weighted fractional Sobolev
inequality on $\mathbb{R}^N$ and transfers it to bounded Lipschitz domains
using the extension theorem. A weighted Poincaré-type inequality and further
estimates for the weighted Gagliardo seminorm are also established.

\item[Section \ref{Section : 7}:]  Proof of quantitative stability for the weighted
fractional Hardy inequality. The proof is divided according to the two
regimes $p \geq 2$ (Theorem \ref{Theorem: Quantitative frac hardy p geq 2}) and $1 < p < 2$ (Theorem \ref{Theorem: Quantitative frac hardy 1 < p < 2}), using the corresponding remainder
estimates and Lorentz-type distance functionals to obtain quantitative
control of the Hardy deficit.

\item[Section \ref{Appendix}:]  Appendix: local flattening of Lipschitz boundaries away
from the origin. This section provides the boundary-flattening construction
used in the proof of the weighted fractional extension theorem and establishes
the corresponding properties of the bi-Lipschitz transformation.
\end{itemize}

\section{Lorentz Spaces and Weak-\texorpdfstring{$L^p$}{Lp} Estimates}\label{Section 2 : Lorentz Spaces and Weak-Lp Estimates}

We collect here the basic notation and definitions concerning decreasing rearrangements and Lorentz spaces that will be used in the sequel. In particular, the weak-$L^p$ space plays a role in describing the integrability properties of the virtual extremizers and in formulating the distance from this family. We therefore recall the relevant definitions for completeness.

\smallskip

Let $u:\mathbb{R}^{N}\to\mathbb{R}$ be measurable. Its distribution function is defined by
\begin{equation*}
    d_u(\lambda)
    :=
    \bigl|\{x\in\mathbb{R}^{N}: |u(x)|>\lambda\}\bigr|,
    \qquad \lambda>0.
\end{equation*}
The decreasing rearrangement of $u$, denoted by $u^{*}$, is then given by
\begin{equation}\label{eq:decreasing-rearrangement}
    u^{*}(t)
    :=
    \inf\bigl\{\lambda>0:d_u(\lambda)\leq t\bigr\},
    \qquad t>0.
\end{equation}

\begin{definition}\label{defi-lorentz}
For $0<p,q\leq\infty$, the Lorentz quasi-norm of a measurable function
$u$ on $\mathbb{R}^{N}$ is defined by
\begin{equation*}
    \|u\|_{L^{p,q}(\mathbb{R}^{N})}
    :=
    \begin{cases}
        \displaystyle
        \left(
        \int_{0}^{\infty}
        \bigl(t^{\frac{1}{p}}u^{*}(t)\bigr)^{q}
        \,\frac{dt}{t}
        \right)^{\frac{1}{q}},
        & 0<q<\infty,\\[1.2em]
        \displaystyle
        \sup_{t>0}t^{\frac{1}{p}}u^{*}(t),
        & q=\infty.
    \end{cases}
\end{equation*}
The corresponding Lorentz space consists of all measurable functions
$u$ for which this quantity is finite. When $q=\infty$, we write
$L^{p,\infty}(\mathbb{R}^{N})$ and refer to it as the weak-$L^p$ space,
also known as the Marcinkiewicz space.  In particular, when $q=p$, the
Lorentz space coincides with the usual Lebesgue space, that is,
\begin{equation*}
    L^{p,p}(\mathbb{R}^{N})=L^{p}(\mathbb{R}^{N}).
\end{equation*}
\end{definition}

\begin{remark}\label{Remark on Lorentz norm}
   For all $0<p,r< \infty$  and $0< q \leq \infty$, we have (See \cite[Remark $1.4.7$]{GrafakosBook})
   \begin{equation*}
      \||u|^{r}\|_{L^{p,q}(\mathbb{R}^{N})}  = \|u\|^{r}_{L^{pr,qr}(\mathbb{R}^{N})}
   \end{equation*}
\end{remark}

The Lorentz spaces form a natural refinement of the usual $L^p$ spaces, with the second index measuring the degree of integrability more precisely. In particular, when the first exponent $p$ is fixed, increasing the second
exponent enlarges the corresponding Lorentz space. Thus, if $q<r$, one
expects functions belonging to $L^{p,q}$ to also belong to $L^{p,r}$.
The following proposition makes this inclusion precise (see \cite[Proposition $1.4.10$]{GrafakosBook}).

\begin{proposition}\label{Proposition on Lorentz space}
Suppose that $0<p\leq\infty$ and  $0<q<r\leq\infty$. Then there exists a constant $C(p,q,r)>0$, depending only on $p$, $q$, and
$r$, such that
\begin{equation*}
    \|f\|_{L^{p,r}}
\leq
C(p,q,r)\,
\|f\|_{L^{p,q}}.
\end{equation*}
In other words, $L^{p,q}\subseteq L^{p,r}$.
\end{proposition}

\smallskip

\textbf{Hardy--Littlewood rearrangement inequality.} Let $f,g$ be nonnegative measurable functions on $\mathbb{R}^{N}$. Then
\begin{equation}\label{Hardy-Littlehood}
\int_{\mathbb{R}^{N}} f(x)g(x)\,dx
\leq \int_{0}^{\infty} f^{*}(t)\, g^{*}(t)\, dt.
\end{equation}

\smallskip

The following lemma shows that the  virtual extremizers associated with the local and fractional Hardy inequality \eqref{Weighted Hardy inequality} and \eqref{Weighted fractional Hardy : point singularity} is contained in weak-$L^{p}$ (also known as Marcinkiewicz space).

\begin{lemma}\label{Marcinkiewicz}
Let $0<sp-\alpha p<N$ with $0<s \leq 1$ and define
\begin{equation*}
    \omega_{s, \alpha}(x) := |x|^{-\frac{N-sp+ \alpha p}{p}}, \qquad x \in \mathbb{R}^{N} \setminus \{0\}.
\end{equation*}
Then $\omega \in L^{p^{*}_{s, \alpha},\infty}(\mathbb{R}^{N})$, where $p^{*}_{s, \alpha}= \frac{Np}{N-sp+\alpha p}$.
\end{lemma}
\begin{proof}
For $\lambda>0$, the distribution function of $\omega_{s, \alpha}$ is given by
\begin{align*}
\{ x \in \mathbb{R}^{N} : \omega_{s, \alpha}(x) > \lambda \}
=
\left\{ x \in \mathbb{R}^{N} : |x|^{-\frac{N-sp+ \alpha p}{p}} > \lambda \right\} =
\left\{ x \in \mathbb{R}^{N} : |x| < \lambda^{-\frac{p}{N-sp+ \alpha p}} \right\}.
\end{align*}
Since the measure of a ball of radius $r$ in $ \mathbb{R}^{N}$ is $\left( \frac{\mathbb{S}^{N-1}}{N} \right) r^{N}$, we obtain
\begin{equation*}
d_{\omega_{s, \alpha}}(\lambda)
=
\left( \frac{\mathbb{S}^{N-1}}{N} \right) \lambda^{- \frac{Np}{N-sp+ \alpha p}}
=
\left( \frac{\mathbb{S}^{N-1}}{N} \right) \lambda^{-p^{*}_{s, \alpha}}.
\end{equation*}
Therefore, by the definition of the decreasing rearrangement,
\begin{align*}
\omega^{*}_{s, \alpha}(t)
=
\inf \left\{ \lambda>0 : \left(\frac{\mathbb{S}^{N-1}}{N} \right) \lambda^{-p^{*}_{s, \alpha}} \leq t \right\} & =
\inf \left\{ \lambda>0 : \lambda \geq \left( \frac{\mathbb{S}^{N-1}}{N} \right)^{\frac{1}{p^{*}_{s, \alpha}}}  \frac{1}{t^{\frac{1}{p^{*}_{s, \alpha}}}}\right\} \\ & =
\left( \frac{\mathbb{S}^{N-1}}{N} \right)^{\frac{1}{p^{*}_{s, \alpha}}}  \frac{1}{t^{\frac{1}{p^{*}_{s, \alpha}}}}.
\end{align*}
Hence,
\begin{equation*}
\|\omega_{s, \alpha}\|_{L^{p^{*}_{s, \alpha},\infty}}
=
\sup_{t>0}
t^{\frac{1}{p^{*}_{s, \alpha}}}
\left( \frac{\mathbb{S}^{N-1}}{N} \right)^{\frac{1}{p^{*}_{s, \alpha}}}  \frac{1}{t^{\frac{1}{p^{*}_{s, \alpha}}}}
=
\left( \frac{\mathbb{S}^{N-1}}{N} \right)^{\frac{1}{p^{*}_{s, \alpha}}} 
<
\infty.
\end{equation*}
This proves that $\omega_{s, \alpha} \in L^{p^{*}_{s, \alpha},\infty}(\mathbb{R}^N)$.
\end{proof}

The next lemma provides a useful estimate of the local and fractional Hardy
potential in terms of a Lorentz norm. The estimate is an immediate consequence of the Hardy--Littlewood rearrangement inequality.

\begin{lemma}\label{Lemma: Hardy potential and Lorentz}
Let $u$ be a measurable function on $\mathbb{R}^{N}$ and assume that $0<sp-\alpha p<N$ with $0 < s \leq 1$. Then
\begin{equation}
\int_{\mathbb{R}^{N}} \frac{|u(x)|^{p}}{|x|^{sp- \alpha p}} \, dx
\leq \left( \frac{\mathbb{S}^{N-1}}{N} \right)^{\frac{sp- \alpha p}{N}} \, \|u\|_{L^{p^{*}_{s, \alpha},\,p}(\mathbb{R}^{N})}^{p},
\end{equation}
where $p^{*}_{s,\alpha}=\frac{Np}{N-sp+\alpha p}$.

\end{lemma}
\begin{proof}
Set
\begin{equation*}
    g(x):=|x|^{-(sp-\alpha p)}.
\end{equation*}
Since $0<sp-\alpha p<N$, the distribution function of $g$ is
\begin{align*}
    d_g(\lambda)
    &=
    \bigl|\{x\in\mathbb{R}^{N}:g(x)>\lambda\}\bigr| =
    \left|
    \left\{
    x\in\mathbb{R}^{N}:
    |x|<\lambda^{-\frac{1}{sp-\alpha p}}
    \right\}
    \right| =
    \frac{\mathbb{S}^{N-1}}{N}
    \lambda^{-\frac{N}{sp-\alpha p}},
    \qquad \lambda>0.
\end{align*}
It follows from the definition of the decreasing rearrangement that
\begin{equation*}
    g^{*}(t)
    =
    \left(\frac{\mathbb{S}^{N-1}}{N}\right)^{\frac{sp-\alpha p}{N}}
    t^{-\frac{sp-\alpha p}{N}},
    \qquad t>0.
\end{equation*}

We can now apply the Hardy--Littlewood rearrangement inequality
\eqref{Hardy-Littlehood} to obtain
\begin{align*}
    \int_{\mathbb{R}^{N}}
    \frac{|u(x)|^{p}}{|x|^{sp-\alpha p}}\,dx
    &\leq
    \int_{0}^{\infty}
    (u^{*}(t))^{p}g^{*}(t)\,dt \\
    &=
    \left(\frac{\mathbb{S}^{N-1}}{N}\right)^{\frac{sp-\alpha p}{N}}
    \int_{0}^{\infty}
    (u^{*}(t))^{p}
    t^{-\frac{sp-\alpha p}{N}}\,dt.
\end{align*}
On the other hand, from the definition of $p^{*}_{s,\alpha}$,
\begin{align*}
    \int_{\mathbb{R}^{N}}
    \frac{|u(x)|^{p}}{|x|^{sp-\alpha p}}\,dx
    &\leq
    \left(\frac{\mathbb{S}^{N-1}}{N}\right)^{\frac{sp-\alpha p}{N}}
    \int_{0}^{\infty}
    \left(t^{\frac{1}{p^{*}_{s,\alpha}}}u^{*}(t)\right)^{p}
    \frac{dt}{t} \\
    &=
    \left(\frac{\mathbb{S}^{N-1}}{N}\right)^{\frac{sp-\alpha p}{N}}
    \|u\|_{L^{p^{*}_{s,\alpha},p}(\mathbb{R}^{N})}^{p}.
\end{align*}
This completes the proof.
\end{proof}

\section{Weighted Sobolev and  Hardy--Sobolev inequalities: The local case}\label{Section 3}

In this section, we establish the weighted Sobolev inequality that plays a crucial role in the proof of our main result in the local case. We begin by proving an extension theorem for bounded $C^{1}$ domains $\Omega$ satisfying $0 \notin \overline{\Omega}$. Combining this extension result with the Caffarelli--Kohn--Nirenberg inequality established in \cite{Caffarelli1984}, we subsequently derive a weighted Sobolev inequality on bounded $C^{1}$ domains that are separated from the origin. We also establish a weighted Hardy inequality with a remainder term, which will be useful in the subsequent analysis. 

\subsection{Weighted Sobolev inequality} In this subsection, we establish a weighted Sobolev inequality on a bounded $C^{1}$ domain that does not contain the origin. The following theorem provides an extension result for functions in a weighted Sobolev space defined on a bounded $C^{1}$ domain $\Omega$ satisfying $0 \notin \overline{\Omega}$.

\begin{theorem}[Weighted Sobolev extension theorem]
\label{Theorem: Weighted Local Extension theorem}
Let $p\geq 1$ and $\alpha \geq 0$. Let $\Omega$ be a bounded $C^{1}$ domain such that $ S_{1} := \inf_{x \in \Omega} |x| >0$ and $S_{2}:= \sup_{x \in \Omega } |x| < \infty$. Then, for every $u\in W^{1,p,\alpha}(\Omega)$, there exists $\widetilde{u}\in W^{1,p,\alpha}(\mathbb{R}^{N})$ such that $\widetilde{u}=u$ a.e. in $\Omega$ and
\begin{equation*} \|\widetilde{u}\|_{W^{1,p,\alpha}(\mathbb{R}^{N})}
\leq C\|u\|_{W^{1,p,\alpha}(\Omega)},
\end{equation*}
where $C>0$ depends only on $N$, $p$, $\alpha$, and $\Omega$.
\end{theorem}

\begin{proof}
Since $\Omega$ is a bounded $C^{1}$ domain  and $ \inf_{x \in \Omega} |x| >0$ and $\sup_{x \in \Omega } |x| < \infty$, the weight $|x|^{\alpha p}$ is bounded above and below by positive
constants on $\Omega$. More precisely, we have 
\begin{equation*}
    0<S^{\alpha p}_{1} \leq |x|^{\alpha p} \leq S^{\alpha p}_{2}, \quad \text{for all } x\in\Omega.
\end{equation*}
Consequently, the weighted and unweighted Sobolev norms are
equivalent on $\Omega$. In particular,
\begin{equation*}
    W^{1,p,\alpha}(\Omega)=W^{1,p}(\Omega)
\end{equation*}
with equivalent norms, and
\begin{equation*}
    \|u\|_{W^{1,p}(\Omega)}
\leq S^{-\alpha p}_{1}\|u\|_{W^{1,p,\alpha}(\Omega)}.
\end{equation*}

By the standard Sobolev extension theorem for bounded $C^{1}$
domains (see \cite[Theorem 1, Section 5]{EvansBook}), there exists an
extension
\begin{equation*}
    \widetilde{u}\in W^{1,p}(\mathbb{R}^{N})
\end{equation*}
such that $\widetilde{u}=u$ a.e. in $\Omega$ and
\begin{equation*} \|\widetilde{u}\|_{W^{1,p}(\mathbb{R}^{N})}
\leq C\|u\|_{W^{1,p}(\Omega)}.
\end{equation*}
Moreover, since $\overline{\Omega}$ is compact and
$0\notin\overline{\Omega}$, the extension can be chosen to have
compact support in a bounded open set $V$ satisfying
\begin{equation*}
\overline{\Omega}\subset V \subset \subset \mathbb{R}^{N}\setminus\{0\}.
\end{equation*}
In particular,
\begin{equation*}
\operatorname{supp}\widetilde{u}\subset V.
\end{equation*}
For example, one may establish the extension in such a way that $V= \Omega \cup \left( \cup_{x \in \partial \Omega} B_{S_{1}/2}(x) \right)$.  Since $\overline{V}$ is compact and $0\notin\overline{V}$, there
exist constants $0<c_1<c_2<\infty$ such that
\begin{equation*}
    c_1\leq |x|^{\alpha p}\leq c_2
\qquad\text{for all }x\in V.
\end{equation*}
Hence $|x|^{\alpha p}$ is bounded above and below by positive constants
on $V$. Therefore,
\begin{equation*}
\|\widetilde{u}\|_{W^{1,p,\alpha}(\mathbb{R}^{N})}
=\|\widetilde{u}\|_{W^{1,p,\alpha}(V)} \leq C\|\widetilde{u}\|_{W^{1,p}(V)} = C\|\widetilde{u}\|_{W^{1,p}(\mathbb{R}^{N})}.
\end{equation*}
Combining the above estimates with the equivalence of the weighted
and unweighted norms on $\Omega$, we obtain
\begin{equation*}
\|\widetilde{u}\|_{W^{1,p,\alpha}(\mathbb{R}^{N})}
\leq C\|u\|_{W^{1,p,\alpha}(\Omega)}.
\end{equation*}
Thus, $\widetilde{u}$ is the desired extension, and the proof is
complete.
\end{proof}

Caffarelli, Kohn, and Nirenberg in \cite{Caffarelli1984} studied first-order interpolation inequalities in Sobolev spaces with weights, now known as the Caffarelli--Kohn--Nirenberg inequality, and proved the following theorem. Let $N \geq 1$, $p\geq 1$, $q \geq 1$, $\tau >0$, $a \in [0,1]$, and $\alpha, \beta, \gamma \in \mathbb{R}$ satisfy
\begin{equation*}
    \frac{1}{\tau}+ \frac{\gamma}{N} = a \left( \frac{1}{p} + \frac{\alpha-1}{N} \right) + (1-a) \left( \frac{1}{q} + \frac{\beta}{N} \right)  \ .
\end{equation*}
If $a>0$, assume also that, with $\gamma = a \sigma + (1-a) \beta$,
\begin{equation*}
    0 \leq \alpha - \sigma
\end{equation*}
and 
\begin{equation*}
    \alpha - \sigma \leq 1 \hspace{.5cm} \text{if} \hspace{.5cm} \frac{1}{\tau}+ \frac{\gamma}{N} = \frac{1}{p}+ \frac{\alpha-1}{N}  .
\end{equation*}
If $\frac{1}{\tau} + \frac{\gamma}{N} >0$, then we have
\begin{equation}\label{CKN ineq}
  \| |x|^{\gamma}u \|_{L^{\tau}(\mathbb{R}^N)} \leq C \||x|^{\alpha} |\nabla u|\|^{a}_{L^{p}(\mathbb{R}^N)} \||x|^{\beta} u \|^{(1-a)}_{L^{q} (\mathbb{R}^N)},  \hspace{.3cm}  \forall \ u  \in C^{\infty}_{c}(\mathbb{R}^N)  .
      \end{equation}
\smallskip

Now, if we take $\gamma=0$ and $a=1$, the Caffarelli--Kohn--Nirenberg inequality \eqref{CKN ineq} reduces to  
\begin{equation}  \label{sobolev inequality}
\left( \int_{\mathbb{R}^{N}} |u(x)|^{\tau} \, dx \right)^{\!\frac{1}{\tau}}
 \leq C \left( \int_{\mathbb{R}^{N}} |\nabla u(x)|^{p} |x|^{\alpha p} \, dx \right)^{\!\frac{1}{p}},
\end{equation}
under the conditions
\begin{equation*}
 0 \leq \alpha < 1, \qquad p- \alpha p < N, \qquad \text{and} \qquad 
   \tau = \frac{Np}{N-p+ \alpha p}.  
\end{equation*}
The inequality \eqref{sobolev inequality} is nothing but the Sobolev inequality in weighted spaces, with the critical exponent
\begin{equation*}
    p^{*}_{\alpha} = \tau = \frac{Np}{N-p+ \alpha p}, \qquad \text{if} \quad 0< p- \alpha p < N.
\end{equation*}
We can therefore state the following fundamental result.

\begin{lemma}[Weighted Sobolev inequality]\label{Lemma: Weighted Sobolev inequality}
    Let $p\geq 1$,  $\alpha \in [0, 1) $ be such that $0<p-\alpha p < N$
    and $p^{*}_{\alpha}= \frac{Np}{N-p+\alpha p}$. Then for all $u \in W^{1,p, \alpha}(\mathbb{R}^{N})$, there exists a constant $C=C(N,p,\alpha)>0$ such that
    \begin{align}
        \left( \int_{\mathbb{R}^{N}} |u(x)|^{p^{*}_{\alpha}} \, dx \right)^{\!\frac{1}{p^{*}_{\alpha}}}
        \leq C \left( \int_{\mathbb{R}^{N}} |\nabla u(x)|^{p} |x|^{\alpha p}  \, dx \right)^{\!\frac{1}{p}}   .
    \end{align}
\end{lemma}

The next theorem establishes the weighted Sobolev inequality on bounded $C^{1}$ domains with $0 \notin \overline{\Omega}$. It follows by combining the weighted Sobolev extension theorem with the weighted Sobolev inequality on $\mathbb{R}^{N}$.

\begin{theorem}\label{Theorem: Weighted Sobolev inequality on bounded domain}
    Let $p\geq 1$,  $\alpha \in [0, 1) $ be such that $0<p-\alpha p < N$
    and $p^{*}_{\alpha}= \frac{Np}{N-p+\alpha p}$. Let $\Omega$ be a bounded  $C^{1}$ domain and assume that $0 \notin \overline{\Omega}$. Then there exists a constant $C=C(N,p,\alpha, \Omega)>0$ such that
    \begin{equation}
  \|u\|_{L^{p^{*}_{\alpha}}(\Omega)}
        \leq C \|u\|_{W^{1,p, \alpha}(\Omega)}, \quad \forall \, u \in W^{1,p, \alpha}(\Omega) .
    \end{equation}
\end{theorem}
\begin{proof}
Let $\Omega\subset\mathbb{R}^N$ be a bounded $C^{1}$ domain with $0 \notin \overline{\Omega}$, and let $u\in W^{1,p,\alpha}(\Omega)$. By Theorem \ref{Theorem: Weighted Local Extension theorem}, there exists an extension $\widetilde{u}\in W^{1,p,\alpha}(\mathbb{R}^N)$ such that 
\begin{equation*}
        \widetilde{u}=u \qquad\text{a.e. in }\Omega,
    \end{equation*}
and
\begin{equation*}
        \| \widetilde{u} \|_{W^{1,p, \alpha}(\mathbb{R}^{N})} \leq C  \| u \|_{W^{1,p, \alpha}(\Omega)}, 
\end{equation*}
where $C=C(N,p,\alpha, \Omega)>0$.  On the other hand, by Lemma \ref{Lemma: Weighted Sobolev inequality}, there exists a constant  $C=C(N,p, \alpha)>0$ such that
    \begin{equation*}
\|\widetilde{u}\|_{L^{p^{*}_{\alpha}}(\mathbb{R}^{N})}  \leq  C  \| \nabla \widetilde{u} \|_{L^{p,\alpha}(\mathbb{R}^{N})},
    \end{equation*}
    where $p^{*}_{\alpha} = \frac{Np}{N-p+\alpha p}$. Since  $\| \nabla \widetilde{u}\|_{L^{p,\alpha}(\mathbb{R}^N)} \leq \| \widetilde{u}\|_{W^{1,p,\alpha}(\mathbb{R}^N)}$, we have 
    \begin{equation*} \|\widetilde{u}\|_{L^{p^{*}_{\alpha}}(\mathbb{R}^N)} \leq C \|\widetilde{u}\|_{W^{1,p,\alpha}(\mathbb{R}^N)}. \end{equation*}
Combining this with the extension estimate yields 
\begin{equation*} \|\widetilde{u}\|_{L^{p^{*}_{\alpha}}(\mathbb{R}^N)} \leq C \|u\|_{W^{1,p,\alpha}(\Omega)}. \end{equation*} 
Finally, using $\|u\|_{L^{p^{*}_{\alpha}}(\Omega)} \leq \|\widetilde{u}\|_{L^{p^{*}_{\alpha}}(\mathbb{R}^N)}$, we have
    \begin{equation*}
\|u\|_{L^{p^{*}_{\alpha}}(\Omega)}  \leq  C \| u \|_{W^{1,p, \alpha}(\Omega)}.
    \end{equation*}
 This completes the proof.
\end{proof}

The next lemma proves the weighted Sobolev Poincar\'e inequality, which bounds the weighted $L^{p, \alpha}$-norm of $u-(u)_{\Omega}$ by the weighted gradient seminorm of $u$. Here, $(u)_{\Omega}$ denotes the average of $u$ over $\Omega$, i.e.,
\begin{equation}\label{Defn: average of u}
    (u)_{\Omega} := \frac{1}{|\Omega|} \int_{\Omega} u(x) \, dx = \fint_{\Omega} u(x) \, dx,
\end{equation}
where $|\Omega|$ is the Lebesgue measure of $\Omega$.

\begin{lemma}[Weighted Sobolev Poincar\'e inequality]\label{Lemma: Weighted Sobolev Poincare}
Let $p\geq 1$, $ \alpha \geq 0$, and  let $\Omega$ be a bounded $C^{1}$ domain with $S_{1} = \inf_{x \in \Omega} |x|>0$ and $S_{2} = \sup_{x \in \Omega} |x|<\infty$. Then for all $u \in W^{1,p, \alpha}(\Omega)$, there exists a constant $C=C(N,p,  \Omega)>0$ such that
 \begin{equation}
       \int_{\Omega}  |u(x) - (u)_{\Omega}|^{p} |x|^{\alpha p}  \, dx
        \leq C \left(  \frac{S^{\alpha p}_{2}}{S^{\alpha p}_{1}} \right)  \int_{\Omega} |\nabla u(x)|^{p} |x|^{\alpha p}
 \,dx.
    \end{equation}
\end{lemma}

\begin{proof}
    Since $\Omega$ is a bounded $C^{1}$ domain  and $ \inf_{x \in \Omega} |x| >0$ and $\sup_{x \in \Omega } |x| < \infty$. More precisely, 
    \begin{equation}\label{ineq10}
       0< S^{\alpha p}_{1}<|x|^{\alpha p }< S^{\alpha p}_{2}, \quad \text{for all } x\in\Omega .
    \end{equation}
    Consequently, the weighted and unweighted Sobolev norms are
equivalent on $\Omega$. In particular,
\begin{equation*}
    W^{1,p,\alpha}(\Omega)=W^{1,p}(\Omega).
\end{equation*}
By \cite[Theorem $1$, Section $5.8.1$]{EvansBook}, we have
\begin{equation*}
     \int_{\Omega}  |u(x) - (u)_{\Omega}|^{p}  \, dx  \leq C\int_{\Omega} |\nabla u(x)|^{p} \,dx, \quad \text{for all } u \in W^{1,p}(\Omega).
\end{equation*}
Using the estimate \eqref{ineq10} and the above Poincar\'e inequality, we obtain
\begin{align*}
    \int_{\Omega}  |u(x) - (u)_{\Omega}|^{p} |x|^{\alpha p}  \, dx & \leq S^{\alpha p}_{2} \int_{\Omega}  |u(x) - (u)_{\Omega}|^{p}  \, dx \\ & \leq  C  S^{\alpha p}_{2} \int_{\Omega} |\nabla u(x)|^{p} \,dx \leq C \left(  \frac{S^{\alpha p}_{2}}{S^{\alpha p}_{1}} \right) \int_{\Omega} |\nabla u(x)|^{p} |x|^{\alpha p} \,dx.
\end{align*}
This completes the proof of the lemma.
\end{proof}

The next lemma establishes a scale-invariant weighted Poincar\'e--Sobolev inequality on annuli. In particular, it gives a uniform estimate on $\Omega_\lambda$ for every $\lambda>0$ by exploiting the scaling properties of the weighted gradient seminorm.

\begin{lemma}\label{Lemma: Weighted Sobolev ineq Lambda}
Let $p\geq1$, $\alpha \in [0,1)$ be such that $0<p-\alpha p<N$. Let $\lambda>0$ and set $\Omega_{\lambda}:=\{x\in\mathbb{R}^N:\lambda<|x|<n \lambda\}$, where $n \in \mathbb{N}$ and $n>1$, 
and define $p^*_{\alpha}:=\frac{Np}{N-p+\alpha p}$.
Then there exists a constant $C=C(N,p,s,\alpha, n)>0$ such that, for every $u\in W^{1,p,\alpha}(\Omega_\lambda)$,
\begin{align*}
 \left( \frac{1}{|\Omega_{\lambda}|} \int_{\Omega_{\lambda}}  |u(x)-(u)_{\Omega_\lambda}|^{p^*_{\alpha}} \, dx \right)^{\frac{1}{p^{*}_{\alpha}}} 
\leq
C\left(
\lambda^{p-\alpha p-N}
\int_{\Omega_\lambda} |\nabla u(x)|^{p} |x|^{\alpha p} \, dx
\right)^{\frac{1}{p}}.  
\end{align*}
\end{lemma}

\begin{proof}
  First assume $\lambda=1$. By Theorem \ref{Theorem: Weighted Sobolev inequality on bounded domain}, we have
  \begin{equation*}
       \|u\|_{L^{p^{*}_{\alpha}}(\Omega_{1})}
        \leq C \|u\|_{W^{1,p, \alpha}(\Omega_{1})},
  \end{equation*}
  where $C=C(N,p, \alpha, n)>0$. Applying the above inequality with $u-(u)_{\Omega_{1}}$ and using Lemma \ref{Lemma: Weighted Sobolev Poincare} with $\Omega=\Omega_{1}$, $S_{1}=1$ and $S_{2}=n$, we obtain
  \begin{equation*}
    \left(  \int_{\Omega_{1}} |u(x)-(u)_{\Omega_{1}}|^{p^{*}_{\alpha}} \, dx \right)^{\frac{1}{p^{*}_{\alpha}}} \leq C n^{\alpha p} \left( \int_{\Omega_{1}} |\nabla u(x)|^{p} |x|^{\alpha p} \, dx \right)^{\frac{1}{p}}. 
  \end{equation*}
Dividing both sides by $|\Omega_1|^{1/p^*_{\alpha}}$ and absorbing
the fixed factor into the constant, we obtain
 \begin{equation*}
    \left( \frac{1}{|\Omega_{1}|}  \int_{\Omega_{1}} |u(x)-(u)_{\Omega_{1}}|^{p^{*}_{\alpha}} \, dx \right)^{\frac{1}{p^{*}_{\alpha}}} \leq C \left( \int_{\Omega_{1}} |\nabla  u( x)|^{p} |x|^{\alpha p} \, dx \right)^{\frac{1}{p}}. 
  \end{equation*}
Now, let $\lambda>0$ and apply the above inequality to
$v(x)=u(\lambda x)$. Since
\begin{equation*}
    \frac{1}{|\Omega_1|}
\int_{\Omega_1}u(\lambda y)\,dy
=
\frac{1}{|\Omega_\lambda|}
\int_{\Omega_\lambda}u(y)\,dy
=
(u)_{\Omega_\lambda},
\end{equation*}
we obtain
\begin{equation*}
    \left( \frac{1}{|\Omega_{1}|}  \int_{\Omega_{1}} |u(\lambda x)-(u)_{\Omega_{\lambda}}|^{p^{*}_{\alpha}} \, dx \right)^{\frac{1}{p^{*}_{\alpha}}} \leq C \left(  \lambda^{p} \int_{\Omega_{1}} |\nabla u(\lambda x)|^{p} |x|^{\alpha p} \, dx \right)^{\frac{1}{p}}. 
  \end{equation*}
Using the changes of variable $X=\lambda x$ on both sides, we have
  \begin{align*}
 \left( \frac{1}{|\Omega_{\lambda}|} \int_{\Omega_{\lambda}}  |u(x)-(u)_{\Omega_\lambda}|^{p^*_{\alpha}} \, dx \right)^{\frac{1}{p^{*}_{\alpha}}}
\leq
C\left(
\lambda^{p-\alpha p-N}
\int_{\Omega_\lambda} |\nabla u( x)|^{p} |x|^{\alpha p} \, dx
\right)^{\frac{1}{p}}.  
\end{align*}
This completes
the proof.
\end{proof}

The next lemma provides an estimate for the difference between the averages of a function over two disjoint sets in terms of its mean oscillation over their union. For the proof, we refer to \cite[Lemma $3.3$]{BanerjeeGangulySahu}

\begin{lemma}\label{Lemma: on two disjoint set}
    Let $E$ and $F$ be two disjoint sets in $\mathbb{R}^{N}$ and $q \geq 1$. Then 
    \begin{equation}
        |(u)_{E} - (u)_{F}|^{q} \leq 2^{q} \frac{|E \cup F|}{\min \{ |E|, |F| \} }  \fint_{E \cup F} |u(x)-(u)_{E \cup F}|^{q} \, dx  . 
    \end{equation}
\end{lemma}

\subsection{Weighted Hardy inequalities with a remainder term} 
We now establish a weighted Hardy inequality with a remainder term. This remainder term will be useful in the proof of quantitative stability, as it allows us to reduce the problem to estimating the remainder term.

\smallskip

We first recall the following convexity estimate, valid for all $X,Y\in\mathbb{R}^{N}$ and $p\geq2$ (see \cite[Inequality $(2.13)$]{frank2008}):
\begin{equation}\label{Convex Estimate}
|X+Y|^{p}
\geq
|X|^{p}
+
p|X|^{p-2}X\cdot Y
+
c_{p}|Y|^{p},
\end{equation}
where
\begin{equation*}
    c_{p}:= \min_{0 < \tau<1/2} \left( (1-\tau)^{p}  - \tau^{p}  + p \tau^{p-1} \right) .
\end{equation*}
Applying this convexity estimate yields the following remainder estimate for $p\geq2$.

\begin{lemma}\label{Lemma: weighted Hardy remainder p geq 2}
Let $p \geq 2$ and $\alpha \in [0,1)]$ be such that $0<p-\alpha p<N$. Then for any $u \in W^{1,p, \alpha}(\mathbb{R}^{N})$,
\begin{align} \int_{\mathbb{R}^{N}}  |\nabla u(x)|^{p} |x|^{\alpha p} \, dx -  \left( \frac{N-p+\alpha p}{p} \right)^{p} \int_{\mathbb{R}^N} \frac{|u(x)|^{p}}{|x|^{p- \alpha p}} \, dx  \geq c_{p} \int_{\mathbb{R}^{N}} \frac{|\nabla v(x)|^{p}}{|x|^{N-p}}  \, dx,
\end{align}
where $v(x) = u(x) v^{-1}_{\alpha}(x)$, and $v_{\alpha}(x) = |x|^{-\frac{N-p+\alpha p}{p}}$. \end{lemma}
\begin{proof}
By density, it is sufficient to prove the result for $u\in C_{c}^{\infty}(\mathbb{R}^{N})$. Let $u=v v_{\alpha}$, where $v_{\alpha}$ is defined in
\eqref{omega alpha}. Then
    \begin{equation*}
        \nabla u = v \nabla v_{\alpha} + v_{\alpha} \nabla v. 
    \end{equation*}
Applying the convexity estimate \eqref{Convex Estimate} with $X=v\nabla v_{\alpha}$ and $Y=v_{\alpha}\nabla v$, we obtain
    \begin{align*}
   \int_{\mathbb{R}^{N}} & |\nabla u(x)|^{p} |x|^{\alpha p} \, dx   =  \int_{\mathbb{R}^{N}} \left|X+Y \right|^{p} |x|^{\alpha p} \, dx \\ &\geq \left( \frac{N-p+\alpha p}{p} \right)^{p} \int_{\mathbb{R}^N} \frac{|u(x)|^{p}}{|x|^{p- \alpha p}} \, dx + p \int_{\mathbb{R}^{N}} |\nabla v_{\alpha}|^{p-2} |x|^{\alpha p} v_{\alpha} v |v|^{p-2} \nabla v \cdot \nabla v_{\alpha} \, dx \\ & \quad \quad +  c_{p} \int_{\mathbb{R}^{N}} \frac{|\nabla v(x)|^{p}}{|x|^{N-p}}  \, dx .
\end{align*}
Observe that
\begin{equation*}
p |\nabla v_{\alpha}|^{p-2}|x|^{\alpha p} v_{\alpha}v|v|^{p-2}
\nabla v\cdot\nabla v_{\alpha}
=
v_{\alpha}|x|^{\alpha p} |\nabla v_{\alpha}|^{p-2}
\nabla v_{\alpha}\cdot\nabla(|v|^{p}).
\end{equation*}
Hence, integrating by parts, we obtain 
\begin{align*}
   p \int_{\mathbb{R}^{N}} |\nabla v_{\alpha}|^{p-2} |x|^{\alpha p} v_{\alpha} v |v|^{p-2} \nabla v \cdot \nabla v_{\alpha} \, dx & =  \int_{\mathbb{R}^{N}} v_{\alpha}|x|^{\alpha p} |\nabla v_{\alpha}|^{p-2}
\nabla v_{\alpha}\cdot\nabla(|v|^{p}) \, dx \\ & = - \int_{\mathbb{R}^{N}} |\nabla v_{\alpha}|^{p} |v|^{p} |x|^{\alpha p} \, dx \\ &  \quad  - \int_{\mathbb{R}^{N}} v_{\alpha} \operatorname{div} \left(|x|^{\alpha p} |\nabla v_{\alpha}|^{p-2} \nabla v_{\alpha} \right) |v|^{p} \, dx  . 
\end{align*}
Using the Euler--Lagrange equation \eqref{Euler-Lagrange weighted Hardy}, we have $-v_{\alpha} \operatorname{div} \left(|x|^{\alpha p} |\nabla v_{\alpha}|^{p-2} \nabla v_{\alpha} \right) = |\nabla v_{\alpha}|^{p} |x|^{\alpha p}$. Therefore, the above term vanishes. Hence,
\begin{align*}
\int_{\mathbb{R}^{N}}  |\nabla u(x)|^{p} |x|^{\alpha p} \, dx -  \left( \frac{N-p+\alpha p}{p} \right)^{p} \int_{\mathbb{R}^N} \frac{|u(x)|^{p}}{|x|^{p- \alpha p}} \, dx  \geq c_{p} \int_{\mathbb{R}^{N}} \frac{|\nabla v(x)|^{p}}{|x|^{N-p}}  \, dx.
\end{align*}
This completes the proof of lemma.
\end{proof}

The following lemma provides a quantitative convexity estimate for the
function $X \mapsto |X|^{p}$ in the singular range $1<p<2$. It can be
viewed as a second-order Taylor-type lower bound with a remainder term that remains effective even when the Hessian of $|X|^{p}$ is not uniformly positive. The result is also available in \cite{BanerjeeGangulySahu}. However, for the sake of completeness, we include here a different proof.

\begin{lemma}\label{Lemma: Estimate 1<p<2}
Let $1<p<2$, and let $X,Y \in \mathbb{R}^{N}\setminus\{0\}$. Then
\begin{equation}\label{eq:convexity-estimate}
|X+Y|^{p}
\geq
|X|^{p}
+
p|X|^{p-2}X\cdot Y
+
\frac{p(p-1)}{2}
(|X|+|Y|)^{p-2}|Y|^{2}.
\end{equation}
\end{lemma}

\begin{proof}
We first assume that $X+tY\neq 0$ for all $t\in[0,1]$. Define
\begin{equation*}
   \phi(t):=|X+tY|^{p},
\quad \forall \, t\in[0,1]. 
\end{equation*}
Since $X+tY\neq0$ on $[0,1]$, the function $\phi$ belongs to
$C^{2}([0,1])$. Differentiating, we obtain
\begin{equation*}
    \phi'(t)
=
p|X+tY|^{p-2}(X+tY)\cdot Y,
\end{equation*}
and
\begin{equation*}
  \phi''(t)
=
p|X+tY|^{p-2}|Y|^{2}
+
p(p-2)|X+tY|^{p-4}
\big((X+tY)\cdot Y\big)^{2}.  
\end{equation*}
By the Cauchy--Schwarz inequality, $\big((X+tY)\cdot Y\big)^{2} \leq |X+tY|^{2}|Y|^{2}$. Since $p-2<0$, it follows that
\begin{align*}
\phi''(t)
&\geq
p|X+tY|^{p-2}|Y|^{2}
+
p(p-2)|X+tY|^{p-2}|Y|^{2}
\\
&=
p(p-1)|X+tY|^{p-2}|Y|^{2}.
\end{align*}
Moreover, using the triangle inequality and the fact that $p-2<0$, $|X+tY| \leq |X|+|Y|$, and hence
\begin{equation*}
    |X+tY|^{p-2}
\geq
(|X|+|Y|)^{p-2}.
\end{equation*}
Therefore,
\begin{equation*}
   \phi''(t)
\geq
p(p-1)(|X|+|Y|)^{p-2}|Y|^{2}. 
\end{equation*}
Applying Taylor's formula with integral remainder yields
\begin{equation*}
  \phi(1)
=
\phi(0)+\phi'(0)
+
\int_{0}^{1}(1-t)\phi''(t)\,dt.  
\end{equation*}
Since $\phi(0)=|X|^{p}$ and $\phi'(0)=p|X|^{p-2}X\cdot Y$, we obtain
\begin{equation*}
    |X+Y|^{p}
=
|X|^{p}
+
p|X|^{p-2}X\cdot Y
+
\int_{0}^{1}(1-t)\phi''(t)\,dt.
\end{equation*}
Using the lower bound for $\phi''$, we find
\begin{align*}
\int_{0}^{1}(1-t)\phi''(t)\,dt
&\geq
p(p-1)(|X|+|Y|)^{p-2}|Y|^{2}
\int_{0}^{1}(1-t)\,dt
\\
&=
\frac{p(p-1)}{2}
(|X|+|Y|)^{p-2}|Y|^{2}.
\end{align*}
Consequently,
\begin{equation*}
  |X+Y|^{p}
\geq
|X|^{p}
+
p|X|^{p-2}X\cdot Y
+
\frac{p(p-1)}{2}
(|X|+|Y|)^{p-2}|Y|^{2}.  
\end{equation*}
This proves the claim whenever $X+tY\neq0$ for all $t\in[0,1]$.

It remains to consider the case in which $X+tY=0$ for some $t\in[0,1]$. In this situation $X$ and $Y$ are collinear and
there exists $t\in[0,1]$ such that $X=-tY$. Hence $|X|=t|Y|$, $|X+Y|=(1-t)|Y|$,
and $X\cdot Y=-t|Y|^{2}$. Define
\[
F(t)
:=
(1-t)^{p}
+t^{p-1}(p-t)
-\frac{p(p-1)}{2}(1+t)^{p-2},
\qquad t\in[0,1].
\]
Since $1-t \in [0,1]$ and $1<p<2$, we have $(1-t)^{p} \geq (1-t)^{2}$. Also using $0<p-1<1$, we obtain $t^{p-1}(p-t) > t (p-t)$. Therefore, using these estimates and $(1+t)^{p-2} \leq 1$, we arrive at
    \begin{align*}
        F(t) \geq (1-t)^{2} + t(p-t) - \frac{p(p-1)}{2} = 1+ (p-2) t - \frac{p(p-1)}{2}.
    \end{align*}
    Since $p-2<0$, we have $(p-2)t \geq (p-2)$, for all $t \in [0,1]$. Therefore,
    \begin{equation*}
       F(t) \geq 1 + p-2 - \frac{p(p-1)}{2} = (p-1) \left( 1- \frac{p}{2} \right) >0. 
    \end{equation*}
    Therefore, using $F(t) >0$ and the definition of $F$, we obtain
    \begin{align*}
        (1-t)^{p} \geq t^{p}- p t^{p-1}- \frac{p(p-1)}{2}(1+t)^{p-2}, \quad \forall \, t \in [0,1]. 
    \end{align*}
    Multiplying the above inequality with $|Y|^{p}$ and using the relations above, we conclude that
    \begin{align*}
        |X+Y|^{p}  \geq  |X|^{p} + p|X|^{p-2} X\cdot Y 
+ \frac{p(p-1)}{2} (|X|+|Y|)^{p-2} |Y|^{2}.
    \end{align*}
This completes the proof.
\end{proof}

The following lemma establishes a weighted Hardy inequality with a remainder term in the case $1<p<2$. By combining the quantitative convexity estimate obtained above with suitable assumptions, we derive the following result.

\begin{lemma}\label{Lemma: Local Hardy Remainder 1<p<2}
   Let $N \geq 2$, $1 < p < 2$ and $\alpha \in [0,1)$ be such that $0<p-\alpha p<N$. Suppose that $u \in C_c^1(\mathbb{R}^N)$ satisfies $\delta_{p, \alpha}(u) \leq 1$ and $\|u\|_{L^{p^{*}_{\alpha},p}(\mathbb{R}^N)} = 1$. Then there exists a constant $C = C(N,p, \alpha) > 0$ such that
    \begin{equation}
       \int_{\mathbb{R}^N}
|\nabla u(x)|^p |x|^{\alpha p}\,dx - \left( \frac{N-p+ \alpha p}{p} \right)^{p} \int_{\mathbb{R}^N} \frac{|u(x)|^{p}}{|x|^{p- \alpha p}} \, dx  \geq C  \left(\int_{\mathbb{R}^{N}} \frac{|\nabla v(x)|^{p}}{|x|^{N-p}} \, dx \right)^{\frac{2}{p}},
    \end{equation}
    where $v(x) = u(x) v^{-1}_{\alpha}(x)$, and $v_{\alpha}(x) = |x|^{-\frac{N-p+\alpha p}{p}}$.
\end{lemma}
\begin{proof}
   Let $u \in C^{1}_{c}(\mathbb{R}^{N})$ be a nonzero function.  Given that $\delta_{p}(u) \leq 1$ and $\|u\|_{L^{p^{*}_{\alpha},p}(\mathbb{R}^N)} = 1$, applying Lemma \ref{Lemma: Hardy potential and Lorentz} with $s = 1$, we obtain
\begin{align}\label{ineq7654}
   \int_{\mathbb{R}^N}
|\nabla u(x)|^p |x|^{\alpha p}\,dx  & \leq 1 +  \left( \frac{N-p+ \alpha p}{p} \right)^{p} \int_{\mathbb{R}^N} \frac{|u(x)|^{p}}{|x|^{p- \alpha p}} \, dx \nonumber \\ & \leq 1 + \left( \frac{N-p+ \alpha p}{p} \right)^{p} \left( \frac{|\mathbb{S}^{N-1}|}{N} \right)^{\frac{p- \alpha p}{N}}.
\end{align}
Since $u = v v_{\alpha}$, we have
\begin{equation*}
    \nabla u = v \nabla v_{\alpha} + v_{\alpha} \nabla v.
\end{equation*}
Applying Lemma \ref{Lemma: Estimate 1<p<2} with $X = v \nabla v_{\alpha}$, and $Y = v_{\alpha} \nabla v,$ we obtain
\begin{align*}
  &  \int_{\mathbb{R}^{N}} |\nabla u(x)|^{p}|x|^{\alpha p} \, dx   =  \int_{\mathbb{R}^{N}} |X+Y|^{p} |x|^{\alpha p} \, dx  \\  &  \quad \geq \left( \frac{N-p+\alpha p}{p} \right)^{p} \int_{\mathbb{R}^N} \frac{|u(x)|^{p}}{|x|^{p- \alpha p}} \, dx + p \int_{\mathbb{R}^{N}} |\nabla v_{\alpha}|^{p-2} |x|^{\alpha p} v_{\alpha} v |v|^{p-2} \nabla v \cdot \nabla v_{\alpha} \, dx \\ & \quad + \frac{p(p-1)}{2} \int_{\mathbb{R}^{N}} \frac{|\nabla v(x)|^{2}}{|x|^{\left(\frac{N-p+ \alpha p}{p}\right)2}} |x|^{\alpha p} \left( |v(x) \nabla v_{\alpha}(x)| + |v_{\alpha}(x) \nabla v(x)| \right)^{p-2} \, dx .
\end{align*}
Following the proof of Lemma \ref{Lemma: weighted Hardy remainder p geq 2}, the second term on the right-hand side term vanishes. Therefore, we obtain
\begin{align*}
    \int_{\mathbb{R}^{N}} &\frac{|\nabla v(x)|^{2}}{|x|^{\left(\frac{N-p+ \alpha p}{p} \right) 2}}  |x|^{\alpha p} \left( |v(x) \nabla v_{\alpha}(x)| + |v_{\alpha}(x) \nabla v(x)| \right)^{p-2} \, dx \\ & \leq \frac{2}{p(p-1)} \left(   \int_{\mathbb{R}^N}
|\nabla u(x)|^p |x|^{\alpha p}\,dx - \left( \frac{N-p+ \alpha p}{p} \right)^{p} \int_{\mathbb{R}^N} \frac{|u(x)|^{p}}{|x|^{p- \alpha p}} \, dx \right) \\ &  =\left( \frac{2}{p(p-1)} \right) \delta_{p, \alpha}(u).
\end{align*}
Applying Hölder's inequality with exponents $\gamma =  \frac{2}{p}$ and $\eta = \frac{2}{2-p}$, and writing
\begin{equation*}
    |x|^{\alpha p} = |x|^{\alpha p\left( \frac{1}{\gamma} + \frac{1}{\eta} \right)} = |x|^{\frac{\alpha p}{\gamma}} |x|^{\frac{\alpha p}{\eta}},
\end{equation*}
and using the above estimate, we obtain
\begin{align*}
    \int_{\mathbb{R}^{N}} \frac{|\nabla v(x)|^{p}}{|x|^{N-p}} \, dx & = \int_{\mathbb{R}^{N}} \frac{|\nabla v(x)|^{p}}{|x|^{N-p+ \alpha p}} |x|^{\alpha p}  \frac{\left( |v(x) \nabla v_{\alpha}(x)| + |v_{\alpha}(x) \nabla v(x)| \right)^{(p-2)\frac{p}{2}}}{\left( |v(x) \nabla v_{\alpha}(x)| + |v_{\alpha}(x) \nabla v(x)| \right)^{(p-2)\frac{p}{2}}} \, dx \\ & \leq \left( \int_{\mathbb{R}^{N}} \frac{|\nabla v(x)|^{2}}{|x|^{\left(\frac{N-p+ \alpha p}{p} \right) 2}}  |x|^{\alpha p} \left( |v(x) \nabla v_{\alpha}(x)| + |v_{\alpha}(x) \nabla v(x)| \right)^{p-2} \, dx \right)^{\frac{p}{2}} \\ & \quad \times \left( \int_{\mathbb{R}^{N}} |x|^{\alpha p} \left(  |v \nabla v_{\alpha}| + |v_{\alpha} \nabla v|  \right)^{p} \, dx \right)^{\frac{2-p}{2}} 
\\ & \leq C (\delta_{p, \alpha}(u))^{\frac{p}{2}} \left( \int_{\mathbb{R}^{N}} |x|^{\alpha p} \left(  |v \nabla v_{\alpha}| + |v_{\alpha} \nabla v|  \right)^{p} \, dx \right)^{\frac{2-p}{2}},
\end{align*}
where $C=C(p)>0$. Using $u= v v_{\alpha}$, we have $|v_{\alpha} \nabla v| \leq |v_{\alpha} \nabla v + v \nabla v_{\alpha}| + |v \nabla v_{\alpha}| = |\nabla u|+ |v \nabla v_{\alpha}|$. Therefore, the above inequality reduces to 
\begin{equation*}
    \int_{\mathbb{R}^{N}} \frac{|\nabla v(x)|^{p}}{|x|^{N-p}} \, dx \leq C (\delta_{p, \alpha}(u))^{\frac{p}{2}} \left( \int_{\mathbb{R}^{N}} |x|^{\alpha p} \left( 2 |v \nabla v_{\alpha}| + |\nabla u|  \right)^{p} \, dx \right)^{\frac{2-p}{2}}. 
\end{equation*}
Also, using Lemma \ref{Lemma: Hardy potential and Lorentz} with $s=1$ and the assumption $\|u\|_{L^{p^{*}_{\alpha},p}(\mathbb{R}^{N})} =1$, we have
\begin{equation*}
    \int_{\mathbb{R}^{N}} |v \nabla v_{\alpha}|^{p} |x|^{\alpha p} \, dx  =  \left( \frac{N-p+ \alpha p}{p} \right)^{p} \int_{\mathbb{R}^{N}} \frac{|u(x)|^{p}}{|x|^{p-\alpha p}} \, dx \leq \left( \frac{N-p+ \alpha p}{p} \right)^{p} \left( \frac{|\mathbb{S}^{N-1}|}{N} \right)^{\frac{p-\alpha p}{N}}.
\end{equation*}
Therefore, using the above estimate and the inequality \eqref{ineq7654}, we arrive at 
\begin{equation*}
    \int_{\mathbb{R}^{N}} \frac{|\nabla v(x)|^{p}}{|x|^{N-p}} \, dx \leq C (\delta_{p, \alpha}(u))^{\frac{p}{2}}. 
\end{equation*}
This finishes the proof of Lemma.
\end{proof}

\section{Proof of Quantitative stablity for weighted Hardy inequality}\label{Section 4}

In this section, we prove Theorem \ref{Theorem: Quantitative Stabiliy Weighted Hardy inequality} and establish the quantitative stability of the weighted Hardy inequality \eqref{Weighted Hardy inequality}. The weighted Hardy inequality with a remainder term, established in the previous section, reduces the proof to estimating this remainder term.

\subsection{Proof of Theorem \ref{Theorem: Quantitative Stabiliy Weighted Hardy inequality}}
We first consider the case of nonnegative functions. Let $u \in C^{1}_{c}(\mathbb{R}^{N})$ with $u \geq 0$. Let $p>1$, $\alpha \in [0,1)$ be such that $0<p-\alpha p<N$, and let $\beta := \max \{4, 2p \} $. We also assume that 
\begin{equation*}
    \| u \|_{L^{p^{*}_{\alpha},p}(\mathbb{R}^{N})} = 1,
\end{equation*}    
where $p^{*}_{\alpha} = \frac{Np}{N-p+\alpha p}$. First assume that $\delta_{p, \alpha}(u)>1$. Since  $L^{p^{*}_{\alpha},p}(\mathbb{R}^{N}) \hookrightarrow L^{p^{*}_{\alpha},\infty}(\mathbb{R}^{N})$ (see Proposition \ref{Proposition on Lorentz space}), the normalization above yields
\begin{equation*}
   \delta_{p, \alpha}(u)^{\frac{1}{\beta}} > 1 =  \| u\|_{L^{p^{*}_{\alpha},p}(\mathbb{R}^{N})} \geq C \| u\|_{L^{p^{*}_{\alpha},\infty}(\mathbb{R}^{N})} \geq C \inf_{a \geq 0} \| u- a \, v_{\alpha}\|_{L^{p^{*}_{\alpha},\infty}(\mathbb{R}^{N})},
\end{equation*}
where $v_{\alpha} =|x|^{-\frac{N-p+\alpha p}{p}}$. Consequently,
\begin{equation*}
    \inf_{a \geq 0} \| u- a \, v_{\alpha}\|^{\beta}_{L^{p^{*}_{\alpha},\infty}(\mathbb{R}^{N})} \leq \delta_{p, \alpha}(u). 
\end{equation*}
It remains to consider the case $\delta_{p, \alpha}(u) \leq 1$. Let $v(x) = u(x) |x|^{\frac{N-p+ \alpha p}{p}}=  u(x) v^{-1}_{\alpha}(x)$. For $a>0$, to be choosen later, and define the set 
\begin{equation*}
    \mathcal{F}= \{ x \in \mathbb{R}^{N} : v(x)> \delta_{p,\alpha}(u)^{a} \}.
\end{equation*}
Since $u \in C^{1}_{c}(\mathbb{R}^{N})$, the set $\mathcal{F}$ is open and bounded. 

For $\ell\in\mathbb{Z}$, introduce the dyadic annuli
\begin{equation}\label{Defn: A_{ell}}
    \mathcal{A}_{\ell} := \{ x \in \mathbb{R}^{N} : 2^{\ell} < |x| \leq 2^{\ell+1} \}.
\end{equation}
Moreover, $\lim_{|x| \to 0} v(x) = 0$. Hence, there exists smallest $r_{0} \in \mathbb{Z}$ such that 
\begin{equation*}
    v(x) \leq \delta_{p, \alpha}(u)^{a}, \quad \text{for every } x \in \mathcal{A}_{r_{0}} ,
\end{equation*}
while $\mathcal{A} \cap \mathcal{A}_{r_{0}+1} \neq \phi$. In particular, $(v)_{\mathcal{A}_{r_{0}}} \leq \delta_{p, \alpha}(u)^{a}$, where $(v)_{\mathcal{A}_{r_{0}}}$ denotes the average of $v$ over $\mathcal{A}_{r_{0}}$ (see \eqref{Defn: average of u}).
Therefore, for every $x \in \mathcal{F}$,
\begin{equation*}
    0< v(x) - \delta_{p, \alpha}(u)^{a} \leq v(x) - (v)_{\mathcal{A}_{r_{0}}}.
\end{equation*} 

\begin{figure}[H]
    \centering
    \includegraphics[width=0.8\textwidth]{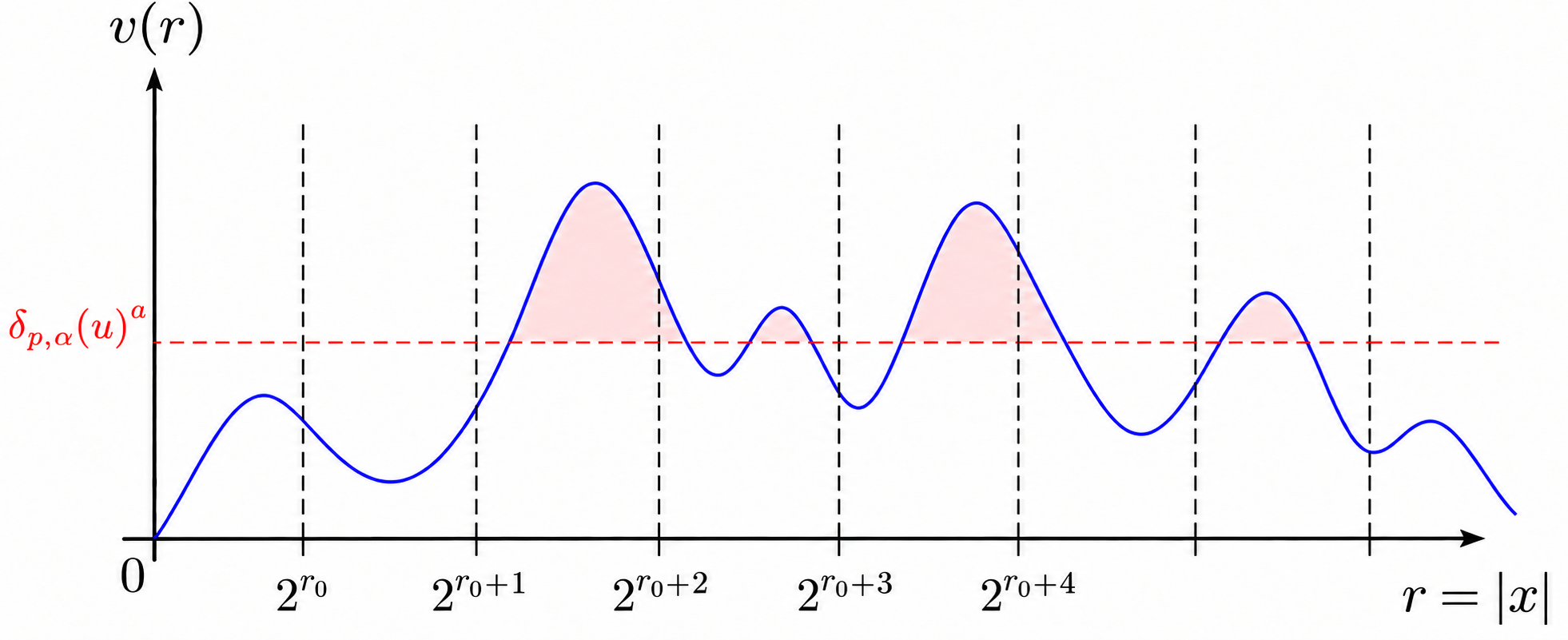}
    \caption{A possible radial profile of $v(x)$.}
    \label{fig:radial-profile}
\end{figure}

It follows that
\begin{align}\label{ineq111}
    \int_{\mathcal{F}} |v(x) - \delta_{p, \alpha}(u)^{a}|^{p^{*}_{\alpha}} \, \frac{dx}{|x|^{N}} & \leq \int_{\mathcal{F}} |v(x) - (v)_{\mathcal{A}_{r_{0}}}|^{p^{*}_{\alpha}} \, \frac{dx}{|x|^{N}} \nonumber \\ 
    & = \sum_{\substack{\ell \in \mathbb{Z} \\ \mathcal{F} \cap \mathcal{A}_{\ell} \neq \phi}} \int_{\mathcal{F} \cap \mathcal{A}_{\ell}} |v(x) - (v)_{\mathcal{A}_{r_{0}}}|^{p^{*}_{\alpha}} \, \frac{dx}{|x|^{N}}. 
\end{align}
For each $\ell \in \mathbb{Z}$ satisfying $\mathcal{F} \cap \mathcal{A}_{\ell} \neq \phi$, we derive
\begin{align}\label{ineq9877}
  \int_{\mathcal{F} \cap \mathcal{A}_{\ell}} |v(x) - (v)_{\mathcal{A}_{r_{0}}}|^{p^{*}_{\alpha}} \, \frac{dx}{|x|^{N}} & \leq 2^{p^{*}_{\alpha}-1} \int_{\mathcal{F} \cap \mathcal{A}_{\ell}} |v(x) - (v)_{\mathcal{A}_{\ell}}|^{p^{*}_{\alpha}} \, \frac{dx}{|x|^{N}} \nonumber \\ & \quad + 2^{p^{*}_{\alpha}-1} |(v)_{\mathcal{A}_{\ell}} - (v)_{\mathcal{A}_{r_{0}}}|^{p^{*}_{\alpha}} \int_{\mathcal{F} \cap \mathcal{A}_{\ell}}  \, \frac{dx}{|x|^{N}} \nonumber \\ & \leq 2^{p^{*}_{\alpha}-1} \int_{\mathcal{A}_{\ell}} |v(x) - (v)_{\mathcal{A}_{\ell}}|^{p^{*}_{\alpha}} \, \frac{dx}{|x|^{N}} \nonumber \\ & \quad + 2^{p^{*}_{\alpha}-1} |(v)_{\mathcal{A}_{\ell}} - (v)_{\mathcal{A}_{r_{0}}}|^{p^{*}_{\alpha}} \int_{\mathcal{F} \cap \mathcal{A}_{\ell}}  \, \frac{dx}{|x|^{N}}  .
\end{align}
We estimate the first time on the right-hand side by means of Lemma \ref{Lemma: Weighted Sobolev ineq Lambda}. Taking $n=2$ and $\lambda=2^{\ell}$, we obtain
\begin{align*}
    \int_{\mathcal{A}_{\ell}} |v(x) - (v)_{\mathcal{A}_{\ell}}|^{p^{*}_{\alpha}} \, \frac{dx}{|x|^{N}}  & \leq C \frac{1}{|\mathcal{A}_{\ell}|} \int_{\mathcal{A}_{\ell}} |v(x) - (v)_{\mathcal{A}_{\ell}}|^{p^{*}_{\alpha}} \, dx \\ &  \leq C \left( 2^{\ell(p-\alpha p-N)} \int_{\mathcal{A}_{\ell}} |\nabla v(x)|^{p} |x|^{\alpha p} \, dx  \right)^{\frac{p^{*}_{\alpha}}{p}} .
\end{align*}
The constant $C=C(N,p,\alpha)$ is independent of $\ell$. For all $x \in \mathcal{A}_{\ell}$, we have $|x| \leq 2^{\ell+1}$, and hence  $|x|^{(N-p+ \alpha p)} \leq 2^{(\ell+1)(N-p+ \alpha p)}$. Therefore,
\begin{equation*}
\int_{\mathcal{A}_{\ell}} |v(x) - (v)_{\mathcal{A}_{\ell}}|^{p^{*}_{\alpha}} \, \frac{dx}{|x|^{N}} \leq C \left( \int_{\mathcal{A}_{\ell}}  \, \frac{|\nabla v(x)|^{p}}{|x|^{N-p}}  \, dx \right)^{\frac{p^{*}_{\alpha}}{p}},
\end{equation*}
where $C=C(N,p, \alpha)>0$. Therefore, the estimate \eqref{ineq9877}, reduces to 
\begin{align}\label{ineq87651}
 \int_{\mathcal{F} \cap \mathcal{A}_{\ell}} |v(x) - (v)_{\mathcal{A}_{r_{0}}}|^{p^{*}_{\alpha}} \, \frac{dx}{|x|^{N}} & \leq C \left( \int_{\mathcal{A}_{\ell}}   \frac{|\nabla v(x)|^{p}}{|x|^{N-p}}  \, dx\right)^{\frac{p^{*}_{\alpha}}{p}} \nonumber \\ & \quad + 2^{p^{*}_{\alpha}-1} |(v)_{\mathcal{A}_{\ell}} - (v)_{\mathcal{A}_{r_{0}}}|^{p^{*}_{\alpha}} \int_{\mathcal{F} \cap \mathcal{A}_{\ell}}  \, \frac{dx}{|x|^{N}}.
\end{align}
We next control the difference of the average appearing in the above inequality. Applying Lemma \ref{Lemma: on two disjoint set} with $E = \mathcal{A}_{k}$ and $F = \mathcal{A}_{k+1}$, for $r_{0} \leq k \leq \ell-1$, we obtain
\begin{align*}
  |(v)_{\mathcal{A}_{\ell}} - (v)_{\mathcal{A}_{r_{0}}}|^{p^{*}_{\alpha}} & \leq 2^{p^{*}_{\alpha}-1} \sum_{k = r_{0}}^{\ell-1}  |(v)_{\mathcal{A}_{k}} - (v)_{\mathcal{A}_{k+1}}|^{p^{*}_{\alpha}} \\ & \leq C \sum_{k = r_{0}}^{\ell-1} \fint_{\mathcal{A}_{k} \cup \mathcal{A}_{k+1}}  |v(x) - (v)_{\mathcal{A}_{k} \cup \mathcal{A}_{k+1}}|^{p^{*}_{\alpha}} \, dx, 
\end{align*}
where $C=C(N,p, \alpha)>0$. We apply Lemma \ref{Lemma: Weighted Sobolev ineq Lambda} with $n=4$ and $\lambda = 2^{k}$ and using $|x|^{(N-p+ \alpha p)} \leq 2^{(k+2)(N-p+ \alpha p)}$ for all $x\in A_{k} \cup A_{k+1}$, we obtain
\begin{align*}
    |(v)_{\mathcal{A}_{\ell}} - (v)_{\mathcal{A}_{r_{0}}}|^{p^{*}_{\alpha}} & \leq C \sum_{k = r_{0}}^{\ell-1} \left( 2^{k(p-\alpha p-N)}  \int_{A_{k} \cup A_{k+1}} |\nabla v(x)|^{p} |x|^{\alpha p} \, dx \right)^{\frac{p^{*}_{\alpha}}{p}} \\ & \leq C \sum_{k = r_{0}}^{\ell-1} \left(  \int_{A_{k} \cup A_{k+1}}   \frac{|\nabla v(x)|^{p}}{|x|^{N-p}} \, dx \right)^{\frac{p^{*}_{\alpha}}{p}}  \leq C \left( \int_{\mathbb{R}^{N}}  \, \frac{|\nabla v(x)|^{p}}{|x|^{N-p}} \, dx \right)^{\frac{p^{*}_{\alpha}}{p}},
\end{align*}
where $C = C(N,p, \alpha)>0$.
On the other hand, if $x \in \mathcal{F} \cap \mathcal{A}_{\ell}$, we have $v(x) > \delta_{p, \alpha}(u)^a$. Hence,
\begin{equation*}
    \int_{\mathcal{F} \cap \mathcal{A}_{\ell}}  \, \frac{dx}{|x|^{N}} \leq \frac{1}{\delta_{p, \alpha}(u)^{a p}} \int_{\mathcal{F} \cap \mathcal{A}_{\ell}} \frac{|u(x)|^{p}}{|x|^{p- \alpha p}} \, dx.
\end{equation*}
Therefore, combining the above two estimates, and using Lemma \ref{Lemma: weighted Hardy remainder p geq 2} and Lemma \ref{Lemma: Local Hardy Remainder 1<p<2} under the assumptions $\delta_{p, \alpha}(u) \leq 1$ and $\|u\|_{L^{p^{*}_{\alpha},p}(\mathbb{R}^{N})} = 1$, we obtain
\begin{align*}
    |(v)_{\mathcal{A}_{\ell}}  -  (v)_{\mathcal{A}_{r_{0}}}|^{p^{*}_{\alpha}}  \int_{\mathcal{F} \cap \mathcal{A}_{\ell}}  \, \frac{dx}{|x|^{N}} & \leq C \delta_{p, \alpha}(u)^{\frac{p^{*}_{\alpha}}{p} b}  \, \frac{1}{\delta_{p, \alpha}(u)^{a p}} \int_{\mathcal{F} \cap \mathcal{A}_{\ell}} \frac{|u(x)|^{p}}{|x|^{p-\alpha p}} \, dx \\ & = C \delta_{p, \alpha}(u)^{\frac{p^{*}_{\alpha}}{p} b - a p} \int_{\mathcal{F} \cap \mathcal{A}_{\ell}} \frac{|u(x)|^{p}}{|x|^{p-\alpha p}} \, dx,
\end{align*}
where $b:=\min \{ 1, \frac{p}{2} \}$. Substituting the above estimate into \eqref{ineq87651}, we obtain
\begin{equation*}
    \int_{\mathcal{F} \cap \mathcal{A}_{\ell}} |v(x) - (v)_{\mathcal{A}_{r_{0}}}|^{p^{*}_{\alpha}} \, \frac{dx}{|x|^{N}}  \leq C \left( \int_{\mathcal{A}_{\ell}}  \frac{|\nabla v(x)|^{p}}{|x|^{N-p}}  \, dx \right)^{\frac{p^{*}_{\alpha}}{p}}+ C \delta_{p, \alpha}(u)^{\frac{p^{*}_{\alpha}}{p} b - a p} \int_{\mathcal{F} \cap \mathcal{A}_{\ell}} \frac{|u(x)|^{p}}{|x|^{p- \alpha p}} \, dx.
\end{equation*}
Therefore, combining \eqref{ineq111} with the estimates above, and using Lemma \ref{Lemma: weighted Hardy remainder p geq 2} and Lemma \ref{Lemma: Local Hardy Remainder 1<p<2} under the assumptions $\delta_{p, \alpha}(u) \leq 1$ and the normalization $\| u \|_{L^{p^{*}_{\alpha},p}(\mathbb{R}^{N})} = 1$, together with Lemma \ref{Lemma: Hardy potential and Lorentz} with $s = 1$, we obtain
\begin{align*}
    \int_{\mathcal{F}} |v(x) - \delta_{p, \alpha}(u)^{a}|^{p^{*}_{\alpha}} \, \frac{dx}{|x|^{N}} & \leq \sum_{\substack{\ell \in \mathbb{Z} \\ \mathcal{F} \cap \mathcal{A}_{\ell} \neq \phi}} \int_{\mathcal{F} \cap \mathcal{A}_{\ell}} |v(x) - (v)_{\mathcal{A}_{r_{0}}}|^{p^{*}_{\alpha}} \, \frac{dx}{|x|^{N}} \\ & \leq C \delta_{p, \alpha}(u)^{\frac{p^{*}_{\alpha}}{p}b} + C \delta_{p, \alpha}(u)^{\frac{p^{*}_{\alpha}}{p}b - a p} \int_{\mathbb{R}^{N}} \frac{|u(x)|^{p}}{|x|^{p- \alpha p}} \, dx \\ & \leq C \left(\delta_{p, \alpha}(u)^{\frac{p^{*}_{\alpha}}{p}b} + \delta_{p, \alpha}(u)^{\frac{p^{*}_{\alpha}}{p} b - a p} \right),
\end{align*}
where $C=C(N,p, \alpha)>0$. Choose $a = \frac{p^{*}_{\alpha}}{2 p^{2}} b$.  Then, using the assumption $\delta_{p, \alpha}(u) \leq 1$, we obtain
\begin{align*}
    \int_{\mathcal{F}} |v(x) - \delta_{p, \alpha}(u)^{a}|^{p^{*}_{\alpha}} \, \frac{dx}{|x|^{N}} & \leq C \left( \delta_{p, \alpha}(u)^{\frac{p^{*}_{\alpha}}{p} b} + \delta_{p, \alpha}(u)^{\frac{p^{*}_{\alpha}}{2p} b} \right) \\ & \leq C \left( \delta_{p, \alpha}(u)^{\frac{p^{*}_{\alpha}}{2p}b} + \delta_{p, \alpha}(u)^{\frac{p^{*}_{\alpha}}{2p}b }  \right) \leq C \delta_{p, \alpha}(u)^{\frac{p^{*}_{\alpha}}{2p}b} .
\end{align*}
Since $v_{\alpha}(x) = |x|^{- \frac{N-p+ \alpha p}{p}}$, the Lorentz embedding gives (see Proposition \ref{Proposition on Lorentz space} and note that $L^{p^{*}_{\alpha}, p^{*}_{\alpha}} = L^{p^{*}_{\alpha}}$) 
\begin{align*}
    \| u - \delta_{p, \alpha}(u)^{a} \, v_{\alpha} \|_{L^{p^{*}_{\alpha}, \infty}(\mathcal{F})} & \leq C  \left( \int_{\mathcal{F}} | u(x) - \delta_{p, \alpha}(u)^{a} \, v_{\alpha}(x) |^{p^{*}_{\alpha}} \, dx \right)^{\frac{1}{p^{*}_{\alpha}}} \\ &  =  C  \left( \int_{\mathcal{F}} |v(x) - \delta_{p, \alpha}(u)^{a}|^{p^{*}_{\alpha}} \, \frac{dx}{|x|^{N}} \right)^{\frac{1}{p^{*}_{\alpha}}}  \leq C \delta_{p, \alpha}(u)^{\frac{b}{2p}}.
\end{align*}
Consequently,
\begin{equation*}
    \| u - \delta_{p, \alpha}(u)^{a}  \,v_{\alpha}\|^{\beta}_{L^{p^{*}_{\alpha}, \infty}(\mathcal{F})} \leq C \delta_{p, \alpha}(u) ,
\end{equation*}
where $\beta= \max \{4, 2p \}$. It remains to estimate the same quantity on $\mathcal{F}^c$.
For $x\in\mathcal{F}^c$, $v(x)\leq
    \delta_{p, \alpha}(u)^{a}$, and hence
\begin{equation*}
 v^{-1}_{\alpha}(x) |u(x)- \delta_{p, \alpha}(u)^{a} \,  v_{\alpha}(x)| =  |v(x) - \delta_{p, \alpha}(u)^{a}| \leq 2 \delta_{p, \alpha}(u)^{a}, \quad \forall \ x \in \mathcal{F}^{c},
\end{equation*}
and therefore
\begin{equation*}
     |u(x)- \delta_{p, \alpha}(u)^{a} \, v_{\alpha}(x)| \leq 2 \delta_{p, \alpha}(u)^{a} \, v_{\alpha}(x), \quad \forall \ x \in \mathcal{F}^{c} .
\end{equation*}
Since $v_{\alpha} \in L^{p^{*}_{\alpha}, \infty}(\mathbb{R}^{N})$, we obtain
\begin{equation*}
    \|u- \delta_{p, \alpha}(u)^{a} \, v_{\alpha}(x) \|_{L^{p^{*}_{\alpha}, \infty}(\mathcal{F}^{c})} \leq C \delta_{p, \alpha}(u)^{a} \| v_{\alpha} \|_{L^{p^{*}_{\alpha}, \infty}(\mathcal{F}^{c})} \leq C \delta_{p, \alpha}(u)^{\frac{p^{*}_{\alpha}}{2p^{2}}b} \leq C \delta_{p, \alpha}(u)^{\frac{b}{2p}}. 
\end{equation*}
In the above inequality, we have used $\delta_{p, \alpha}(u) \leq 1$ and $\frac{p^{*}_{\alpha}}{p}>1$. Consequently,
\begin{equation*}
     \|u- \delta_{p, \alpha}(u)^{a} \,  v_{\alpha}\|^{\beta}_{L^{p^{*}_{\alpha}, \infty}(\mathcal{F}^{c})} \leq C \delta_{p, \alpha}(u).
\end{equation*}
Combining the estimates on $\mathcal{F}$ and $\mathcal{F}^c$, we obtain
\begin{align*}
    \|u- \delta_{p, \alpha}(u)^{a} \, v_{\alpha}\|_{L^{p^{*}_{\alpha}, \infty}(\mathbb{R}^{N})} & \leq C \Big( \|u- \delta_{p, \alpha}(u)^{a} \, v_{\alpha}\|_{L^{p^{*}_{\alpha}, \infty}(\mathcal{F})} \\ & \quad \quad  + \|u- \delta_{p, \alpha}(u)^{a} \, v_{\alpha}\|_{L^{p^{*}_{\alpha}, \infty}(\mathcal{F}^{c})} \Big) \\ & \leq C \delta_{p, \alpha}(u)^{\frac{b}{2p}}. 
\end{align*}
Therefore, combining the cases $\delta_{p, \alpha}(u) > 1$ and $\delta_{p, \alpha}(u) \leq 1$, and under the normalization $\| u \|_{L^{p^{*}_{\alpha},p}(\mathbb{R}^{N})} =1$, we obtain
\begin{equation*}
   \inf_{a \geq 0} \|u- a \, v_{\alpha}\|^{\beta}_{L^{p^{*}_{\alpha}, \infty}(\mathbb{R}^{N})} \leq \|u- \delta_{p, \alpha}(u)^{a} \, v_{\alpha}\|^{\beta}_{L^{p^{*}_{\alpha}, \infty}(\mathbb{R}^{N})} \leq C \delta_{p, \alpha}(u). 
\end{equation*}
We next remove the normalization. Set
\begin{equation*}
    \widetilde{u} = \frac{u}{\| u \|_{L^{p^{*}_{\alpha},p}(\mathbb{R}^{N})} }.
\end{equation*} 
Applying the normalized estimate to $\widetilde{u}$ and using the
homogeneity of the involved quantities, we obtain
\begin{align*}
  \inf_{a \geq 0} \frac{\| u -  a \, v_{\alpha}\|^{\beta}_{L^{p^{*}_{\alpha},\infty}(\mathbb{R}^{N})}}{ \| u \|^{\beta}_{L^{p^{*}_{\alpha},p}(\mathbb{R}^{N})} }  \leq C \frac{\delta_{p, \alpha}(u)}{ \| u \|^{p}_{L^{p^{*}_{\alpha},p}(\mathbb{R}^{N})}  }.   
\end{align*}
Hence, for every nonnegative function $u \in C^{1}_{c}(\mathbb{R}^{N})$,
\begin{equation}\label{ineq112}
    \inf_{a \geq 0} \| u - a v_{\alpha}\|^{\beta}_{L^{p^{*}_{\alpha},\infty}(\mathbb{R}^{N})} \leq C \, \delta_{p, \alpha}(u)\| u \|^{\beta - p}_{L^{p^{*}_{\alpha},p}(\mathbb{R}^{N})}  ,
\end{equation}
where $\beta= \max \{4, 2p \}$. We finally pass to arbitrary real-valued functions. Let
$u\in C^{1}_{c}(\mathbb{R}^N)$ and write $u_{+}(x) = \max \{ u(x),0 \} $ and $u_{-}(x) = \max \{ -u(x),0 \}$, the positive and the negative parts of $u$ respectively, so that $u=u_{+}- u_{-}$. For $b,c\geq0$, the triangle inequality gives
\begin{align*}
  \inf_{a \in \mathbb{R}}  \| u  - a\, v_{ \alpha}\|_{L^{p^{*}_{\alpha},\infty}(\mathbb{R}^{N})} & = \inf_{b,c \geq 0 }  \| u_{+} - u_{-} - (b-c)\, v_{\alpha}\|_{L^{p^{*}_{\alpha},\infty}(\mathbb{R}^{N})}  \\ & \leq \inf_{b \geq 0}  \| u_{+} -b \,  v_{\alpha}\|_{L^{p^{*}_{\alpha},\infty}(\mathbb{R}^{N})}  + \inf_{c \geq 0}  \| u_{-} - c\, v_{\alpha}\|_{L^{p^{*}_{\alpha},\infty}(\mathbb{R}^{N})}.
\end{align*}
Applying \eqref{ineq112} to $u_+$ and $u_-$, respectively, we obtain
\begin{align*}
    \sum_{\pm} \inf_{a \geq 0} \| u_{\pm} - a\, v_{\alpha} \|_{L^{p^{*}_{\alpha}, \infty}(\mathbb{R}^{N})} & \leq C \sum_{\pm}  \left( \delta_{p, \alpha}(u_{\pm}) \right)^{\frac{1}{\beta}} \| u_{\pm} \|^{\frac{\beta - p}{\beta}}_{L^{p^{*}_{\alpha},p}(\mathbb{R}^{N})}   \\ & \leq C \left( \delta_{p, \alpha}(u_{+})+ \delta_{p, \alpha}(u_{-}) \right)^{\frac{1}{\beta}}  \| u \|^{\frac{\beta - p}{\beta}}_{L^{p^{*}_{\alpha},p}(\mathbb{R}^{N})}  \\ & \leq C \left( \delta_{p, \alpha}(u)  \right)^{\frac{1}{\beta}}  \| u \|^{\frac{\beta - p}{\beta}}_{L^{p^{*}_{\alpha},p}(\mathbb{R}^{N})} .
\end{align*}
Combining the preceding estimates with Lemma
\ref{Lemma: Hardy potential and Lorentz} with $s=1$, we conclude that
\begin{align*}
\left( \frac{N}{\mathbb{S}^{N-1}} \right)^{\frac{p- \alpha p}{N}} & \inf_{a \in \mathbb{R}} \frac{ \| u -a\, v_{\alpha}\|^{\beta}_{L^{p^{*}_{\alpha},\infty}(\mathbb{R}^{N})}}{\| u \|^{\beta}_{L^{p^{*}_{\alpha},p}(\mathbb{R}^{N})}}  \int_{\mathbb{R}^{N}} \frac{|u(x)|^{p}}{|x|^{p-\alpha p}} \, dx \\ & \leq  \inf_{a \in \mathbb{R}} \frac{ \| u - a\, v_{\alpha}\|^{\beta}_{L^{p^{*}_{\alpha},\infty}(\mathbb{R}^{N})}}{\| u \|^{\beta}_{L^{p^{*}_{\alpha},p}(\mathbb{R}^{N})}}  \| u \|^{p}_{L^{p^{*}_{\alpha},p}(\mathbb{R}^{N})}  \leq C \, \delta_{p, \alpha}(u).
\end{align*}
Therefore, using the definition of $d_{p, \alpha}$, defined in \eqref{Defn: d_p, alpha}, we obtain
\begin{equation*}
    \delta_{p, \alpha} (u) \geq C \left( \int_{\mathbb{R}^{N}} \frac{|u(x)|^{p}}{|x|^{p-\alpha p}} \, dx  \right) d_{p, \alpha}(u)^{\beta},
\end{equation*}
where $\beta= \max \{ 4, 2p \}$. This proves Theorem \ref{Theorem: Quantitative Stabiliy Weighted Hardy inequality}.

\section{Extending a \texorpdfstring{$W^{s,p, \alpha}(\Omega)$}{Wsp} functions to the whole of \texorpdfstring{$\mathbb{R}^{N}$}{Rn}}\label{Section : 5}

In this section, we establish the extension theorem for the weighted fractional Sobolev space $W^{s,p,\alpha}$. In particular, we show that every open set $\Omega$ of class
$C^{0,1}$ with bounded boundary and $0 \notin \partial \Omega$ is an extension domain for $W^{s,p,\alpha}$. This result will be useful in establishing the weighted fractional Sobolev inequality on a bounded Lipschitz domain $\Omega$ with $0 \notin \partial \Omega$. Throughout this section, we assume that $\alpha=\alpha_{1}+\alpha_{2}\geq0$ and $\alpha_{1}p, \, \alpha_{2}p\in(-N,sp)$. These conditions ensure that the weighted Gagliardo seminorm \eqref{Weighted Gagliardo Seminorm} is finite for every $C_c^1$ function (see \cite[Lemma $2.1$]{Valdinoci2015}). Consequently, the weighted fractional Sobolev space $W^{s,p,\alpha}$ is well defined under these assumptions. For the unweighted fractional Sobolev space, this extension theorem was established in \cite[Section $5$]{HitchhikersGuide2012}.

\smallskip

The next lemma establishes an extension result for functions $u$ whose support is compactly contained in an open set $\Omega$.

\begin{lemma}\label{Lemma: Extension 1}
Let $p \geq 1$, $s\in(0,1)$, and let $\alpha_1,\alpha_2\in\mathbb{R}$ with $\alpha=\alpha_1+\alpha_2 \geq 0$ satisfy $\alpha_{1} p, \, \alpha_{2} p\in(-N,sp)$. Let $\Omega\subset\mathbb{R}^N$ be an open set and let $u\in W^{s,p,\alpha}(\Omega)$. Suppose that there exists a compact set $K\subset \subset \Omega$ such that $u\equiv 0$ in $\Omega\setminus K$. Define the zero extension of $u$ to $\mathbb{R}^N$ by
\begin{equation*}
    \widetilde u(x):=
\begin{cases}
u(x), & x\in\Omega,\\
0, & x\in\mathbb{R}^N\setminus\Omega.
\end{cases}
\end{equation*}
Then $\widetilde u\in W^{s,p,\alpha}(\mathbb{R}^N)$ and there exists a constant $C=C(N,p,s,\alpha_1,\alpha_2,K,\Omega)>0$ such that
\begin{equation*}
    \|\widetilde u\|_{W^{s,p,\alpha}(\mathbb{R}^N)}
\leq
C \|u\|_{W^{s,p,\alpha}(\Omega)} .
\end{equation*}
\end{lemma}
\begin{proof}
First, by the definition of the zero extension,
\begin{align}\label{ineq01}
   \|\widetilde u\|_{\mathcal{L}^{p,\alpha}(\mathbb{R}^N)}^p
&=
\int_{\mathbb{R}^N}
|\widetilde u(x)|^p \left( |x|^{\alpha_{1} p} + |x|^{\alpha_{2} p} \right)\,dx \nonumber \\
&=
\int_{\Omega}
| u(x)|^p \left( |x|^{\alpha_{1} p} + |x|^{\alpha_{2} p} \right)\,dx =
\|u\|_{\mathcal{L}^{p,\alpha}(\Omega)}^p. 
\end{align}

It remains to estimate the weighted Gagliardo seminorm. By the definition of $\widetilde u$, we have
\begin{align*}
[\widetilde u]_{W^{s,p,\alpha}(\mathbb{R}^N)}^p
={}&
\int_{\Omega}\int_{\Omega}
\frac{|u(x)-u(y)|^p}
{|x-y|^{N+sp}}
|x|^{\alpha_{1} p}|y|^{\alpha_{2} p}
\,dx\,dy
\\
&+
\int_{\Omega}\int_{\mathbb{R}^N\setminus\Omega}
\frac{|u(x)|^p}
{|x-y|^{N+sp}}
|x|^{\alpha_{1} p}|y|^{\alpha_{2} p}
\,dy\,dx \\ & + \int_{\mathbb{R}^N\setminus\Omega}\int_{\Omega}
\frac{|u(y)|^p}
{|x-y|^{N+sp}}
|x|^{\alpha_{1} p}|y|^{\alpha_{2} p}
\,dy\,dx.
\end{align*}
Therefore,
\begin{align*}
[\widetilde u]_{W^{s,p,\alpha}(\mathbb{R}^N)}^p
= &
[u]_{W^{s,p,\alpha}(\Omega)}^p
+
\int_{\Omega} |u(x)|^p |x|^{\alpha_{1} p} \left( \int_{\mathbb{R}^N\setminus\Omega}
\frac{|y|^{\alpha_{2} p}}
{|x-y|^{N+sp}}
\,dy \right) dx \\ & + \int_{\Omega} |u(y)|^p |y|^{\alpha_{2} p} \left( \int_{\mathbb{R}^N\setminus\Omega}  
\frac{|x|^{\alpha_{1} p}}
{|x-y|^{N+sp}}
\,dx \right) dy.
\end{align*}
Since $u\equiv0$ in $\Omega\setminus K$, the second and third integral reduces to
\begin{equation*}
I_{i,j}:=
\int_K |u(x)|^p|x|^{\alpha_ip}
\left(
\int_{\mathbb{R}^N\setminus\Omega}
\frac{|y|^{\alpha_jp}}
{|x-y|^{N+sp}}\,dy
\right)dx, \quad \text{for } i,j \in \{ 1,2 \} \quad \text{and} \quad i \neq j.   
\end{equation*}
Since $K\subset \subset\Omega$, there exists $d>0$ such that $d:=\operatorname{dist}(K,\mathbb{R}^N\setminus\Omega)>0$. Choose $R>0$ such that $K\subset B_R(0)$. We split the inner integral into two parts:
\begin{equation*}
\int_{\mathbb{R}^N\setminus\Omega}
\frac{|y|^{\alpha_jp}}
{|x-y|^{N+sp}}\,dy
=
\int_{(\mathbb{R}^N\setminus\Omega)\cap B_{2R}}
\frac{|y|^{\alpha_jp}}
{|x-y|^{N+sp}}\,dy
+
\int_{(\mathbb{R}^N\setminus\Omega)\setminus B_{2R}}
\frac{|y|^{\alpha_jp}}
{|x-y|^{N+sp}}\,dy.  
\end{equation*}

For $x\in K$ and $y\in\mathbb{R}^N\setminus\Omega$, we have $|x-y|\geq d$. Consequently,
\begin{align*}
\int_{(\mathbb{R}^N\setminus\Omega)\cap B_{2R}}
\frac{|y|^{\alpha_jp}}
{|x-y|^{N+sp}}\,dy
\leq
d^{-(N+sp)}
\int_{B_{2R}}|y|^{\alpha_jp}\,dy \leq C,
\end{align*}
where the last integral is finite since $\alpha_j p>-N$. 

On the other hand, if $x\in K\subset B_R(0)$ and $|y|>2R$,  then $|y|> 2 |x|$ and therefore,
\begin{equation*}
    |x-y|
\geq |y|-|x| = \frac{|y|}{2} + \frac{|y|}{2} - |x|
\geq \frac{|y|}{2}.
\end{equation*}
Therefore using $\alpha_j p<sp$ and the above inequality, we obtain,
\begin{equation*}
\int_{(\mathbb{R}^N\setminus\Omega)\setminus B_{2R}}
\frac{|y|^{\alpha_jp}}
{|x-y|^{N+sp}}\,dy
\leq
C\int_{|y|>2R}
|y|^{-N-sp+\alpha_j p}\,dy \leq C.
\end{equation*}
Therefore, we obtain
\begin{equation*}
    I_{i,j} \leq C \int_K |u(x)|^p|x|^{\alpha_ip} \, dx, \quad \text{for } i,j \in \{ 1,2 \} \quad \text{and} \quad i \neq j. 
\end{equation*}
Hence,
\begin{equation*}
    [\widetilde u]_{W^{s,p,\alpha}(\mathbb{R}^N)}^p
\leq 
[u]_{W^{s,p,\alpha}(\Omega)}^p
+
C \int_{\Omega} |u(x)|^p  \left( |x|^{\alpha_{1} p} + |x|^{\alpha_{2} p}  \right) dx,
\end{equation*}
where $C=C(N,p,s, \alpha_{1}, \alpha_{2},K, \Omega)>0$. Combining the above inequality with \eqref{ineq01} completes the proof.
\end{proof}

The next lemma establishes the weighted reflection property for the fractional Sobolev space $W^{s,p,\alpha}$. In particular, it shows that the even reflection of a function defined on $\Omega_+$ belongs to $W^{s,p,\alpha}(\Omega)$, with its norm controlled by the norm of the original function on $\Omega_+$.

\begin{lemma}\label{Lemma: Extension 2}
Let $p\geq 1$, $s\in(0,1)$, and let  $\alpha_1,\alpha_2\in\mathbb{R}$
with  $\alpha=\alpha_1+\alpha_2 \geq 0$ satisfy $\alpha_{1} p, \, \alpha_{2} p\in(-N,sp)$. Let $\Omega\subset\mathbb{R}^N$ be an open set which is symmetric
with respect to the hyperplane $\{ x= (x',x_{N}) \in \mathbb{R}^{N-1} \times \mathbb{R} : x_N=0\}$. Set
\begin{equation*}
  \Omega_+
:=
\{x=(x',x_N)\in\Omega:x_N>0\},
\qquad
\Omega_-
:=
\{x=(x',x_N)\in\Omega:x_N<0\}.  
\end{equation*}
Let $u\in W^{s,p,\alpha}(\Omega_+)$. Define the even reflection $\overline{u}$ of $u$ by
\begin{equation*}
    \overline{u}(x',x_N)
:=
\begin{cases}
u(x',x_N), & x_N>0,\\[1mm]
u(x',-x_N), & x_N<0.
\end{cases}
\end{equation*}
Then $\overline{u}\in W^{s,p,\alpha}(\Omega)$ and
\begin{equation*}
  \|\overline{u}\|_{W^{s,p,\alpha}(\Omega)}
\leq
4\|u \|_{W^{s,p,\alpha}(\Omega_+)}.
\end{equation*}
\end{lemma}

\begin{proof}
For $x=(x',x_N)\in\mathbb{R}^N$, denote by
\begin{equation*}
    x^*:=(x',-x_N)
\end{equation*}
its reflection with respect to the hyperplane $\{x_N=0\}$.
Notice that $|x^*|=|x|$. Consequently,
\begin{equation*}
   |x^*|^{\alpha_{1} p}=|x|^{\alpha_{1} p},
\qquad
|x^*|^{\alpha_{2} p}=|x|^{\alpha_{2} p}. 
\end{equation*}
Thus, the weights appearing in the definition of $W^{s,p,\alpha}$ are invariant under the reflection. 

We first consider the weighted $L^p$-norm. By symmetry of $\Omega$,
\begin{align*}
  \|\overline{u}\|_{\mathcal{L}^{p,\alpha}(\Omega)}^p &= \int_{\Omega} |\overline{u}(x)|^p \left(|x|^{\alpha_{1} p} + |x|^{\alpha_{2} p} \right) \,dx\\ &= \int_{\Omega_+} |u(x)|^p\left(|x|^{\alpha_{1} p} + |x|^{\alpha_{2} p} \right)\,dx + \int_{\Omega_-} |u(x^*)|^p\left(|x|^{\alpha_{1} p} + |x|^{\alpha_{2} p} \right)\,dx.  
\end{align*}
Making the change of variables  $x=(x',-z_N)=z^*$, where $z\in\Omega_+$, in the second integral, and using $|z^*|=|z|$, we obtain
\begin{align*}
    \int_{\Omega_-} |u(x^*)|^p\left(|x|^{\alpha_{1} p} + |x|^{\alpha_{2} p} \right)\,dx &= \int_{\Omega_+} |u(z)|^p\left(|z^{*}|^{\alpha_{1} p} + |z^{*}|^{\alpha_{2} p} \right)\,dz\\ &= \int_{\Omega_+} |u(z)|^p\left(|z|^{\alpha_{1} p} + |z|^{\alpha_{2} p} \right)\,dz.
\end{align*}
Therefore,
\begin{equation*}
    \|\overline{u}\|^{p}_{\mathcal{L}^{p,\alpha}(\Omega)} = 2 \|\overline{u}\|^{p}_{\mathcal{L}^{p,\alpha}(\Omega_{+})}
\end{equation*}

We now estimate the weighted Gagliardo seminorm. Splitting
$\Omega=\Omega_+\cup\Omega_-$, we obtain
\begin{align*}
[\overline{u}]_{W^{s,p,\alpha}(\Omega)}^p
={}&
\int_{\Omega_+}\int_{\Omega_+}
\frac{|u(x)-u(y)|^p}
{|x-y|^{N+sp}}
|x|^{\alpha_{1} p}|y|^{\alpha_{2} p}
\,dx\,dy
\\
&+
\int_{\Omega_+}\int_{\Omega_-}
\frac{|u(x)-u(y^*)|^p}
{|x-y|^{N+sp}}
|x|^{\alpha_{1} p}|y|^{\alpha_{2} p}
\,dy\,dx
\\
&+
\int_{\Omega_-}\int_{\Omega_+}
\frac{|u(x^*)-u(y)|^p}
{|x-y|^{N+sp}}
|x|^{\alpha_{1} p}|y|^{\alpha_{2} p}
\,dy\,dx
\\
&+
\int_{\Omega_-}\int_{\Omega_-}
\frac{|u(x^*)-u(y^*)|^p}
{|x-y|^{N+sp}}
|x|^{\alpha_{1} p}|y|^{\alpha_{2} p}
\,dx\,dy.
\end{align*}
The first term is exactly $[u]_{W^{s,p,\alpha}(\Omega_+)}^p$.

For the last term, making the change of variables
\begin{equation*}
    x=z^*,\qquad y=w^*,
\qquad z,w\in\Omega_+,
\end{equation*}
and using
\begin{equation*}
   |z^*-w^*|=|z-w|,
\qquad
|z^*|=|z|,
\qquad
|w^*|=|w|, 
\end{equation*}
we obtain
\begin{equation*}
    \int_{\Omega_-}\int_{\Omega_-}
\frac{|u(x^*)-u(y^*)|^p}
{|x-y|^{N+sp}}
|x|^{\alpha_{1} p}|y|^{\alpha_{2} p}
\,dx\,dy
=
\int_{\Omega_+}\int_{\Omega_+}
\frac{|u(z)-u(w)|^p}
{|z-w|^{N+sp}}
|z|^{\alpha_{1} p}|w|^{\alpha_{2} p}
\,dz\,dw.
\end{equation*}

It remains to estimate the two cross terms. Consider
\begin{equation*}
    I:=
\int_{\Omega_+}\int_{\Omega_-}
\frac{|u(x)-u(y^*)|^p}
{|x-y|^{N+sp}}
|x|^{\alpha_{1} p}|y|^{\alpha_{2} p}
\,dy\,dx.
\end{equation*}
Making the change of variables $y=z^*$, where $ z\in\Omega_+$, and using $|y|=|z|$, we obtain
\begin{equation*}
    I
=
\int_{\Omega_+}\int_{\Omega_+}
\frac{|u(x)-u(z)|^p}
{|x-z^*|^{N+sp}}
|x|^{\alpha_{1} p}|z|^{\alpha_{2} p}
\,dz\,dx.
\end{equation*}
For $x=(x',x_N),z=(z',z_N)\in\Omega_+$, we have
\begin{align*}
|x-z|^2
&=
|x'-z'|^2+(x_N-z_N)^2,\\
|x-z^*|^2
&=
|x'-z'|^2+(x_N+z_N)^2.
\end{align*}
Since $x_N,z_N>0$, we have $(x_N-z_N)^2\leq(x_N+z_N)^2$, and therefore
\begin{equation*}
  |x-z|\leq|x-z^*|.  
\end{equation*}
Hence
\begin{equation*}
   I
\leq
\int_{\Omega_+}\int_{\Omega_+}
\frac{|u(x)-u(z)|^p}
{|x-z|^{N+sp}}
|x|^{\alpha_{1} p}|z|^{\alpha_{2} p}
\,dz\,dx.
\end{equation*}
The other cross term is estimated in exactly the same way. Therefore,
\begin{equation*}[\overline{u}]_{W^{s,p,\alpha}(\Omega)}^p \leq 
4[u]_{W^{s,p,\alpha}(\Omega_+)}^p.
\end{equation*}
This completes the proof.
\end{proof}

The next lemma establishes the truncation property for the weighted fractional Sobolev space $W^{s,p,\alpha}$. In particular, it shows that the product of a function in $W^{s,p,\alpha}(\Omega)$ with a bounded Lipschitz function remains in $W^{s,p,\alpha}(\Omega)$, with its norm controlled by the norm of the original function. This result will be useful in the construction of the extension operator.

\begin{lemma}\label{Lemma: Extension 3}
Let $p \geq 1$, $s\in(0,1)$, and let $\alpha_1, \, \alpha_2\in\mathbb{R}$,
with $\alpha=\alpha_1+\alpha_2\geq 0$ satisfy $ \alpha_{1} p, \, \alpha_{2} p\in(-N,sp)$. Let $\Omega\subset\mathbb{R}^N$ be an bounded open set. Suppose that $u\in W^{s,p,\alpha}(\Omega)$ and $\psi\in C^{0,1}(\Omega)$,
where $0\leq \psi \leq 1$. Then $\psi u\in W^{s,p,\alpha}(\Omega)$ and 
\begin{equation*}
    \|\psi u\|_{W^{s,p,\alpha}(\Omega)}
\leq
C  \|u\|_{W^{s,p,\alpha}(\Omega)},
\end{equation*}
where $C>0$ depends on $N,p,s,\alpha_1,\alpha_2, \Omega$ and the
Lipschitz constant.
\end{lemma}

\begin{proof}
Since the definition of $W^{s,p,\alpha}(\Omega)$ is symmetric with respect
to $\alpha_1$ and $\alpha_2$, and since $\alpha=\alpha_1+\alpha_2\geq0$,
we may assume without loss of generality that $\alpha_1\geq0$. Since $0 \leq \psi \leq1$, we immediately have
\begin{align*}
     \|\psi u\|_{\mathcal{L}^{p,\alpha}(\Omega)}^p &= \int_\Omega |\psi(x)u(x)|^p \left( |x|^{\alpha_{1} p} + |x|^{\alpha_{2} p} \right) \,dx\\ &\leq \int_\Omega |u(x)|^p\left( |x|^{\alpha_{1} p} + |x|^{\alpha_{2} p} \right)\,dx = \|u\|_{\mathcal{L}^{p,\alpha}(\Omega)}^p.
\end{align*}

It remains to estimate the weighted Gagliardo seminorm.
We write
\begin{equation*}
\psi(x)u(x)-\psi(y)u(y)
=
\psi(x)\bigl(u(x)-u(y)\bigr)
+
u(y)\bigl(\psi(x)-\psi(y)\bigr).
\end{equation*}
Using the elementary inequality $|a+b|^p\leq 2^{p-1}\bigl(|a|^p+|b|^p\bigr)$, we obtain
\begin{align*}
[\psi u]_{W^{s,p,\alpha}(\Omega)}^p
\leq{}&
2^{p-1}
\int_\Omega\int_\Omega
\frac{|\psi(x)|^p|u(x)-u(y)|^p}
{|x-y|^{N+sp}}
|x|^{\alpha_{1} p}|y|^{\alpha_{2} p}
\,dx\,dy
\\
&+
2^{p-1}
\int_\Omega\int_\Omega
\frac{|u(y)|^p|\psi(x)-\psi(y)|^p}
{|x-y|^{N+sp}}
|x|^{\alpha_{1} p}|y|^{\alpha_{2} p}
\,dx\,dy.
\end{align*}
Since $0\leq\psi\leq1$, the first integral is bounded by $2^{p-1}[u]_{W^{s,p,\alpha}(\Omega)}^p$. Hence
\begin{equation*}
   [\psi u]_{W^{s,p,\alpha}(\Omega)}^p
\leq
2^{p-1}[u]_{W^{s,p,\alpha}(\Omega)}^p
+
2^{p-1}I, 
\end{equation*}
where
\begin{equation*}
   I:=
\int_\Omega\int_\Omega
\frac{|u(y)|^p|\psi(x)-\psi(y)|^p}
{|x-y|^{N+sp}}
|x|^{\alpha_{1} p}|y|^{\alpha_{2} p}
\,dx\,dy. 
\end{equation*}

Let $\Lambda$ denote the Lipschitz constant of $\psi$. Then
\begin{equation*}
    |\psi(x)-\psi(y)|
\leq
\Lambda |x-y|.
\end{equation*}
Since $\alpha_{1} \geq 0$ and $\Omega$ is an bounded open set, we have $|x|^{\alpha_{1} p} \leq \sup_{x \in \Omega} |x|^{\alpha_{1} p}=: R>0$. Therefore,
\begin{align*}
    I & \leq R \int_\Omega\int_\Omega
\frac{|u(y)|^p|\psi(x)-\psi(y)|^p}
{|x-y|^{N+sp}}
|y|^{\alpha_{2} p}
\,dx\,dy \\ & \leq R \Lambda^{p} \int_\Omega |u(y)|^p|y|^{\alpha_{2} p}
\left(
\int_{\Omega\cap\{|x-y|\leq1\}}
\frac{|x-y|^{p}}
{|x-y|^{N+sp}}
\,dx
\right)dy \\ & \quad + 2 R \int_\Omega |u(y)|^p|y|^{\alpha_{2} p}
\left(
\int_{\Omega\cap\{|x-y|>1\}}
\frac{1}
{|x-y|^{N+sp}}\,dx
\right)dy \\ & \leq C \int_\Omega |u(y)|^p|y|^{\alpha_{2} p} \, dy . 
\end{align*}
Therefore, we obtain
\begin{equation*}
     [\psi u]_{W^{s,p,\alpha}(\Omega)}^p
\leq
2^{p-1}[u]_{W^{s,p,\alpha}(\Omega)}^p
+ C \int_\Omega |u(y)|^p|y|^{\alpha_{2} p} \, dy . 
\end{equation*}
This proves the lemma.
\end{proof}

Finally, we now establish the weighted fractional extension theorem. Using the preceding lemmas, together with a suitable partition of unity and the boundary flattening technique, we show that every open set $\Omega$ with  Lipschitz boundary and $0 \notin \partial \Omega$ is an extension domain for the weighted fractional Sobolev space $W^{s,p,\alpha}$.

\begin{theorem}[Weighted fractional extension theorem]\label{Theorem: fractional extension theorem}
Let $p \geq 1$, $s\in(0,1)$, and let $\alpha_1, \, \alpha_2\in\mathbb{R}$ with  $\alpha=\alpha_1+\alpha_2 \geq 0$ satisfy $ \alpha_{1} p, \, \alpha_{2} p\in(-N,sp)$. Let $\Omega\subset\mathbb{R}^N$ be an  open set of class
$C^{0,1}$ with bounded boundary and assume that $0 \notin \partial \Omega$. Then there exists a constant
$C=C(N,p,s,\alpha_1,\alpha_2,\Omega)>0$ such that, for every
$u\in W^{s,p,\alpha}(\Omega)$,
there exists an extension $\widetilde u\in W^{s,p,\alpha}(\mathbb{R}^N)$ satisfying $\widetilde u=u$
a.e. in $\Omega$ and 
\begin{equation*}
    \|\widetilde u\|_{W^{s,p,\alpha}(\mathbb{R}^N)} \leq C  \|u\|_{W^{s,p,\alpha}(\Omega)}.
\end{equation*}
\end{theorem}

\begin{proof}
Let $\Omega\subset\mathbb{R}^N$ be an  open set of class
$C^{0,1}$ with bounded boundary and $0 \notin \partial \Omega$. Since $\partial\Omega$ is compact, there exist finitely many balls $B_1,\ldots,B_k$ such that
\begin{equation*}
  \partial\Omega\subset\bigcup_{j=1}^k B_j.  
\end{equation*}
We choose the balls sufficiently small so that, for each
$j\in\{1,\ldots,k\}$, $0 \notin B_{j}$, and there exists a bi-Lipschitz map $T_j:Q^{j}\longrightarrow B_j$ of class $C^{0,1}$ which flattens the boundary, namely $T_j(Q^{j}_+)=B_j\cap\Omega$, where $Q^{j}$ is chosen in such a way that $0 \notin Q^{j}$ (see Subsection \ref{Appendix: Subsection}, Appendix \ref{Appendix}), i.e., 
\begin{equation*}
  Q^{j} \subset (R^{j}_{1},R^{j}_{2})^{N-1} \times (-R^{j},R^{j}),
\quad
Q^{j}_+=Q^{j}\cap\{x_N>0\}, \quad  T^{-1}_{j}(\partial \Omega \cap B_{j}) = \{ x \in Q^{j} : x_{N} = 0 \},
\end{equation*}
where $R^{j}_{1}, \, R^{j}_{2}, \, R^{j} >0 $.
Therefore, by Subsection \ref{Appendix: Subsection} (see Appendix \ref{Appendix}), for any $x \in B_{j}$, we have
\begin{equation*}
    |T^{-1}_{j}(x)| \sim |x|.
\end{equation*}
We take a smooth partition of unity $\psi_0, \, \psi_1,  \ldots, \, \psi_k$ such that
\begin{equation*}
    0\leq\psi_j\leq1,
\qquad
\sum_{j=0}^k\psi_j=1,
\end{equation*}
with $\operatorname{supp} \psi_0 \subset  \mathbb{R}^{N} \setminus \partial \Omega$ and
\begin{equation*}
\operatorname{supp}\psi_j\subset \subset B_j,
\qquad \text{for }
j=1,\ldots,k.   
\end{equation*}
Hence
\begin{equation*}
    u=\sum_{j=0}^k\psi_j u
\qquad\text{in }\Omega.
\end{equation*}

We first consider the term $\psi_0u$. Since
$\operatorname{supp}\psi_0 \equiv 0$ in a neighbourhood of $\partial\Omega$. Therefore, we can
define
\begin{equation*}
   \widetilde{\psi_0u}(x)
:=
\begin{cases}
\psi_0(x)u(x), & x \in \Omega,\\
0, & x \in \mathbb{R}^N \setminus \Omega.
\end{cases} 
\end{equation*}
By the weighted version of Lemma \ref{Lemma: Extension 1} and Lemma \ref{Lemma: Extension 3},
\begin{equation}\label{ineq02}
\|\widetilde{\psi_0u}\|_{W^{s,p,\alpha}(\mathbb{R}^N)} 
\leq
C  \|\psi_0 u\|_{W^{s,p,\alpha}(\Omega)}  \leq 
C \|u\|_{W^{s,p,\alpha}(\Omega)} .
\end{equation}

We now consider $j\in\{1,\ldots,k\}$. Define
\begin{equation*}
    v_j(y):=u(T_j(y)),
\qquad y\in Q^{j}_+.
\end{equation*}
We claim that $v_j\in W^{s,p,\alpha}(Q^{j}_+)$. Indeed, since $T_j$ is bi-Lipschitz, there exists a constant
$C>0$ such that
\begin{equation*}
    C^{-1}|x-y|
\leq
|T_j^{-1}(x)-T_j^{-1}(y)|
\leq
C|x-y|.
\end{equation*}
By the change of variables $x=T_j(\xi)$ and $y=T_j(\eta)$, applying $|T^{-1}_{j}(x)| \sim |x|$ for all $x \in \Omega \cap B_{j}$, and using the bi-Lipschitz property of $T_j$, we obtain
\begin{align*}
 & [v_j]_{W^{s,p,\alpha}(Q^{j}_+)}^p
=\int_{Q^{j}_+}\int_{Q^{j}_+}
  \frac{|u(T_j(\xi))-u(T_j(\eta))|^p}
  {|\xi-\eta|^{N+sp}}
  |\xi|^{\alpha_1p}|\eta|^{\alpha_2p}
  \,d\xi\,d\eta
  \\
  & \qquad \leq
  C
  \int_{\Omega\cap B_j}\int_{\Omega\cap B_j}
  \frac{|u(x)-u(y)|^p}
  {|x-y|^{N+sp}}
  |x|^{\alpha_1p}|y|^{\alpha_2p}
  \left|\det DT_j^{-1}(x)\right|
  \left|\det DT_j^{-1}(y)\right|
  \,dx\,dy.
\end{align*}
Since $T_j^{-1}$ is Lipschitz, it is differentiable almost everywhere and
\begin{equation*}
     \left|\det DT_j^{-1}(x)\right|
  \leq
  \operatorname{Lip}(T_j^{-1})^N
  \qquad\text{for a.e. }x\in B_j.
\end{equation*}
Consequently,
\begin{equation*}
    \left|\det DT_j^{-1}(x)\right|
  \left|\det DT_j^{-1}(y)\right|
  \leq C_j
  \qquad\text{for a.e. }(x,y)\in B_j\times B_j.
\end{equation*}
Therefore,
\begin{align*}
  [v_j]_{W^{s,p,\alpha}(Q^{j}_+)}^p
  \leq C
 \int_{\Omega\cap B_j}\int_{\Omega\cap B_j}
  \frac{|u(x)-u(y)|^p}
  {|x-y|^{N+sp}}
|x|^{\alpha_1p}|y|^{\alpha_2p}
  \,dx\,dy
\end{align*}
Here and below, $C$ denotes a generic constant depending, in particular,
on the Lipschitz constants of the finitely many maps $T_j$ and
$T_j^{-1}$. We now apply the weighted reflection Lemma \ref{Lemma: Extension 2} to $v_j$. Define
\begin{equation*}
   \overline{v_j}(y',y_N)
:=
\begin{cases}
v_j(y',y_N),&y_N>0,\\
v_j(y',-y_N),&y_N<0.
\end{cases} 
\end{equation*}
Then $\overline{v_j}\in W^{s,p,\alpha}(Q^{j})$ and
\begin{equation*} 
\|\overline{v_j}\|_{W^{s,p,\alpha}(Q^{j})}
\leq
4 \|v_j\|_{W^{s,p,\alpha}(Q^{j}_+)} \leq C\|u\|_{W^{s,p,\alpha}(\Omega\cap B_j)}.
\end{equation*}

Define
\[
w_j(x)
:=
\overline{v_j}(T_j^{-1}(x)),
\qquad x\in B_j.
\]
Again, by the bi-Lipschitz property of $T_j$ and the corresponding
weighted change-of-variables estimate, $w_j\in W^{s,p,\alpha}(B_j)$
and
\begin{equation*}
 \|w_j\|_{W^{s,p,\alpha}(B_j)}
\leq
C\|\overline{v_j}\|_{W^{s,p,\alpha}(Q^{j})} \leq
C\|u\|_{W^{s,p,\alpha}(\Omega\cap B_j)}.   
\end{equation*}
Observe that $w_j=u$ in $B_j\cap\Omega$. Therefore,
\begin{equation*}
    \psi_jw_j=\psi_ju
\qquad\text{in }B_j\cap\Omega.
\end{equation*}
Since $\psi_j$ has compact support in $B_j$, Lemma \ref{Lemma: Extension 3} gives
\begin{equation*}
  \|\psi_j w_j\|_{W^{s,p,\alpha}(B_j)}
\leq
C\| w_j\|_{W^{s,p,\alpha}(B_j)} \leq C\|u\|_{W^{s,p,\alpha}(\Omega\cap B_j)}.  
\end{equation*}

Since $\psi_jw_j$ has compact support in $B_j$, we may extend it
by zero outside $B_j$. By the weighted version of Lemma \ref{Lemma: Extension 1}, this
extension, denoted by $\widetilde{\psi_jw_j}$, belongs to $W^{s,p,\alpha}(\mathbb{R}^N)$ and satisfies
\begin{equation*} 
\|\widetilde{\psi_jw_j}\|_{W^{s,p,\alpha}(\mathbb{R}^N)}
\leq
C \|\psi_jw_j\|_{W^{s,p,\alpha}(B_j)} . 
\end{equation*}
Combining the above two inequalities, we obtain
\begin{equation}\label{ineq03}
\|\widetilde{\psi_jw_j}\|_{W^{s,p,\alpha}(\mathbb{R}^N)} \leq C \|u\|_{W^{s,p,\alpha}(\Omega \cap B_j)}.    
\end{equation}

Finally, define
\[
\widetilde u
:=
\widetilde{\psi_0u}
+
\sum_{j=1}^k\widetilde{\psi_jw_j}.
\]
By construction,
\[
\widetilde u=u
\qquad\text{a.e. in }\Omega.
\]
Furthermore, using the triangle inequality and combining \eqref{ineq02} and \eqref{ineq03}, we obtain
\begin{equation*}
   \|\widetilde u\|_{W^{s,p,\alpha}(\mathbb{R}^N)} \leq  C \|u\|_{W^{s,p,\alpha}(\Omega)}.
\end{equation*}
Here, the constant $C$ may initially depend on the bi-Lipschitz maps $T_j$ and the sets $B_j$, where $j=0,1,\ldots,k$, in addition to $N,p,s,\alpha_1,\alpha_2,$ and $\Omega$. Since the collection $\{T_j,B_j\}_{j=0}^k$ is finite, we may choose $C$ to be the maximum of the corresponding constants. Thus, $C$ can be chosen uniformly so that it depends only on  $N,p,s,\alpha_1,\alpha_2$, and $\Omega$. and is independent of the particular maps $T_j$ and sets $B_j$, where $j=0,1,\ldots,k$. This proves the theorem.
\end{proof}

\section{Weighted fractional Sobolev inequalities}\label{Section : 6}

In this section, we establish a weighted fractional Sobolev inequality on bounded Lipschitz domains. Our approach is based on first deriving the corresponding weighted fractional Sobolev inequality on $\mathbb{R}^{N}$. This inequality follows from the fractional Caffarelli--Kohn--Nirenberg inequality established by Nguyen and Squassina \cite{Squassina2018}, by choosing suitable values of the parameters. We then employ the extension theorem proved in the previous section to transfer the inequality from $\mathbb{R}^{N}$ to bounded Lipschitz domains. In addition, we establish a Poincar\'e-type inequality in the corresponding weighted fractional Sobolev space.

\smallskip

Nguyen and Squassina in \cite{Squassina2018} studied Caffarelli–Kohn–Nirenberg type inequalities in fractional Sobolev spaces and proved the following theorem. Let $N \geq 1$, $p>1$, $q \geq 1$, $\tau >0$, $a \in (0,1]$, and $\alpha, \beta, \gamma \in \mathbb{R}$ satisfy
\begin{equation*}
    \frac{1}{\tau}+ \frac{\gamma}{N} = a \left( \frac{1}{p} + \frac{\alpha-s}{N} \right) + (1-a) \left( \frac{1}{q} + \frac{\beta}{N} \right)  \ .
\end{equation*}
If $a>0$, assume also that, with $\gamma = a \sigma + (1-a) \beta$,
\begin{equation*}
    0 \leq \alpha - \sigma
\end{equation*}
and 
\begin{equation*}
    \alpha - \sigma \leq s \hspace{.5cm} \text{if} \hspace{.5cm} \frac{1}{\tau}+ \frac{\gamma}{N} = \frac{1}{p}+ \frac{\alpha-s}{N}  .
\end{equation*}
If $\frac{1}{\tau} + \frac{\gamma}{N} >0$, then we have
\begin{equation}\label{fractional CKN ineq}
          \| |x|^{\gamma}u \|_{L^{\tau}(\mathbb{R}^N)} \leq C [u]^{a}_{W^{s,p, \alpha}(\mathbb{R}^N)} \||x|^{\beta} u \|^{(1-a)}_{L^{q} (\mathbb{R}^N)},  \hspace{.3cm}  \forall \ u  \in C^{1}_{c}(\mathbb{R}^N)  .
      \end{equation}
\smallskip

Now, if we take $\gamma=0$ and $a=1$, the fractional Caffarelli--Kohn--Nirenberg inequality \eqref{fractional CKN ineq} reduces to  
\begin{equation}  \label{frac-sobolev}
\left( \int_{\mathbb{R}^{N}} |u(x)|^{\tau} \, dx \right)^{\!\frac{1}{\tau}}
 \leq C \left( \int_{\mathbb{R}^{N}} \int_{\mathbb{R}^{N}} 
 \frac{ |u(x)-u(y)|^{p}}{|x-y|^{N+sp}} 
 |x|^{\alpha_{1} p} |y|^{\alpha_{2} p}  \, dx \, dy \right)^{\!\frac{1}{p}},
\end{equation}
under the conditions
\begin{equation*}
 0 \leq \alpha < s, \qquad sp- \alpha p < N, \qquad \text{and} \qquad 
   \tau = \frac{Np}{N-sp+ \alpha p}.  
\end{equation*}
The inequality \eqref{frac-sobolev} is nothing but the fractional Sobolev inequality in weighted spaces, with the critical exponent
\begin{equation*}
    p^{*}_{s, \alpha} := \tau = \frac{Np}{N-sp+ \alpha p}, \qquad \text{if} \quad 0< sp- \alpha p < N.
\end{equation*}
We can therefore state the following fundamental result.

\begin{lemma}[Weighted fractional Sobolev inequality]\label{Lemma: Weighted fractional Sobolev inequality}
    Let $p>1$, $s \in (0,1)$, and let $\alpha_{1}, \, \alpha_{2} \in \mathbb{R}$ with $\alpha= \alpha_{1}+ \alpha_{2} \geq 0$ satisfy  $\alpha_{1}p$, $\alpha_{2}p \in (-N, sp)$. Let $0<sp-\alpha p < N$
    and $p^{*}_{s, \alpha}= \frac{Np}{N-sp+\alpha p}$. Then for all $u \in W^{s,p, \alpha}(\mathbb{R}^{N})$, there exists a constant $C=C(N,p,s,\alpha_{1}, \alpha_{2})>0$ such that
    \begin{align}
        \left( \int_{\mathbb{R}^{N}} |u(x)|^{p^{*}_{s, \alpha}} \, dx \right)^{\!\frac{1}{p^{*}_{s, \alpha}}}
        \leq C \left( \int_{\mathbb{R}^{N}} \int_{\mathbb{R}^{N}}
        \frac{ |u(x)-u(y)|^{p}}{|x-y|^{N+sp}} 
        |x|^{\alpha_{1} p} |y|^{\alpha_{2} p}   \, dx \, dy \right)^{\!\frac{1}{p}}   .
    \end{align}
\end{lemma}

The next theorem establishes the weighted fractional Sobolev inequality on bounded Lipschitz domains. It follows by combining the weighted fractional extension theorem with the weighted fractional Sobolev inequality on $\mathbb{R}^{N}$.

\begin{theorem}\label{Theorem: Weighted frac Sobolev inequality}
    Let $p>1$, $s \in (0,1)$, and let $\alpha_{1}, \, \alpha_{2} \in \mathbb{R}$ with $\alpha= \alpha_{1}+ \alpha_{2} \geq 0$ satisfy  $\alpha_{1}p$, $\alpha_{2}p \in (-N, sp)$. Let $\Omega$ be a bounded Lipschitz domain and assume that $0 \notin \partial \Omega$, and let $0<sp-\alpha p < N$
    and $p^{*}_{s, \alpha}= \frac{Np}{N-sp+\alpha p}$. Then there exists a constant $C=C(N,p,s,\alpha_{1}, \alpha_{2}, \Omega)>0$ such that
    \begin{equation}
  \|u\|_{L^{p^{*}_{s, \alpha}}(\Omega)}
        \leq C \|u\|_{W^{s,p, \alpha}(\Omega)}, \quad \forall \, u \in W^{s,p, \alpha}(\Omega) .
    \end{equation}
\end{theorem}
\begin{proof}
Let $\Omega\subset\mathbb{R}^N$ be a bounded Lipschitz domain and let $u\in W^{s,p,\alpha}(\Omega)$. By Theorem \ref{Theorem: fractional extension theorem}, there exists an extension $\widetilde{u}\in W^{s,p,\alpha}(\mathbb{R}^N)$ such that 
\begin{equation*}
        \widetilde{u}=u \qquad\text{a.e. in }\Omega,
    \end{equation*}
and
\begin{equation*}
        \| \widetilde{u} \|_{W^{s,p, \alpha}(\mathbb{R}^{N})} \leq C  \| u \|_{W^{s,p, \alpha}(\Omega)}, 
\end{equation*}
where $C=C(N,p,s,\alpha_{1}, \alpha_{2}, \Omega)>0$.  On the other hand, by Lemma \ref{Lemma: Weighted fractional Sobolev inequality}, there exists a constant  $C=C(N,p,s, \alpha_{1}, \alpha_{2})>0$ such that
    \begin{equation*}
    \|\widetilde{u}\|_{L^{p^{*}_{s, \alpha}}(\mathbb{R}^{N})}  \leq  C  [ \widetilde{u} ]_{W^{s,p, \alpha}(\mathbb{R}^{N})},
    \end{equation*}
    where $p^{*}_{s, \alpha} = \frac{Np}{N-sp+\alpha p}$. Since  $[\widetilde{u}]_{W^{s,p,\alpha}(\mathbb{R}^N)} \leq \|\widetilde{u}\|_{W^{s,p,\alpha}(\mathbb{R}^N)}$, we have \begin{equation*} \|\widetilde{u}\|_{L^{p^{*}_{s, \alpha}}(\mathbb{R}^N)} \leq C \|\widetilde{u}\|_{W^{s,p,\alpha}(\mathbb{R}^N)}. \end{equation*}
Combining this with the extension estimate yields 
\begin{equation*} \|\widetilde{u}\|_{L^{p^{*}_{s, \alpha}}(\mathbb{R}^N)} \leq C \|u\|_{W^{s,p,\alpha}(\Omega)}. \end{equation*} 
Finally, using $\|u\|_{L^{p^{*}_{s, \alpha}}(\Omega)} \leq \|\widetilde{u}\|_{L^{p^{*}_{s, \alpha}}(\mathbb{R}^N)}$, we have
    \begin{equation*}
        \|u\|_{L^{p^{*}_{s, \alpha}}(\Omega)}  \leq  C \| u \|_{W^{s,p, \alpha}(\Omega)}.
    \end{equation*}
 This completes the proof.
\end{proof}

The next lemma proves the weighted fractional Poincar\'e inequality, which bounds the weighted $L^{p, \alpha_{i}}$-norm of $u-(u)_{\Omega}$ by the weighted fractional Gagliardo seminorm of $u$. Recall $(u)_{\Omega}$ denotes the average of $u$ over $\Omega$, i.e.,
\begin{equation*}
    (u)_{\Omega} = \frac{1}{|\Omega|} \int_{\Omega} u(x) \, dx = \fint_{\Omega} u(x) \, dx,
\end{equation*}
where $|\Omega|$ is the Lebesgue measure of $\Omega$.

\begin{lemma}[Weighted fractional Poincar\'e inequality]\label{Lemma: Weighted frac Poincare}
Let $p>1$, $s\in(0,1)$, and let $\alpha_1,\alpha_2\in\mathbb{R}$ with $\alpha:=\alpha_1+\alpha_2$ satisfy $\alpha_{1} p,\alpha_{2} p\in(-N,sp)$ and $\alpha_{1}p + \alpha_{2}p >-N$. Let $\Omega$ be a bounded open set with $S_{1} = \inf_{x \in \Omega} |x|>0$ and $S_{2} = \sup_{x \in \Omega} |x|<\infty$. Then for all $u \in W^{s,p, \alpha}(\Omega)$,
 \begin{equation}
       \int_{\Omega}  |u(x) - (u)_{\Omega}|^{p} |x|^{\alpha_{i}p}  \, dx
        \leq \frac{(\operatorname{diam}\Omega)^{N+sp}}
{|\Omega|\min \left\{S^{\alpha_{j}p}_{1},S^{\alpha_{j}p}_{2} \right\}} \int_{\Omega}\int_{\Omega}
\frac{|u(x)-u(y)|^{p}}
{|x-y|^{N+sp}}
|x|^{\alpha_{1}p}|y|^{\alpha_{2}p}\,dy\,dx,
    \end{equation}
    where $i,j \in \{1,2\}$ and $i \neq j $.
\end{lemma}
\begin{proof}
First assume $i=1$. By Jensen's inequality, we have
\begin{align*}
\int_{\Omega}|u(x)-(u)_{\Omega}|^{p}|x|^{\alpha_{1}p}\,dx
&=
\int_{\Omega}
\left|
\frac{1}{|\Omega|}
\int_{\Omega}\big(u(x)-u(y)\big)\,dy
\right|^{p}
|x|^{\alpha_{1}p}\,dx \\
&\leq
\frac{1}{|\Omega|}
\int_{\Omega}\int_{\Omega}
|u(x)-u(y)|^{p}|x|^{\alpha_{1}p}\,dy\,dx.
\end{align*}
Since $S_{1}<|y|<S_{2}$ for every $y\in\Omega$, we have $|y|^{\alpha_{2}p}\geq \min \left\{S^{\alpha_{2}p}_{1},S^{\alpha_{2}p}_{2} \right\}>0$. Moreover, $|x-y|\leq\operatorname{diam} \Omega$ for all $x,y\in\Omega$. Therefore,
\begin{align*}
\int_{\Omega}|u(x)-(u)_{\Omega}|^{p}|x|^{\alpha_{1}p}\,dx
\leq
\frac{(\operatorname{diam}\Omega)^{N+sp}}
{|\Omega|\min \left\{S^{\alpha_{2}p}_{1},S^{\alpha_{2}p}_{2}\right\}}
\int_{\Omega}\int_{\Omega}
\frac{|u(x)-u(y)|^{p}}
{|x-y|^{N+sp}}
|x|^{\alpha_{1}p}|y|^{\alpha_{2}p}\,dy\,dx.
\end{align*}
Similarly, for $i=2$, the desired inequality follows by interchanging $\alpha_{1}$ and $\alpha_{2}$. This completes the proof.
\end{proof}

The next lemma establishes a scale-invariant weighted fractional Poincar\'e--Sobolev inequality on annuli. In particular, it gives a uniform estimate on $\Omega_\lambda$ for every $\lambda>0$ by exploiting the scaling properties of the weighted fractional seminorm.

\begin{lemma}\label{Lemma: Weighted frac Sobolev ineq Lambda}
Let $p>1$, $s\in(0,1)$, and let $\alpha_1,\alpha_2\in\mathbb{R}$ satisfy $\alpha:=\alpha_1+\alpha_2\geq0$, $\alpha_{1} p,\alpha_{2} p\in(-N,sp)$ and $0<sp-\alpha p<N$. Let $\lambda>0$ and set $\Omega_{\lambda}:=\{x\in\mathbb{R}^N:\lambda<|x|<n \lambda\}$, where $n \in \mathbb{N}$ and $n>1$, 
and define $p^*_{s,\alpha}:=\frac{Np}{N-sp+\alpha p}$.
Then there exists a constant $C=C(N,p,s,\alpha_1,\alpha_2,n)>0$ such that, for every $u\in W^{s,p,\alpha}(\Omega_\lambda)$,
\begin{align*}
 & \left( \frac{1}{|\Omega_{\lambda}|} \int_{\Omega_{\lambda}}  |u(x)-(u)_{\Omega_\lambda}|^{p^*_{s,\alpha}} \, dx \right)^{\frac{1}{p^{*}_{s, \alpha}}} \\ & \hspace{2cm}
\leq
C\left(
\lambda^{sp-\alpha p-N}
\int_{\Omega_\lambda}\int_{\Omega_\lambda}
\frac{|u(x)-u(y)|^p}
{|x-y|^{N+sp}}
|x|^{\alpha_{1} p}|y|^{\alpha_{2} p}
\,dx\,dy
\right)^{\frac1p}.  
\end{align*}
\end{lemma}

\begin{proof}
  First assume $\lambda=1$. By Theorem \ref{Theorem: Weighted frac Sobolev inequality}, we have
  \begin{equation*}
       \|u\|_{L^{p^{*}_{s, \alpha}}(\Omega_{1})}
        \leq C \|u\|_{W^{s,p, \alpha}(\Omega_{1})},
  \end{equation*}
  where $C=C(N,p,s, \alpha_{1}, \alpha_{2},n)>0$. Applying the above inequality with $u-(u)_{\Omega_{1}}$ and using Lemma \ref{Lemma: Weighted frac Poincare} with $\Omega=\Omega_{1}$ and $S_{1}=1$ and $S_{2}=n$, we obtain
  \begin{equation*}
    \left(  \int_{\Omega_{1}} |u(x)-(u)_{\Omega_{1}}|^{p^{*}_{s, \alpha}} \, dx \right)^{\frac{1}{p^{*}_{s, \alpha}}} \leq C \left( \int_{\Omega}\int_{\Omega}
\frac{|u(x)-u(y)|^{p}}
{|x-y|^{N+sp}}
|x|^{\alpha_{1}p}|y|^{\alpha_{2}p}\,dy\,dx\right)^{\frac{1}{p}}. 
  \end{equation*}
Dividing both sides by $|\Omega_1|^{1/p^*_{s,\alpha}}$ and absorbing
the fixed factor into the constant, we obtain
 \begin{equation*}
    \left( \frac{1}{|\Omega_{1}|}  \int_{\Omega_{1}} |u(x)-(u)_{\Omega_{1}}|^{p^{*}_{s, \alpha}} \, dx \right)^{\frac{1}{p^{*}_{s, \alpha}}} \leq C \left( \int_{\Omega}\int_{\Omega}
\frac{|u(x)-u(y)|^{p}}
{|x-y|^{N+sp}}
|x|^{\alpha_{1}p}|y|^{\alpha_{2}p}\,dy\,dx\right)^{\frac{1}{p}}. 
  \end{equation*}
Now, let $\lambda>0$ and apply the above inequality to
$v(x)=u(\lambda x)$. Since
\begin{equation*}
    \frac{1}{|\Omega_1|}
\int_{\Omega_1}u(\lambda y)\,dy
=
\frac{1}{|\Omega_\lambda|}
\int_{\Omega_\lambda}u(y)\,dy
=
(u)_{\Omega_\lambda},
\end{equation*}
we obtain
\begin{equation*}
    \left( \frac{1}{|\Omega_{1}|}  \int_{\Omega_{1}} |u(\lambda x)-(u)_{\Omega_{\lambda}}|^{p^{*}_{s, \alpha}} \, dx \right)^{\frac{1}{p^{*}_{s, \alpha}}} \leq C \left( \int_{\Omega}\int_{\Omega}
\frac{|u(\lambda x)-u(\lambda y)|^{p}}
{|x-y|^{N+sp}}
|x|^{\alpha_{1}p}|y|^{\alpha_{2}p}\,dy\,dx\right)^{\frac{1}{p}}. 
  \end{equation*}
Using the changes of variables $X=\lambda x$ and $Y=\lambda y$ on both sides, we have
  \begin{align*}
 & \left( \frac{1}{|\Omega_{\lambda}|} \int_{\Omega_{\lambda}}  |u(x)-(u)_{\Omega_\lambda}|^{p^*_{s,\alpha}} \, dx \right)^{\frac{1}{p^{*}_{s, \alpha}}} \\ & \hspace{2cm}
\leq
C\left(
\lambda^{sp-\alpha p-N}
\int_{\Omega_\lambda}\int_{\Omega_\lambda}
\frac{|u(x)-u(y)|^p}
{|x-y|^{N+sp}}
|x|^{\alpha_{1} p}|y|^{\alpha_{2} p}
\,dx\,dy
\right)^{\frac1p}.  
\end{align*}
This completes
the proof.
\end{proof}

\smallskip

The following lemma provides a basic estimate for the weighted Gagliardo seminorm that will be useful in the proof of our main results. Its proof can be found in \cite[Lemma $3.2$]{BanerjeeGangulySahu}.

\begin{lemma}\label{Lemma: Positive and Negative part}
Let $u:\mathbb{R}^{N} \to \mathbb{R}$ be a function and define
\begin{equation*}
    u_{+}(x)=\max\{u(x),0\}, \qquad u_{-}(x)=\max\{-u(x),0\}.
\end{equation*}
Then for $0<s<1$, $1 \leq p < \infty$ and any function $\mathcal{K}$ on $\mathbb{R}^{N} \times \mathbb{R}^{N}$, we have
\begin{align*}
    \int_{\mathbb{R}^{N}} \int_{\mathbb{R}^{N}}  \frac{|u_{+}(x) - u_{+}(y)|^{p}}{|x-y|^{N+sp}} & \mathcal{K}(x,y) \, dx \, dy  +  \int_{\mathbb{R}^{N}} \int_{\mathbb{R}^{N}} \frac{|u_{-}(x) - u_{-}(y)|^{p}}{|x-y|^{N+sp}} \mathcal{K}(x,y) \, dx \, dy \\ & \leq \int_{\mathbb{R}^{N}} \int_{\mathbb{R}^{N}} \frac{|u(x) - u(y)|^{p}}{|x-y|^{N+sp}} \mathcal{K}(x,y) \, dx \, dy.
\end{align*}
\end{lemma}

\section{Proof of quantitative stability for weighted fractional Hardy inequality}\label{Section : 7}

In this section, we prove Theorem \ref{Theorem: Quantitative frac hardy p geq 2} and Theorem \ref{Theorem: Quantitative frac hardy 1 < p < 2}
and establish quantitative stability for the weighted fractional Hardy
inequality. We first consider the case $p\geq2$ and subsequently treat the
range $1<p<2$. The two regimes require different arguments, owing to the
distinct structure of the nonlinear remainder estimates available in each
case.

\subsection{Proof of Theorem \ref{Theorem: Quantitative frac hardy p geq 2}} For $p\geq2$, Dyda and Kijaczko \cite[Theorem $1.6$]{dyda2024} established a remainder estimate for the
weighted fractional Hardy inequality under the assumption
\begin{equation*}
 \alpha=\alpha_1+\alpha_2,
\qquad
\alpha_1p, \, \alpha_2p, \, \alpha p\in(-N,sp),
\end{equation*}
and $sp-\alpha p <N$. More precisely, they proved that
\begin{align}\label{Weighted Frac Hardy for p geq 2 with remainder}
& \int_{\mathbb{R}^{N}}\int_{\mathbb{R}^{N}} \frac{|u(x)-u(y)|^{p}}{|x-y|^{N+sp}} |x|^{\alpha_{1} p} |y|^{\alpha_{2} p} \, dx \, dy  - \mathcal{C}  \int_{\mathbb{R}^{N}} \frac{|u(x)|^{p}}{|x|^{sp-\alpha p}} \, dx \nonumber \\  & \quad \quad \geq c_{p}  \int_{\mathbb{R}^{N}}\int_{\mathbb{R}^{N}} \frac{|v(x)-v(y)|^{p}}{|x-y|^{N+sp}} |x|^{- \frac{(N-\alpha_{1} p + \alpha_{2} p-sp)}{2}} |y|^{- \frac{(N+\alpha_{1} p - \alpha_{2} p-sp)}{2}} \, dx \, dy, \quad \forall \, u \in C^{1}_{c}(\mathbb{R}^{N}),
\end{align}
where $v(x) = u(x) \omega^{-1}_{s, \alpha}(x) = u(x) |x|^{\frac{N-sp+\alpha p}{p}}$ and $0<c_{p}\leq 1$ is given by
\begin{equation*}
    c_{p}= \min_{0< \tau < \frac{1}{2}} \left( (1-\tau)^{p} -\tau^{p}+ p \tau^{p-1} \right).
\end{equation*}

Throughout the proof, we assume that $\alpha_{1}p,\,\alpha_{2}p \in (-N,sp)$, $\alpha = \alpha_{1}+ \alpha_{2} \geq 0$ and $0<sp-\alpha p<N$. These assumptions are justified by the fact that the weighted fractional Hardy inequality with a remainder \eqref{Weighted Frac Hardy for p geq 2 with remainder} term holds whenever $\alpha_{1}p,\,\alpha_{2}p,\,\alpha p \in (-N,sp)$, and $sp-\alpha p <N$.

\smallskip

We first consider the case of nonnegative functions. Let
$u\in C^{1}_c(\mathbb{R}^N)$ with $u\geq0$. By homogeneity, we may assume that
\begin{equation*}
    \| u\|_{L^{p^{*}_{s, \alpha},p}(\mathbb{R}^{N})} = 1.
\end{equation*}
Recall the remainder term
\begin{equation*}
    \mathcal{R}(u):= \int_{\mathbb{R}^{N}}\int_{\mathbb{R}^{N}} \frac{|v(x)-v(y)|^{p}}{|x-y|^{N+sp}} |x|^{- \frac{(N-\alpha_{1} p + \alpha_{2} p-sp)}{2}} |y|^{- \frac{(N+\alpha_{1} p - \alpha_{2} p-sp)}{2}} \, dx \, dy.
\end{equation*}
Then, using the definition of weighted fractional Hardy deficit $\delta_{s,p,\alpha}(u)$ given in \eqref{Weighted Frac Hardy deficit} and the weighted fractional Hardy inequality with a remainder for $p \geq 2$ stated above, we have
\begin{equation}\label{ineq113}
\delta_{s,p,\alpha}(u) \geq \, c_{p} \mathcal{R}(u).
\end{equation}
Therefore, it is sufficient to estimate the remainder term $\mathcal{R}(u)$.

We first deal with the case $\mathcal{R}(u)>1$. Since  $L^{p^{*}_{s, \alpha},p}(\mathbb{R}^{N}) \hookrightarrow L^{p^{*}_{s, \alpha},\infty}(\mathbb{R}^{N})$ (see Proposition \ref{Proposition on Lorentz space}), the normalization above yields
\begin{equation*}
   \mathcal{R}(u)^{\frac{1}{2p}} >1 =  \| u \|_{L^{p^{*}_{s, \alpha},p}(\mathbb{R}^{N})} \geq C \| u \|_{L^{p^{*}_{s, \alpha},\infty} (\mathbb{R}^{N})} \geq C \inf_{a \geq 0}  \| u - a \, \omega_{s, \alpha} \|_{L^{p^{*}_{s, \alpha},\infty} (\mathbb{R}^{N})},
\end{equation*}
where $\omega_{s, \alpha}(x) =  \, |x|^{-\frac{N-sp+\alpha p}{p}}$ and $C= C(N,p,s, \alpha_{1}, \alpha_{2})>0$. Consequently,
\begin{equation}\label{ineqn100}
  \inf_{a \geq 0}  \| u - a \, \omega_{s, \alpha} \|^{2p}_{L^{p^{*}_{s, \alpha},\infty} (\mathbb{R}^{N})} \leq C\,  \mathcal{R}(u).
\end{equation}
It remains to consider the case $\mathcal{R}(u) \leq 1$. Set
\begin{equation*}
    \mathcal{A} = \{ x \in \mathbb{R}^{N} : v(x) > \mathcal{R}(u)^{\frac{p^{*}_{s, \alpha}}{2 p^{2}}} \}.
\end{equation*}
Since
$u\in C_c(\mathbb{R}^{N})$,
the set $\mathcal{A}$ is open and bounded. 

For each $\ell \in \mathbb{Z}$, consider the dyadic annuli defined in \eqref{Defn: A_{ell}}. Moreover, $\lim_{|x| \to 0} v(x) = \lim_{|x| \to 0} u(x)|x|^{\frac{N-sp+ \alpha p}{p}} = 0$. Hence, there exists a smallest $\ell_0\in\mathbb{Z}$ such that
\begin{equation*}
    v(x)\leq
    \mathcal{R}(u)^{\frac{p^{*}_{s,\alpha}}{2p^{2}}}
    \qquad\text{for every }x\in\mathcal{A}_{\ell_0},
\end{equation*}
while $\mathcal{A}\cap\mathcal{A}_{\ell_0+1}\neq\phi$. In particular, $ (v)_{\mathcal{A}_{\ell_0}} \leq \mathcal{R} (u)^{\frac{p^{*}_{s,\alpha}}{2p^{2}}}$, where $(v)_{\mathcal{A}_{\ell_0}}$ denotes the average of $v$ over $\mathcal{A}_{\ell_0}n $. Thus, for every $x\in\mathcal{A}$,
\begin{equation*}
    0<
    v(x)-\mathcal{R}(u)^{\frac{p^{*}_{s,\alpha}}{2p^{2}}}
    \leq
    v(x)-(v)_{\mathcal{A}_{\ell_0}}. 
\end{equation*}
It follows that
\begin{align}\label{Ineq01}
    \int_{\mathcal{A}} |v(x) - \mathcal{R}(u)^{\frac{p^{*}_{s, \alpha}}{2 p^{2}}}|^{p^{*}_{s, \alpha}} \, \frac{dx}{|x|^{N}} & \leq \int_{\mathcal{A}} |v(x) - (v)_{\mathcal{A}_{\ell_{0}}}|^{p^{*}_{s, \alpha}} \, \frac{dx}{|x|^{N}} \nonumber \\ 
    & = \sum_{\substack{\ell \in \mathbb{Z} \\ \mathcal{A} \cap \mathcal{A}_{\ell} \neq \phi}} \int_{\mathcal{A} \cap \mathcal{A}_{\ell}} |v(x) - (v)_{\mathcal{A}_{\ell_{0}}}|^{p^{*}_{s, \alpha}} \, \frac{dx}{|x|^{N}}. 
\end{align}
For each $\ell$ satisfying
$\mathcal{A}\cap\mathcal{A}_{\ell}\neq\phi$, the elementary inequality
$|a+b|^q\leq2^{q-1}(|a|^q+|b|^q)$ gives
\begin{align}\label{ineq986}
  \int_{\mathcal{A} \cap \mathcal{A}_{\ell}} |v(x) - (v)_{\mathcal{A}_{\ell_{0}}}|^{p^{*}_{s, \alpha}} \, \frac{dx}{|x|^{N}} & \leq 2^{p^{*}_{s, \alpha}-1} \int_{\mathcal{A} \cap \mathcal{A}_{\ell}} |v(x) - (v)_{A_{\ell}}|^{p^{*}_{s, \alpha}} \, \frac{dx}{|x|^{N}} \nonumber \\ & \quad + 2^{p^{*}_{s, \alpha}-1} |(v)_{\mathcal{A}_{\ell}} - (v)_{\mathcal{A}_{\ell_{0}}}|^{p^{*}_{s, \alpha}} \int_{\mathcal{A} \cap \mathcal{A}_{\ell}}  \, \frac{dx}{|x|^{N}} \nonumber \\ & \leq 2^{p^{*}_{s, \alpha}-1} \int_{\mathcal{A}_{\ell}} |v(x) - (v)_{\mathcal{A}_{\ell}}|^{p^{*}_{s, \alpha}} \, \frac{dx}{|x|^{N}} \nonumber \\ & \quad + 2^{p^{*}_{s, \alpha}-1} |(v)_{\mathcal{A}_{\ell}} - (v)_{\mathcal{A}_{\ell_{0}}}|^{p^{*}_{s, \alpha}} \int_{\mathcal{A} \cap \mathcal{A}_{\ell}}  \, \frac{dx}{|x|^{N}}  .
\end{align}
We estimate the first term on the right-hand side by means of
Lemma \ref{Lemma: Weighted frac Sobolev ineq Lambda}. Taking $n=2$ and
$\lambda=2^\ell$, we obtain
\begin{align*}
    \int_{\mathcal{A}_{\ell}} |v(x) - (v)_{\mathcal{A}_{\ell}}|^{p^{*}_{s, \alpha}} \, \frac{dx}{|x|^{N}} & \leq C \frac{1}{|\mathcal{A}_{\ell}|} \int_{\mathcal{A}_{\ell}} |v(x) - (v)_{\mathcal{A}_{\ell}}|^{p^{*}_{s, \alpha}} \, dx \\ & \leq C \left( 2^{\ell(sp-\alpha p - N)} \int_{\mathcal{A}_{\ell}} \int_{\mathcal{A}_{\ell}} \frac{|v(x)-v(y)|^{p}}{|x-y|^{N+sp}} |x|^{\alpha_{1}p} |y|^{\alpha_{2} p} \, dx \, dy  \right)^{\frac{p^{*}_{s, \alpha}}{p}}.
\end{align*}
The constant $C=C(N,s,p,\alpha_1,\alpha_2)$ is independent of $\ell$. For all $x,y \in \mathcal{A}_\ell$, we have $|x| \leq 2^{\ell+1}$, $|y| \leq 2^{\ell+1}$, and hence
\begin{equation*}
    |x|^{\frac{N-sp+ \alpha p}{2}} |y|^{\frac{N-sp+ \alpha p}{2}} \leq 2^{(\ell+1)(N-sp+ \alpha p)}.
\end{equation*}
Therefore, 
\begin{align*}
     \int_{\mathcal{A}_{\ell}} & |v(x) - (v)_{\mathcal{A}_{\ell}}|^{p^{*}_{s, \alpha}} \, \frac{dx}{|x|^{N}} \\ & \leq C \left( \int_{\mathcal{A}_{\ell}} \int_{\mathcal{A}_{\ell}} \frac{|v(x)-v(y)|^{p}}{|x-y|^{N+sp}} \, |x|^{- \frac{(N-\alpha_{1} p + \alpha_{2} p-sp)}{2}} |y|^{- \frac{(N+\alpha_{1} p - \alpha_{2} p-sp)}{2}} \, dx \, dy \right)^{\frac{p^{*}_{s, \alpha}}{p}} ,
\end{align*}
where $C= C(N,p,s, \alpha_{1}, \alpha_{2})>0$. Consequently, \eqref{ineq986} yields
\begin{align}\label{ineq04}
    \int_{\mathcal{A} \cap \mathcal{A}_{\ell}} & |v(x) - (v)_{\mathcal{A}_{\ell_{0}}}|^{p^{*}_{s, \alpha}} \, \frac{dx}{|x|^{N}} \nonumber \\ & \leq C \left( \int_{\mathcal{A}_{\ell}} \int_{\mathcal{A}_{\ell}} \frac{|v(x)-v(y)|^{p}}{|x-y|^{N+sp}} \, |x|^{- \frac{(N-\alpha_{1} p + \alpha_{2} p-sp)}{2}} |y|^{- \frac{(N+\alpha_{1} p - \alpha_{2} p-sp)}{2}} \, dx \, dy \right)^{\frac{p^{*}_{s, \alpha}}{p}} \nonumber \\ & \quad + 2^{p^{*}_{s, \alpha}-1}|(v)_{\mathcal{A}_{\ell}} - (v)_{\mathcal{A}_{\ell_{0}}}|^{p^{*}_{s, \alpha}}  \int_{\mathcal{A} \cap \mathcal{A}_{\ell}}  \, \frac{dx}{|x|^{N}}  .
\end{align}

We next control the difference of the averages appearing in
\eqref{ineq04}. Applying Lemma \ref{Lemma: on two disjoint set} with
$E=\mathcal{A}_k$ and $F=\mathcal{A}_{k+1}$, for
$\ell_0\leq k\leq\ell-1$, gives
\begin{align*}
  |(v)_{\mathcal{A}_{\ell}} - (v)_{\mathcal{A}_{\ell_{0}}}|^{p^{*}_{s, \alpha}} & \leq 2^{p^{*}_{s, \alpha}-1} \sum_{k = \ell_{0}}^{\ell-1}  |(v)_{\mathcal{A}_{k}} - (v)_{\mathcal{A}_{k+1}}|^{p^{*}_{s, \alpha}} \\ & \leq C \sum_{k = \ell_{0}}^{\ell-1} \fint_{\mathcal{A}_{k} \cup \mathcal{A}_{k+1}}  |v(x) - (v)_{\mathcal{A}_{k} \cup \mathcal{A}_{k+1}}|^{p^{*}_{s, \alpha}} \, dx, 
\end{align*}
where $C=C(N,s,p, \alpha_{1}, \alpha_{2})>0$. We apply Lemma \ref{Lemma: Weighted frac Sobolev ineq Lambda} with $n=4$ and $\lambda=2^k$. For
$x,y\in \mathcal{A}_{k}\cup \mathcal{A}_{k+1}$, we have 
\begin{equation*}
    |x|^{\frac{N-sp+ \alpha p}{2}}|y|^{\frac{N-sp+ \alpha p}{2}} \leq 2^{(k+2)(N-sp+ \alpha p)}.
\end{equation*}
It follows that
\begin{align*}
  &  |(v)_{\mathcal{A}_{\ell}} - (v)_{\mathcal{A}_{\ell_{0}}}|^{p^{*}_{s, \alpha}} \\ &  \leq C \sum_{k = \ell_{0}}^{\ell-1} \left( 2^{k (sp-\alpha p-N)} \int_{\mathcal{A}_{k} \cup \mathcal{A}_{k+1}} \int_{\mathcal{A}_{k} \cup \mathcal{A}_{k+1}} \frac{|v(x)-v(y)|^{p}}{|x-y|^{N+sp}} |x|^{\alpha_{1}p} |y|^{\alpha_{2} p} \, dx \, dy \right)^{\frac{p^{*}_{s, \alpha}}{p}} \\ & \leq C \sum_{k = \ell_{0}}^{\ell-1} \left( \int_{\mathcal{A}_{k} \cup \mathcal{A}_{k+1}} \int_{\mathcal{A}_{k} \cup \mathcal{A}_{k+1}} \frac{|v(x)-v(y)|^{p}}{|x-y|^{N+sp}} \, |x|^{- \frac{(N-\alpha_{1} p + \alpha_{2} p-sp)}{2}} |y|^{- \frac{(N+\alpha_{1} p - \alpha_{2} p-sp)}{2}} \, dx \, dy \right)^{\frac{p^{*}_{s, \alpha}}{p}} \\ &  \leq C \left( \int_{\mathbb{R}^{N}} \int_{\mathbb{R}^{N}} \frac{|v(x)-v(y)|^{p}}{|x-y|^{N+sp}} \, |x|^{- \frac{(N-\alpha_{1} p + \alpha_{2} p-sp)}{2}} |y|^{- \frac{(N+\alpha_{1} p - \alpha_{2} p-sp)}{2}} \, dx \, dy \right)^{\frac{p^{*}_{s, \alpha}}{p}} \\ & = C  \, \mathcal{R}(u)^{\frac{p^{*}_{s, \alpha}}{p}},
\end{align*}
where $C = C(N,s,p, \alpha_{1}, \alpha_{2})>0$. On the other hand, if
$x\in\mathcal{A}\cap\mathcal{A}_{\ell}$, then $v(x) > \mathcal{R}(u)^\frac{p^{*}_{s, \alpha}}{2 p^{2}}$. Hence,
\begin{equation*}
    \int_{\mathcal{A} \cap \mathcal{A}_{\ell}}  \, \frac{dx}{|x|^{N}} \leq \frac{1}{\mathcal{R}(u)^{\frac{p^{*}_{s, \alpha}}{2 p}}} \int_{\mathcal{A} \cap \mathcal{A}_{\ell}} \frac{|u(x)|^{p}}{|x|^{sp- \alpha p}} \, dx.
\end{equation*}
Combining the preceding two estimates, we obtain
\begin{align*}
    |(v)_{\mathcal{A}_{\ell}}  -  (v)_{\mathcal{A}_{\ell_{0}}}|^{p^{*}_{s, \alpha}}  \int_{\mathcal{A} \cap \mathcal{A}_{\ell}}  \, \frac{dx}{|x|^{N}} & \leq C \, \mathcal{R}(u)^{\frac{p^{*}_{s, \alpha}}{p}}   \frac{1}{\mathcal{R}(u)^{\frac{p^{*}_{s, \alpha}}{2 p}}} \int_{\mathcal{A} \cap \mathcal{A}_{\ell}} \frac{|u(x)|^{p}}{|x|^{sp-\alpha p}} \, dx \\ & \leq C \, \mathcal{R}(u)^{\frac{p^{*}_{s, \alpha}}{2 p}} \int_{\mathcal{A} \cap \mathcal{A}_{\ell}} \frac{|u(x)|^{p}}{|x|^{sp-\alpha p}} \, dx.
\end{align*}
Substitution into \eqref{ineq04} gives
\begin{align*}
\int_{\mathcal{A} \cap \mathcal{A}_{\ell}} & |v(x) - (v)_{\mathcal{A}_{\ell_{0}}}|^{p^{*}_{s, \alpha}} \, \frac{dx}{|x|^{N}} \\ & \leq C \left( \int_{\mathcal{A}_{\ell}} \int_{\mathcal{A}_{\ell}} \frac{|v(x)-v(y)|^{p}}{|x-y|^{N+sp}} \, |x|^{- \frac{(N-\alpha_{1} p + \alpha_{2} p-sp)}{2}} |y|^{- \frac{(N+\alpha_{1} p - \alpha_{2} p-sp)}{2}} \, dx \, dy \right)^{\frac{p^{*}_{s, \alpha}}{p}} \nonumber \\ & \quad +  C \, \mathcal{R}(u)^{\frac{p^{*}_{s, \alpha}}{2 p} } \int_{\mathcal{A} \cap \mathcal{A}_{\ell}} \frac{|u(x)|^{p}}{|x|^{sp-\alpha p}} \, dx.
\end{align*}
Summing over all $\ell$ such that
$\mathcal{A}\cap\mathcal{A}_{\ell}\neq\phi$, and using
\eqref{Ineq01}, the definition of $\mathcal{R}(u)$, and
Lemma \ref{Lemma: Hardy potential and Lorentz}, we arrive at
\begin{align*}
    \int_{\mathcal{A}} |v(x) - \mathcal{R}(u)^{\frac{p^{*}_{s, \alpha}}{2 p^{2}}}|^{p^{*}_{s, \alpha}} \, \frac{dx}{|x|^{N}} & \leq \sum_{\substack{\ell \in \mathbb{Z} \\ \mathcal{A} \cap \mathcal{A}_{\ell} \neq \phi}} \int_{\mathcal{A} \cap \mathcal{A}_{\ell}} |v(x) - (v)_{\mathcal{A}_{\ell_{0}}}|^{p^{*}_{s, \alpha}} \, \frac{dx}{|x|^{N}} \\ & \leq C \, \mathcal{R}(u)^{\frac{p^{*}_{s, \alpha}}{p}} + C \, \mathcal{R}(u)^{\frac{p^{*}_{s, \alpha}}{2 p} } \int_{\mathbb{R}^{N}} \frac{|u(x)|^{p}}{|x|^{sp- \alpha p}} \, dx \\ & \leq C \left(\mathcal{R}(u)^{\frac{p^{*}_{s, \alpha}}{p}} +  \mathcal{R}(u)^{ \frac{p^{*}_{s, \alpha}}{2 p}} \right),
\end{align*}
where the last inequality follows from
$\|u\|_{L^{p^{*}_{s,\alpha},p}(\mathbb{R}^N)}=1$.
Since $\mathcal{R}(u)\leq1$,
\begin{equation*}
    \mathcal{R}(u)^{\frac{p^{*}_{s,\alpha}}{p}}
    \leq
    \mathcal{R}(u)^{\frac{p^{*}_{s,\alpha}}{2p}},
\end{equation*}
and consequently
\begin{align*}
    \int_{\mathcal{A}} |v(x) - \mathcal{R}(u)^{\frac{p^{*}_{s, \alpha}}{2 p^{2}}}|^{p^{*}_{s, \alpha}} \, \frac{dx}{|x|^{N}} & \leq C \left( \mathcal{R}(u)^{\frac{p^{*}_{s, \alpha}}{p}} + \mathcal{R}(u)^{\frac{p^{*}_{s, \alpha}}{2p}} \right) \\ & \leq C \left( \mathcal{R}(u)^{\frac{p^{*}_{s, \alpha}}{2p}} + \mathcal{R}(u)^{\frac{p^{*}_{s, \alpha}}{2p} }  \right) \leq C \, \mathcal{R}(u)^{\frac{p^{*}_{s, \alpha}}{2p}} .
\end{align*}
Since $\omega_{s, \alpha}(x) = |x|^{- \frac{N-sp+ \alpha p}{p}}$, the Lorentz embedding gives (see Proposition \ref{Proposition on Lorentz space} and note that $L^{p^{*}_{s, \alpha}, p^{*}_{s, \alpha}} = L^{p^{*}_{s, \alpha}}$) 
\begin{align*}
    \| u - \mathcal{R}(u)^{\frac{p^{*}_{s, \alpha}}{2 p^{2}}} \omega_{s, \alpha} \|_{L^{p^{*}_{s, \alpha}, \infty}(\mathcal{A})} & \leq C \left( \int_{\mathcal{A}} | u(x) - \mathcal{R}(u)^{\frac{p^{*}_{s, \alpha}}{2 p^{2}}} \omega_{s, \alpha}(x) |^{p^{*}_{s, \alpha}} \, dx \right)^{\frac{1}{p^{*}_{s, \alpha}}} \\ &  = C \left( \int_{\mathcal{A}} |v(x) - \mathcal{R}(u)^{\frac{p^{*}_{s, \alpha}}{2 p^{2}}}|^{p^{*}_{s, \alpha}} \, \frac{dx}{|x|^{N}} \right)^{\frac{1}{p^{*}_{s, \alpha}}} \\ &  \leq C \, \mathcal{R}(u)^{\frac{1}{2p}},
\end{align*}
where the constant $C= C(N, s, p, \alpha_{1}, \alpha_{2})$ is independent of the set $\mathcal{A}$. Therefore,
\begin{equation*}
    \| u - \mathcal{R}(u)^{\frac{p^{*}_{s, \alpha}}{2 p^{2}}} \omega_{s, \alpha} \|^{2p}_{L^{p^{*}_{s, \alpha}, \infty}(\mathcal{A})} \leq C \, \mathcal{R}(u) .
\end{equation*}

It remains to estimate the same quantity on $\mathcal{A}^c$.
For $x\in\mathcal{A}^c$, $v(x)\leq
    \mathcal{R}(u)^{\frac{p^{*}_{s,\alpha}}{2p^2}}$, and hence 
\begin{equation*}
 \omega_{s, \alpha}(x)^{-1} |u(x)- \mathcal{R}(u)^{\frac{p^{*}_{s, \alpha}}{2 p^{2}}} \omega_{s, \alpha}(x)| =  |v(x) - \mathcal{R}(u)^{\frac{p^{*}_{s, \alpha}}{2 p^{2}}}| \leq 2 \mathcal{R}(u)^{\frac{p^{*}_{s, \alpha}}{2 p^{2}}}, \quad \forall \ x \in \mathcal{A}^{c}.
\end{equation*}
Thus,
\begin{equation*}
     |u(x)- \mathcal{R}(u)^{\frac{p^{*}_{s, \alpha}}{2 p^{2}}} \omega_{s, \alpha}(x)| \leq 2 \mathcal{R}(u)^{\frac{p^{*}_{s, \alpha}}{2 p^{2}}} \omega_{s, \alpha}(x), \quad \forall \ x \in \mathcal{A}^{c} .
\end{equation*}
Since
$\omega_{s,\alpha}\in
L^{p^{*}_{s,\alpha},\infty}(\mathbb{R}^N)$,
we obtain
\begin{equation*}
    \|u- \mathcal{R}(u)^{\frac{p^{*}_{s, \alpha}}{2 p^{2}}} \omega_{s, \alpha}\|_{L^{p^{*}_{s, \alpha}, \infty}(\mathcal{A}^{c})} \leq C \, \mathcal{R}(u)^{\frac{p^{*}_{s, \alpha}}{2 p^{2}}} \| \omega_{s, \alpha} \|_{L^{p^{*}_{s, \alpha}, \infty}(\mathcal{A}^{c})} \leq C \, \mathcal{R}(u)^{\frac{p^{*}_{s, \alpha}}{2p^{2}}} \leq C \, \mathcal{R}(u)^{\frac{1}{2p}},
\end{equation*}
where we have used $\mathcal{R}(u)\leq1$. Consequently,
\begin{equation*}
     \|u- \mathcal{R}(u)^{\frac{p^{*}_{s, \alpha}}{2 p^{2}}} \omega_{s, \alpha}\|^{2p}_{L^{p^{*}_{s, \alpha}, \infty}(\mathcal{A}^{c})} \leq C \, \mathcal{R}(u).
\end{equation*}
Combining the estimates on $\mathcal{A}$ and $\mathcal{A}^c$, we
obtain
\begin{align*}
    \|u- \mathcal{R}(u)^{\frac{p^{*}_{s, \alpha}}{2 p^{2}}} \omega_{s, \alpha}\|_{L^{p^{*}_{s, \alpha}, \infty}(\mathbb{R}^{N})} & \leq C \Big( \|u- \mathcal{R}(u)^{\frac{p^{*}_{s, \alpha}}{2 p^{2}}} \omega_{s, \alpha}\|_{L^{p^{*}_{s, \alpha}, \infty}(\mathcal{A})} \\ & \quad  + \|u- \mathcal{R}(u)^{\frac{p^{*}_{s, \alpha}}{2 p^{2}}} \omega_{s, \alpha}\|_{L^{p^{*}_{s, \alpha}, \infty}(\mathcal{A}^{c})} \Big) \\ & \leq C \, \mathcal{R}(u)^{\frac{1}{2p}}. 
\end{align*}
Therefore, combining the cases $\mathcal{R}(u) > 1$ and $\mathcal{R}(u) \leq 1$, and under the normalization 
\begin{equation*}
    \| u \|_{L^{p^{*}_{s, \alpha},p}(\mathbb{R}^{N})} =1,
\end{equation*}
we obtain
\begin{equation*}
   \inf_{a \geq 0} \|u- a\,  \omega_{s, \alpha}\|^{2p}_{L^{p^{*}_{s, \alpha}, \infty}(\mathbb{R}^{N})} \leq \|u-  \mathcal{R}(u)^{\frac{p^{*}_{s, \alpha}}{2 p^{2}}} \omega_{s, \alpha}\|^{2p}_{L^{p^{*}_{s, \alpha}, \infty}(\mathbb{R}^{N})} \leq C \, \mathcal{R}(u). 
\end{equation*}
We next remove the normalization. Set
\begin{equation*}
    \widetilde{u} = \frac{u}{\| u \|_{L^{p^{*}_{s, \alpha},p}(\mathbb{R}^{N})} }.
\end{equation*} 
Applying the normalized estimate to $\widetilde{u}$ and using the
homogeneity of the involved quantities, we obtain
\begin{align*}
  \inf_{a \geq 0} \frac{\| u -  a \, \omega_{s, \alpha}\|^{2p}_{L^{p^{*}_{s, \alpha},\infty}(\mathbb{R}^{N})}}{ \| u \|^{2p}_{L^{p^{*}_{s, \alpha},p}(\mathbb{R}^{N})} }  \leq C \frac{\mathcal{R}(u)}{ \| u \|^{p}_{L^{p^{*}_{s, \alpha},p}(\mathbb{R}^{N})}  }.   
\end{align*}
Hence, for every nonnegative $u \in C^{1}_{c}(\mathbb{R}^{N})$,
\begin{equation}\label{ineq1234}
    \inf_{a \geq 0} \| u - a \, \omega_{s, \alpha}\|^{2p}_{L^{p^{*}_{s, \alpha},\infty}(\mathbb{R}^{N})} \leq C \, \mathcal{R}(u)\| u \|^{p}_{L^{p^{*}_{s, \alpha},p}(\mathbb{R}^{N})}  .
\end{equation}
We finally pass to arbitrary real-valued functions. Let
$u\in C^{1}_c(\mathbb{R}^N)$ and write $u_{+}(x) = \max \{ u(x),0 \} $ and $u_{-}(x) = \max \{ -u(x),0 \}$, the positive and the negative parts of $u$ respectively, so that $u=u_{+}- u_{-}$. For $b,c\geq0$, the triangle inequality gives
\begin{align*}
  \inf_{a \in \mathbb{R}}  \| u  - a\, \omega_{s, \alpha}\|_{L^{p^{*}_{s, \alpha},\infty}(\mathbb{R}^{N})} & = \inf_{b,c \geq 0 }  \| u_{+} - u_{-} - (b-c)\, \omega_{s,\alpha}\|_{L^{p^{*}_{s, \alpha},\infty}(\mathbb{R}^{N})}  \\ & \leq \inf_{b \geq 0}  \| u_{+} -b \,  \omega_{s, \alpha}\|_{L^{p^{*}_{s, \alpha},\infty}(\mathbb{R}^{N})}  + \inf_{c \geq 0}  \| u_{-} - c\, \omega_{s, \alpha}\|_{L^{p^{*}_{s, \alpha},\infty}(\mathbb{R}^{N})}.
\end{align*}
Applying \eqref{ineq1234} to $u_+$ and $u_-$, respectively, and
using Lemma \ref{Lemma: Positive and Negative part}, we find
\begin{align*}
    \sum_{\pm} \inf_{a \geq 0} \| u_{\pm} - a\, \omega_{s, \alpha} \|_{L^{p^{*}_{s, \alpha}, \infty}(\mathbb{R}^{N})} & \leq C \sum_{\pm}  \left( \mathcal{R}(u_{\pm}) \right)^{\frac{1}{2p}} \| u_{\pm} \|^{\frac{1}{2}}_{L^{p^{*}_{s, \alpha},p}(\mathbb{R}^{N})}   \\ & \leq C \left( \mathcal{R}(u_{+})+ \mathcal{R}(u_{-}) \right)^{\frac{1}{2p}}  \| u \|^{\frac{1}{2}}_{L^{p^{*}_{s, \alpha},p}(\mathbb{R}^{N})}  \\ & \leq C \left( \mathcal{R}(u)  \right)^{\frac{1}{2p}}  \| u \|^{\frac{1}{2}}_{L^{p^{*}_{s, \alpha},p}(\mathbb{R}^{N})} .
\end{align*}
Combining the preceding estimates with Lemma
\ref{Lemma: Hardy potential and Lorentz}, we conclude that
\begin{align*}
\left( \frac{N}{\mathbb{S}^{N-1}} \right)^{\frac{sp- \alpha p}{N}} & \inf_{a \in \mathbb{R}} \frac{ \| u -a\, \omega_{s, \alpha}\|^{2p}_{L^{p^{*}_{s, \alpha},\infty}(\mathbb{R}^{N})}}{\| u \|^{2p}_{L^{p^{*}_{s, \alpha},p}(\mathbb{R}^{N})}}  \int_{\mathbb{R}^{N}} \frac{|u(x)|^{p}}{|x|^{sp-\alpha p}} \, dx \\ & \leq  \inf_{a \in \mathbb{R}} \frac{ \| u - a\, \omega_{s, \alpha}\|^{2p}_{L^{p^{*}_{s, \alpha},\infty}(\mathbb{R}^{N})}}{\| u \|^{2p}_{L^{p^{*}_{s, \alpha},p}(\mathbb{R}^{N})}}  \| u \|^{p}_{L^{p^{*}_{s, \alpha},p}(\mathbb{R}^{N})}  \leq C \, \mathcal{R}(u).
\end{align*}
Finally, by the definition of $\mathcal{D}_{s,p, \alpha}$ and $\mathcal{R}(u)$, together with the fractional Hardy inequality with remainder for $p\geq2$ given in
\eqref{Weighted Frac Hardy for p geq 2 with remainder}, we obtain

\begin{align*}
\delta_{s,p, \alpha}(u) \geq C 
    \left(\int_{\mathbb{R}^{N}} \frac{|u(x)|^{p}}{|x|^{sp-\alpha p}} \, dx \right) \mathcal{D}_{s,p, \alpha}(u)^{2p},
\end{align*}
for all $u \in C^{1}_c(\mathbb{R}^N)$.

\subsection{Proof of Theorem \ref{Theorem: Quantitative frac hardy 1 < p < 2}} A weighted fractional Hardy inequality with a remainder term in the case
$1<p<2$ was established by Dyda and Kijaczko in \cite[Theorem $2$]{Dyda2024JFA}.
More precisely, they proved that, for $1<p<2$, $s\in(0,1)$, and
$\alpha=\alpha_{1}+\alpha_{2}$, where $\alpha_{1}p,\, \alpha_{2}p,\, \alpha p\in(-N,sp)$, the following inequality holds for all $u\in C_c^1(\mathbb{R}^N)$ when $sp-\alpha p<N$,
and for all
$u\in C_c^1(\mathbb{R}^N\setminus\{0\})$ when $sp-\alpha p>N$,
\begin{align*}
    & \int_{\mathbb{R}^{N}}\int_{\mathbb{R}^{N}} \frac{|u(x)-u(y)|^{p}}{|x-y|^{N+sp}} |x|^{\alpha_{1} p} |y|^{\alpha_{2} p} \, dx \, dy  - \mathcal{C}  \int_{\mathbb{R}^{N}} \frac{|u(x)|^{p}}{|x|^{sp-\alpha p}} \, dx \nonumber \\  & \quad \quad \geq C_{p}  \int_{\mathbb{R}^{N}}\int_{\mathbb{R}^{N}} \frac{|v(x)^{\langle \frac{p}{2} \rangle}-v(y)^{\langle \frac{p}{2} \rangle}|^{2}}{|x-y|^{N+sp}} \mathcal{W}(x,y) \, |x|^{\alpha_{1}p} |y|^{\alpha_{2}p} \, dx \, dy, \quad v(x)^{\langle \frac{p}{2} \rangle} = |v(x)|^{\frac{p}{2}} \operatorname{sgn} v(x),
\end{align*}
where $v(x)=u(x)|x|^{\frac{N-sp+\alpha p}{p}}$ if $u$ is real-valued, while $v(x) =|u(x)|\,|x|^{\frac{N-sp+\alpha p}{p} } $ if $u$ is complex-valued. Moreover,
\begin{equation*}
   \mathcal{W}(x,y)
    =
    \min\left\{
        |x|^{-\frac{N-sp+\alpha p}{p}}, \, 
        |y|^{-\frac{N-sp+\alpha p}{p}}
    \right\}
    \max\left\{
        |x|^{-\frac{N-sp+\alpha p}{p}}, \, 
        |y|^{-\frac{N-sp+\alpha p}{p}}
    \right\}^{p-1},  
\end{equation*}
and
\begin{equation*}
   C_p
    =
    \max\left\{
        \frac{p-1}{p},
        \frac{p(p-1)}{2}
    \right\}.  
\end{equation*}
When $u\geq0$ or $u$ is complex-valued, the constant $C_p$ in the above inequality can be replaced by $p-1$.

\smallskip

Let $u \in C^{1}_{c}(\mathbb{R}^{N})$ with $u \geq 0$. By homogeneity, as in the proof of Theorem \ref{Theorem: Quantitative frac hardy p geq 2}, we may assume that
\begin{equation*}
    \| u \|_{L^{p^{*}_{s, \alpha},p}(\mathbb{R}^{N})} = 1.
\end{equation*}
Define the remainder term 
\begin{equation*}
    \mathcal{R}_{1}(u) : = \int_{\mathbb{R}^{N}}\int_{\mathbb{R}^{N}} \frac{|v(x)^{\langle \frac{p}{2} \rangle}-v(y)^{\langle \frac{p}{2} \rangle}|^{2}}{|x-y|^{N+sp}} \mathcal{W}(x,y) \, |x|^{\alpha_{1}p} |y|^{\alpha_{2}p} \, dx \, dy,
\end{equation*}
and let
\begin{equation*}
    \mathcal{W}^{p}_{r,s, \alpha}(x,y)
    =
    \min\left\{
        |x|^{-\frac{N-sp+\alpha p}{r}}, \, 
        |y|^{-\frac{N-sp+\alpha p}{r}}
    \right\}
    \max\left\{
        |x|^{-\frac{N-sp+\alpha p}{r}}, \,
        |y|^{-\frac{N-sp+\alpha p}{r}}
    \right\}^{r-1},  
\end{equation*}
where $r>1$. Note that $\mathcal{W}^{p}_{p,s, \alpha} = \mathcal{W}$. Similarly to the proof of Theorem~\ref{Theorem: Quantitative frac hardy p geq 2}, suppose that
$\mathcal{R}_{1}(u)>1$. By the continuous embedding $L^{p^{*}_{s, \alpha},p}(\mathbb{R}^{N}) \hookrightarrow L^{p^{*}_{s, \alpha},\infty}(\mathbb{R}^{N})$ (see Proposition \ref{Proposition on Lorentz space}), and using (see Remark \ref{Remark on Lorentz norm})
\begin{equation*}
   \| u \|_{L^{p^{*}_{s, \alpha},\infty}(\mathbb{R}^{N})} =  \left\| \left( u^{\frac{p}{2}} \right)^{\frac{2}{p}} \right\|_{L^{p^{*}_{s, \alpha}, \infty}(\mathbb{R}^{N})}= \left\|  u^{\frac{p}{2}}  \right\|^{\frac{2}{p}}_{L^{q_{s, \alpha}, \infty}(\mathbb{R}^{N})}, \quad \text{where} \quad  q_{s, \alpha}= \frac{2N}{N-sp+\alpha p},
\end{equation*} 
we obtain
\begin{align*}
   \mathcal{R}_{1}(u)^{\frac{1}{2p}} >1 & =  \| u \|_{L^{p^{*}_{s, \alpha},p}(\mathbb{R}^{N})}  \\ & \geq C \| u \|_{L^{p^{*}_{s, \alpha},\infty} (\mathbb{R}^{N})} = C \left\|  u^{\frac{p}{2}}  \right\|^{\frac{2}{p}}_{L^{q_{s, \alpha}, \infty}(\mathbb{R}^{N})} \geq C \inf_{a \geq 0}  \| u^{\frac{p}{2}} - a \, \omega^{\frac{p}{2}}_{s, \alpha} \|^{\frac{2}{p}}_{L^{q_{s, \alpha},\infty} (\mathbb{R}^{N})},
\end{align*} 
where $\omega_{s, \alpha}(x) =  \, |x|^{-\frac{N-sp+\alpha p}{p}}$ and $C= C(N,p,s, \alpha_{1}, \alpha_{2})>0$. Consequently,
\begin{equation}
  \inf_{a \geq 0}  \| u^{\frac{p}{2}} - a \, \omega^{\frac{p}{2}}_{s, \alpha} \|^{4}_{L^{q_{s, \alpha},\infty} (\mathbb{R}^{N})} \leq C \,  \mathcal{R}_{1}(u).
\end{equation}
Therefore, it remains to consider the case $\mathcal{R}_{1}(u) \leq 1$. Set 
\begin{equation*}
    \mathcal{B} := \left\{ x \in \mathbb{R}^{N} : \left(  v(x)\right)^{\frac{p}{2}} > \mathcal{R}_{1}(u)^{\frac{q_{s, \alpha}}{4p}} \right\}.
\end{equation*}
Since
$u\in C^{1}_c(\mathbb{R}^{N})$, the set $\mathcal{B}$ is open and bounded. 

For each $\ell \in \mathbb{Z}$, consider the dyadic annuli defined in \eqref{Defn: A_{ell}}. Moreover, $\lim_{|x| \to 0} v(x) = 0$. Hence, there exists a smallest $n_0\in\mathbb{Z}$ such that
\begin{equation*}
    ( v(x))^{\frac{p}{2}}\leq
    \mathcal{R}_{1}(u)^{\frac{q_{s, \alpha}}{4p}},
    \qquad\text{for every }x\in\mathcal{A}_{n_0},
\end{equation*}
while $\mathcal{B}\cap\mathcal{A}_{n_0+1}\neq\phi$. In particular, $ (v^{\frac{p}{2}})_{\mathcal{A}_{n_0}} \leq \mathcal{R}_{1} (u)^{\frac{q_{s, \alpha}}{4p}}$, where $(v^{\frac{p}{2}})_{\mathcal{A}_{n_0}}$ denotes the average of $v^{\frac{p}{2}}$ over $\mathcal{A}_{n_{0}}$. Thus, for every $x\in\mathcal{B}$,
\begin{equation*}
    0<
    (v(x))^{\frac{p}{2}}-\mathcal{R}_{1}(u)^{\frac{q_{s, \alpha}}{4p}}
    \leq
    (v(x))^{\frac{p}{2}}-(v^{\frac{p}{2}})_{\mathcal{A}_{n_0}}. 
\end{equation*}
It follows that
\begin{align}\label{ineq11}
    \int_{\mathcal{B}} |(v(x))^{\frac{p}{2}} - \mathcal{R}_{1}(u)^{\frac{q_{s, \alpha}}{4p}}|^{q_{s, \alpha}} \, \frac{dx}{|x|^{N}} & \leq \int_{\mathcal{B}} |(v(x))^{\frac{p}{2}} - (v^{\frac{p}{2}})_{\mathcal{A}_{n_{0}}}|^{q_{s, \alpha}} \, \frac{dx}{|x|^{N}} \nonumber \\ 
    & = \sum_{\substack{\ell \in \mathbb{Z} \\ \mathcal{B} \cap \mathcal{A}_{\ell} \neq \phi}} \int_{\mathcal{B} \cap \mathcal{A}_{\ell}} |(v(x))^{\frac{p}{2}} - (v^{\frac{p}{2}})_{\mathcal{A}_{n_{0}}}|^{q_{s, \alpha}} \, \frac{dx}{|x|^{N}},
\end{align}
where $q_{s, \alpha}= \frac{2N}{N-sp+ \alpha p}$. We now let
\begin{equation}\label{Defn:  s'}
  s' = \frac{sp}{2} \in (0,1), \quad \text{for} \quad 1<p<2 \quad \text{and} \quad s \in (0,1).
\end{equation}
For any $\ell \in \mathbb{Z}$ satisfying $\mathcal{B} \cap  A_{\ell} \neq \phi$, we have
\begin{align}\label{ineq12}
  \int_{\mathcal{B} \cap \mathcal{A}_{\ell}} |(v(x))^{\frac{p}{2}} - (v^{\frac{p}{2}})_{\mathcal{A}_{n_{0}}}|^{2^{*}_{s', \beta}} \, \frac{dx}{|x|^{N}} & \leq 2^{2^{*}_{s', \beta}-1} \int_{\mathcal{B} \cap \mathcal{A}_{\ell}} |(v(x))^{\frac{p}{2}} - (v^{\frac{p}{2}})_{A_{\ell}}|^{2^{*}_{s', \beta}} \, \frac{dx}{|x|^{N}} \nonumber \\ & \quad + 2^{2^{*}_{s', \beta}-1} |(v^{\frac{p}{2}})_{\mathcal{A}_{\ell}} - (v^{\frac{p}{2}})_{\mathcal{A}_{n_{0}}}|^{2^{*}_{s', \beta}} \int_{\mathcal{B} \cap \mathcal{A}_{\ell}}  \, \frac{dx}{|x|^{N}} \nonumber \\ & \leq 2^{2^{*}_{s', \beta}-1} \int_{\mathcal{A}_{\ell}} |(v(x))^{\frac{p}{2}} - (v^{\frac{p}{2}})_{\mathcal{A}_{\ell}}|^{2^{*}_{s', \beta}} \, \frac{dx}{|x|^{N}} \nonumber \\ & \quad + 2^{2^{*}_{s', \beta}-1} |(v^{\frac{p}{2}})_{\mathcal{A}_{\ell}} - (v^{\frac{p}{2}})_{\mathcal{A}_{n_{0}}}|^{2^{*}_{s', \beta}} \int_{\mathcal{B} \cap \mathcal{A}_{\ell}}  \, \frac{dx}{|x|^{N}}  ,
\end{align}
where $2^{*}_{s', \beta} = \frac{2N}{N-2s'+2 \beta}$ and $0<2s'- 2 \beta < N$ and $s' \in (0,1)$ defined in \eqref{Defn:  s'}.

By Lemma \ref{Lemma: Weighted frac Sobolev ineq Lambda} with $n=2$ and $\lambda = 2^{\ell}$, we obtain
\begin{align*}
   \int_{\mathcal{A}_{\ell}} & |(v(x))^{\frac{p}{2}} - (v^{\frac{p}{2}})_{\mathcal{A}_{\ell}}|^{2^{*}_{s', \beta}}  \, \frac{dx}{|x|^{N}}  \leq C \frac{1}{|\mathcal{A}_{\ell}|}\int_{\mathcal{A}_{\ell}} |(v(x))^{\frac{p}{2}} - (v^{\frac{p}{2}})_{\mathcal{A}_{\ell}}|^{2^{*}_{s', \beta}} \, dx \\ & \leq C \left( 2^{\ell(2s'-2 \beta-N)} \int_{\mathcal{A}_{\ell}} \int_{\mathcal{A}_{\ell}} \frac{|(v(x))^{\frac{p}{2}}-(v(y))^{\frac{p}{2}}|^{2}}{|x-y|^{N+2s'}} |x|^{2 \beta_{1}} |y|^{2 \beta_{2}} \, dx \, dy  \right)^{\frac{2^{*}_{s', \beta}}{2}} ,
\end{align*}
where $\beta =  \beta_{1} + \beta_{2} \geq 0$, $2 \beta_{1}, \, 2 \beta_{2} \in (-N, 2s')$ and $0<2s' -2 \beta < N$.
For all $x,y \in \mathcal{A}_{\ell}$, we have $|x| \leq 2^{\ell+1}$, $|y| \leq 2^{\ell+1}$, and hence
\begin{align*}
    2^{(\ell+1)(2s'-2\beta -N)} & \leq   \min \left\{ |x|^{-\frac{N-2s' + 2 \beta}{r}},  |y|^{-\frac{N-2s' + 2 \beta}{r}} \right\} \max \left\{ |x|^{-\frac{N-2s'+ 2 \beta}{r}},  |y|^{-\frac{N-2s'+ 2 \beta}{r}} \right\}^{r-1} \\ & = \mathcal{W}^{2}_{r,s', \beta}(x,y).
\end{align*}
Therefore, we arrive at
\begin{align*}
 \int_{\mathcal{A}_{\ell}} & |(v(x))^{\frac{p}{2}} - (v^{\frac{p}{2}})_{\mathcal{A}_{\ell}}|^{2^{*}_{s', \beta}}  \, \frac{dx}{|x|^{N}} \\ & \leq C \left( \int_{\mathcal{A}_{\ell}} \int_{\mathcal{A}_{\ell}} \frac{|(v(x))^{\frac{p}{2}}-(v(y))^{\frac{p}{2}}|^{2}}{|x-y|^{N+2s'}} \, \mathcal{W}^{2}_{r,s', \beta} (x,y) \, |x|^{2 \beta_{1}} |y|^{2 \beta_{2}} \, dx \, dy \right)^{\frac{2^{*}_{s', \beta}}{2}} ,
\end{align*}
where $C= C(N,s,p, \beta_{1}, \beta_{2})>0$. Substituting the above estimate into \eqref{ineq12}, we obtain
\begin{align*}
    \int_{\mathcal{B}  \cap \mathcal{A}_{\ell}} & |(v(x))^{\frac{p}{2}}  - (v^{\frac{p}{2}})_{\mathcal{A}_{n_{0}}}|^{2^{*}_{s', \beta}} \, \frac{dx}{|x|^{N}} \\ & \leq C \left( \int_{\mathcal{A}_{\ell}} \int_{\mathcal{A}_{\ell}} \frac{|(v(x))^{\frac{p}{2}}-(v(y))^{\frac{p}{2}}|^{2}}{|x-y|^{N+2s'}} \, \mathcal{W}^{2}_{r,s',\beta} (x,y) \, |x|^{2 \beta_{1}} |y|^{2 \beta_{2}} \, dx \, dy \right)^{\frac{2^{*}_{s', \beta}}{2}} \\ & \quad +  2^{2^{*}_{s', \beta}-1} |(v^{\frac{p}{2}})_{\mathcal{A}_{\ell}} - (v^{\frac{p}{2}})_{\mathcal{A}_{n_{0}}}|^{2^{*}_{s', \beta}} \int_{\mathcal{B} \cap \mathcal{A}_{\ell}}  \, \frac{dx}{|x|^{N}} . 
\end{align*}
In the above inequality, replacing $\beta_{1}$ and $\beta_{2}$ by
$\frac{\alpha_{1}p}{2}$ and $\frac{\alpha_{2}p}{2}$, respectively, and
using the definition of $s'$ given in \eqref{Defn: s'}, taking $r=p$, we obtain
\begin{align}\label{ineq123}
     \int_{\mathcal{B} \cap \mathcal{A}_{\ell}} |(v(x))^{\frac{p}{2}} & - (v^{\frac{p}{2}})_{\mathcal{A}_{n_{0}}}|^{q_{s, \alpha}} \, \frac{dx}{|x|^{N}} \nonumber \\ & \leq C \left( \int_{\mathcal{A}_{\ell}} \int_{\mathcal{A}_{\ell}} \frac{|(v(x))^{\frac{p}{2}}-(v(y))^{\frac{p}{2}}|^{2}}{|x-y|^{N+sp}} \, \mathcal{W} (x,y) |x|^{\alpha_{1}p} |y|^{\alpha_{2}p} \, dx \, dy \right)^{\frac{q_{s, \alpha}}{2}} \nonumber \\ & \quad +  C |(v^{\frac{p}{2}})_{\mathcal{A}_{\ell}} - (v^{\frac{p}{2}})_{A_{n_{0}}}|^{q_{s, \alpha}} \int_{\mathcal{B} \cap \mathcal{A}_{\ell}}  \, \frac{dx}{|x|^{N}} ,
\end{align}
where $q_{s, \alpha} = \frac{2N}{N-sp+ \alpha p}$. In this inequality, we have used the fact that $\mathcal{W}^{p}_{p,s',\beta} = \mathcal{W}$. Also, Using the fact that $\beta = \beta_{1}+\beta_{2} \geq 0$, $2 \beta_{1}, \, 2\beta_{2}, \,  \in (-N,2s')$, $0<2s' - 2 \beta <N$, and the defintion of $s'$ defined in \eqref{Defn:  s'}, we get
\begin{equation*}
 \alpha = \alpha_{1} + \alpha_{2}\geq 0, \quad  \alpha_{1}p, \,\alpha_{2}p \in (-N,sp), \quad \text{and} \quad 0<sp-\alpha p < N .  
\end{equation*}

\smallskip

Similarly, by following the same steps as in the case $p \geq 2$, we obtain for the case $1<p<2$,
\begin{align*}
 & |(v^{\frac{p}{2}})_{\mathcal{A}_{\ell}} - (v^{\frac{p}{2}})_{\mathcal{A}_{n_{0}}}|^{2^{*}_{s', \beta}}  \leq C \sum_{k = n_{0}}^{\ell-1} |(v^{\frac{p}{2}})_{\mathcal{A}_{k}} - (v^{\frac{p}{2}})_{\mathcal{A}_{k+1}}|^{2^{*}_{s', \beta}}   \\ & \leq C \sum_{\ell=n_{0}}^{\ell-1} \left( \int_{\mathcal{A}_{k} \cup\mathcal{A}_{k+1}} \int_{\mathcal{A}_{k} \cup \mathcal{A}_{k+1}} \frac{|(v(x))^{\frac{p}{2}}-(v(y))^{\frac{p}{2}}|^{2}}{|x-y|^{N+2s'}} \, \mathcal{W}^{2}_{r,s',\beta} (x,y) \, |x|^{2\beta_{1}} |y|^{2 \beta_{2}} \, dx \, dy \right)^{\frac{2^{*}_{s', \beta}}{2}} .
\end{align*}
Again, in the above inequality, by replacing $\beta_{1}$ and $\beta_{2}$ with $\frac{\alpha_{1} p}{2}$ and $\frac{\alpha_{2} p}{2}$ respectively, and
using the definition of $s'$ given in \eqref{Defn: s'}, taking $r=p$, and using the fact $\mathcal{W}^{p}_{p,s',\beta}=\mathcal{W}$,  we get
\begin{align*}
& |(v^{\frac{p}{2}})_{\mathcal{A}_{\ell}} - (v^{\frac{p}{2}})_{\mathcal{A}_{n_{0}}}|^{q_{s, \alpha}} \\ &  \leq C  \sum_{\ell=n_{0}}^{\ell-1} \left( \int_{\mathcal{A}_{k} \cup\mathcal{A}_{k+1}} \int_{\mathcal{A}_{k} \cup \mathcal{A}_{k+1}} \frac{|(v(x))^{\frac{p}{2}}-(v(y))^{\frac{p}{2}}|^{2}}{|x-y|^{N+sp}} \, \mathcal{W} (x,y) \,  |x|^{\alpha_{1}p} |y|^{\alpha_{2}p} \, dx \, dy \right)^{\frac{q_{s, \alpha}}{2}}  \\ & \leq C \left( \int_{\mathbb{R}^{N}} \int_{\mathbb{R}^{N}} \frac{|(v(x))^{\frac{p}{2}}-(v(y))^{\frac{p}{2}}|^{2}}{|x-y|^{N+sp}} \, \mathcal{W} (x,y) \, |x|^{\alpha_{1}p} |y|^{\alpha_{2}p} \, dx \, dy \right)^{\frac{q_{s, \alpha}}{2}} \\ & = C \,  \mathcal{R}_{1}(u)^{\frac{q_{s, \alpha}}{2}}.
\end{align*}
Also, using the fact that for any $x \in \mathcal{B} \cap \mathcal{A}_{\ell}$, we have $v(x) > \mathcal{R}_{1}(u)^{\frac{q_{s, \alpha}}{4p}}$, it follows that
\begin{equation*}
    \int_{\mathcal{B} \cap \mathcal{A}_{\ell}}  \, \frac{dx}{|x|^{N}} \leq \frac{1}{\mathcal{R}_{1}(u)^{\frac{q_{s, \alpha}}{4p} p}} \int_{\mathcal{B} \cap \mathcal{A}_{\ell}} \frac{|u(x)|^{p}}{|x|^{sp- \alpha p}} \, dx.
\end{equation*}
Therefore, combining the above two estimates, we obtain
\begin{align*}
  |(v^{\frac{p}{2}})_{\mathcal{A}_{\ell}} - (v^{\frac{p}{2}})_{\mathcal{A}_{n_{0}}}|^{q_{s, \alpha}}   \int_{\mathcal{B} \cap \mathcal{A}_{\ell}}  \, \frac{dx}{|x|^{N}} \leq C \,  \mathcal{R}_{1}(u)^{\frac{q_{s, \alpha}}{4} } \int_{\mathcal{B} \cap \mathcal{A}_{\ell}} \frac{|u(x)|^{p}}{|x|^{sp-\alpha p}} \, dx. 
\end{align*}
Therefore, combining the above inequality  with the inequality \eqref{ineq123}, the inequality \eqref{ineq11} reduces to 
\begin{align*}
   & \int_{ \mathcal{B}} |(v(x))^{\frac{p}{2}}  - \mathcal{R}_{1}(u)^{\frac{q_{s, \alpha}}{4p}}|^{q_{s, \alpha}} \, \frac{dx}{|x|^{N}} \\ & \leq \sum_{\substack{\ell \in \mathbb{Z} \\ \mathcal{B} \cap \mathcal{A}_{\ell} \neq \phi}}  \left( \int_{\mathcal{A}_{\ell}} \int_{\mathcal{A}_{\ell}} \frac{|(v(x))^{\frac{p}{2}}-(v(y))^{\frac{p}{2}}|^{2}}{|x-y|^{N+sp}} \, \mathcal{W} (x,y) \,  |x|^{\alpha_{1}p} |y|^{\alpha_{2}p} \, dx \, dy \right)^{\frac{q_{s, \alpha}}{2}}  \\ & \quad + C \mathcal{R}_{1}(u)^{\frac{q_{s, \alpha}}{4} }  \sum_{\substack{\ell \in \mathbb{Z} \\ \mathcal{B} \cap \mathcal{A}_{\ell} \neq \phi}} \int_{\mathcal{B} \cap \mathcal{A}_{\ell}} \frac{|u(x)|^{p}}{|x|^{sp-\alpha p}} \, dx \\ & \leq \left( \int_{\mathbb{R}^{N}} \int_{\mathbb{R}^{N}} \frac{|(v(x))^{\frac{p}{2}}-(v(y))^{\frac{p}{2}}|^{2}}{|x-y|^{N+sp}} \, \mathcal{W} (x,y) \, |x|^{\alpha_{1}p} |y|^{\alpha_{2}p} \, dx \, dy \right)^{\frac{q_{s, \alpha}}{2}} \\ & \quad + C \mathcal{R}_{1}(u)^{\frac{q_{s, \alpha}}{4} } \int_{\mathbb{R}^{N}} \frac{|u(x)|^{p}}{|x|^{sp-\alpha p}} \, dx . 
\end{align*}
Using Lemma \ref{Lemma: Hardy potential and Lorentz}, the normalization $\| u \|_{L^{p^{*}_{s, \alpha},p}(\mathbb{R}^{N})} = 1$, the definition of $\mathcal{R}_{1}(u)$, and using the fact $\mathcal{R}_{1}(u) \leq 1$, we obtain
 \begin{equation*}
     \int_{ \mathcal{B}} |(v(x))^{\frac{p}{2}}  - \mathcal{R}_{1}(u)^{\frac{q_{s, \alpha}}{4p}}|^{q_{s, \alpha}} \, \frac{dx}{|x|^{N}} \leq C\left(  \mathcal{R}_{1}(u)^{\frac{q_{s, \alpha}}{2}} + \mathcal{R}_{1}(u)^{ \frac{q_{s, \alpha}}{4} } \right) \leq C \,  \mathcal{R}_{1}(u)^{\frac{q_{s, \alpha}}{4}} .
 \end{equation*}
Therefore, using $\omega_{s, \alpha}(x)= |x|^{-\frac{N-sp+ \alpha p}{p}}$, and the Lorentz embedding gives (see Proposition \ref{Proposition on Lorentz space} and note that $L^{q_{s, \alpha}, q_{s, \alpha}} = L^{q_{s, \alpha}}$) 
\begin{align*}
    \| u^{\frac{p}{2}} - \mathcal{R}_{1}(u)^{\frac{q_{s, \alpha}}{4p}} \omega^{\frac{p}{2}}_{s, \alpha} \|_{L^{q_{s, \alpha}, \infty}(\mathcal{B})} & \leq C  \left( \int_{\mathcal{B}} | (u(x))^{\frac{p}{2}} - \mathcal{R}_{1}(u)^{\frac{q_{s, \alpha}}{4p}} (\omega_{s, \alpha}(x))^{\frac{p}{2}}|^{q_{s, \alpha}} \, dx \right)^{\frac{1}{q_{s, \alpha}}} \\ &  = C  \left( \int_{\mathcal{B}} |(v(x))^{\frac{p}{2}} - \mathcal{R}_{1}(u)^{\frac{q_{s, \alpha}}{4p}}|^{q_{s, \alpha}} \, \frac{dx}{|x|^{N}} \right)^{\frac{1}{q_{s, \alpha}}}  \leq C \mathcal{R}_{1}(u)^{\frac{1}{4}}.
\end{align*}
Consequently,
\begin{equation*}
  \inf_{a \geq 0}  \| u^{\frac{p}{2}} -  a \, \omega^{\frac{p}{2}}_{s, \alpha} \|^{4}_{L^{q_{s, \alpha}, \infty}(\mathcal{B})} \leq  \| u^{\frac{p}{2}} - \mathcal{R}_{1}(u)^{\frac{q_{s, \alpha}}{4p}} \omega^{\frac{p}{2}}_{s, \alpha} \|^{4}_{L^{q_{s, \alpha}, \infty}(\mathcal{B})} \leq C \mathcal{R}_{1}(u) .
\end{equation*}
For any $x \in \mathcal{B}^{c}$, we have $(v(x))^{\frac{p}{2}} \leq \mathcal{R}_{1}(u)^{\frac{q_{s, \alpha}}{4p}}$. In particular,
\begin{equation*}
 \omega_{s, \alpha}(x)^{-\frac{p}{2}} |(u(x))^{\frac{p}{2}}- \mathcal{R}_{1}(u)^{\frac{q_{s, \alpha}}{4}} (\omega_{s, \alpha}(x))^{\frac{p}{2}}| =  |(v(x))^{\frac{p}{2}} - \mathcal{R}_{1}(u)^{\frac{q_{s, \alpha}}{4p}}| \leq 2 \mathcal{R}_{1}(u)^{\frac{q_{s, \alpha}}{4p}}, \, \forall \ x \in \mathcal{B}^{c},
\end{equation*}
and therefore
\begin{equation*}
     |(u(x))^{\frac{p}{2}}- \mathcal{R}_{1}(u)^{\frac{q_{s, \alpha}}{4p}} (\omega_{s, \alpha}(x))^{\frac{p}{2}}| \leq 2 \mathcal{R}_{1}(u)^{\frac{q_{s, \alpha}}{4p}} (\omega_{s, \alpha}(x))^{\frac{p}{2}}, \quad \forall \ x \in \mathcal{B}^{c} .
\end{equation*}
Using that $\omega^{\frac{p}{2}}_{s, \alpha} \in L^{q_{s, \alpha}, \infty}(\mathbb{R}^{N})$, where $q_{s, \alpha}= \frac{2N}{N-sp+ \alpha p}$, and the assumption $\mathcal{R}_{1}(u) \leq 1$, we obtain
\begin{equation*}
    \|u^{\frac{p}{2}}- \mathcal{R}_{1}(u)^{\frac{q_{s, \alpha}}{4p}} \omega^{\frac{p}{2}}_{s, \alpha}\|_{L^{q_{s, \alpha}, \infty}(\mathcal{B}^{c})} \leq C \mathcal{R}_{1}(u)^{\frac{q_{s, \alpha}}{4p}} \| \omega^{\frac{p}{2}} \|_{L^{q_{s, \alpha}, \infty}(\mathcal{B}^{c})} \leq C \mathcal{R}_{1}(u)^{\frac{q_{s, \alpha}}{4p}} \leq C \mathcal{R}_{1}(u)^{\frac{1}{4}}. 
\end{equation*}
Consequently,
\begin{equation*}
     \|u^{\frac{p}{2}}- \mathcal{R}_{1}(u)^{\frac{q_{s, \alpha}}{4p}} \omega^{\frac{p}{2}}_{s, \alpha}\|^{4}_{L^{q_{s, \alpha}, \infty}(\mathcal{B}^{c})} \leq C \mathcal{R}_{1}(u).
\end{equation*}
Combining this with the estimate on $\mathcal{B}$, we conclude that
\begin{align*}
    \|u^{\frac{p}{2}}- \mathcal{R}_{1}(u)^{\frac{q_{s, \alpha}}{4p}} \omega^{\frac{p}{2}}_{s, \alpha}\|_{L^{q_{s, \alpha}, \infty}(\mathbb{R}^{N})} & \leq C \Big( \|u^{\frac{p}{2}}- \mathcal{R}_{1}(u)^{\frac{q_{s, \alpha}}{4p}} \omega^{\frac{p}{2}}_{s, \alpha}\|_{L^{q_{s, \alpha}, \infty}(\mathcal{B})} \\ & \quad \quad  + \|u^{\frac{p}{2}}- \mathcal{R}_{1}(u)^{\frac{q_{s, \alpha}}{4p}} \omega^{\frac{p}{2}}_{s, \alpha}\|_{L^{q_{s, \alpha}, \infty}(\mathcal{B}^{c})} \Big) \\ & \leq C \mathcal{R}_{1}(u)^{\frac{1}{4}}. 
\end{align*}
Therefore, combining the cases $\mathcal{R}_{1}(u) > 1$ and $\mathcal{R}_{1}(u) \leq 1$, and under the normalization $\| u \|_{L^{p^{*}_{s, \alpha},p}(\mathbb{R}^{N})} =1$, we obtain
\begin{equation*}
   \inf_{a \geq 0} \|u^{\frac{p}{2}}-  a\,  \omega^{\frac{p}{2}}_{s, \alpha}\|^{4}_{L^{q_{s, \alpha}, \infty}(\mathbb{R}^{N})} \leq \|u^{\frac{p}{2}}- \mathcal{R}_{1}(u)^{\frac{q_{s, \alpha}}{4p}} \omega^{\frac{p}{2}}_{s, \alpha}\|^{4}_{L^{q_{s, \alpha}, \infty}(\mathbb{R}^{N})} \leq C \mathcal{R}_{1}(u). 
\end{equation*}
To remove the normalization assumption $\| u \|_{L^{p^{*}_{s, \alpha},p}(\mathbb{R}^{N})} = 1$, we apply the previous estimate to the function $\widetilde{u}$, defined as in the previous subsection. Thus, for $u \geq 0$, we obtain
\begin{align*}
\inf_{a \geq 0} \|u^{\frac{p}{2}}-  a \, \omega^{\frac{p}{2}}_{s, \alpha}\|^{4}_{L^{q_{s, \alpha}, \infty}(\mathbb{R}^{N})}  \leq C \,  \mathcal{R}_{1}(u) \| u \|^{p}_{L^{p^{*}_{s, \alpha},p}(\mathbb{R}^{N})}.  
\end{align*}
Now, given any $u \in C^{1}_{c}(\mathbb{R}^{N})$, define $u^{\frac{p}{2}}_{+} = \max \{ (u(x))^{\langle \frac{p}{2}\rangle }, 0 \}$ and $u^{\frac{p}{2}}_{-} = \max \{ - (u(x))^{\langle \frac{p}{2}\rangle }, 0 \}$, the positive and negative parts of $u^{\langle \frac{p}{2}\rangle} $ respectively, so that $u^{\langle \frac{p}{2} \rangle} = u^{\frac{p}{2}}_{+} - u^{\frac{p}{2}}_{-}$. Then, following a similar approach to that used in the case $p \geq 2$ and using the definition of $\widetilde{\mathcal{D}}_{s,p}$, we obtain
\begin{align*}
     \delta_{s, p, \alpha} (u) \geq C \left( \int_{\mathbb{R}^{N}}  \frac{|u(x)|^{p}}{|x|^{sp- \alpha p}} \, dx \right)  \widetilde{\mathcal{D}}_{s,p}(u)^{4} ,
\end{align*}
for all $u \in C^{1}_{c}(\mathbb{R}^{N})$. This completes the proof of Theorem \ref{Theorem: Quantitative frac hardy 1 < p < 2}.

\section{Appendix}\label{Appendix}

In this appendix, we describe a local flattening transformation for
Lipschitz boundaries away from the origin. We shall also record the
properties of this transformation that are relevant for weighted
estimates involving the singular weight at the origin.

\subsection{Local flattening of the boundary away from the origin}\label{Appendix: Subsection}
\label{subsec:boundary-flattening} Let $\Omega\subset\mathbb{R}^{N}$ be a bounded Lipschitz domain and assume that $0\notin\partial\Omega$. Fix $x_{0}\in\partial\Omega$. Since $0\notin\partial\Omega$, we have $x_{0}\neq0$.

Since $x_{0}\neq0$ ($x_{0} =   (x_{0}', x_{0,N}) \in \mathbb{R}^{N-1} \times \mathbb{R}$), we may choose a coordinate direction that is not parallel to $x_{0}$ to serve as the new $x_{N}$-axis. Because $\Omega$ is Lipschitz, the boundary $\partial\Omega$ near $x_{0}$ can still be written as a graph over the remaining $N-1$ coordinates in this rotated frame, for any axis direction close enough to the original one. With this choice, the rotated coordinates of $x_{0}$ satisfy $x_{0}'\neq0$.

Choose \(r>0\) sufficiently small such that $r<\min\left\{\frac{|x_{0}|}{2},\frac{|x_{0}'|}{2}\right\}$. Then
\begin{equation*}
    0\notin\overline{B_{r}(x_{0})}.
\end{equation*}
More importantly, for every $x=(x',x_{N})\in\overline{B_{r}(x_{0})}$, we have
\begin{equation*}
    |x'|
\geq |x_{0}'|-|x'-x_{0}'|
\geq |x_{0}'|-|x-x_{0}|
\geq |x_{0}'|-r
>\frac{|x_{0}'|}{2}>0.
\end{equation*}
Hence
\begin{equation}\label{eq:xprime-nonzero}
x'\neq0
\qquad\text{for every }x\in\overline{B_{r}(x_{0})}.
\end{equation}

Since \(\Omega\) is a Lipschitz domain, after possibly decreasing \(r>0\),
there exists a Lipschitz function
\begin{equation*}
    \gamma:\mathcal{U}\subset\mathbb{R}^{N-1}\longrightarrow\mathbb{R}
\end{equation*}
such that
\begin{equation*}
  \partial\Omega\cap B_{r}(x_{0})
=
\left\{
(x',x_{N})\in B_{r}(x_{0}):x_{N}=\gamma(x')
\right\},  
\end{equation*}
and
\begin{equation*}
   \Omega\cap B_{r}(x_{0})
=
\left\{
(x',x_{N})\in B_{r}(x_{0}):x_{N}>\gamma(x')
\right\}. 
\end{equation*}

\begin{figure}[ht]
\centering

\begin{tikzpicture}[
    scale=1.2,
    >=Latex,
    every node/.style={font=\small}
]


\draw[->] (-0.2,0) -- (4.8,0) node[right] {$x'$};
\draw[->] (0,-0.2) -- (0,4.0) node[above] {$x_N$};

\fill (0,0) circle (1.5pt);
\node[below left] at (0,0) {$0$};

\draw[dashed] (2.65,2.25) circle (1.25);
\node at (3.80,3.40) {$B_r(x_0)$};


\draw[thick]
    (1.25,1.40)
    .. controls (1.60,1.65) and (1.95,2.00) ..
             (2.30,2.10)
    .. controls (2.65,2.25) and (2.95,2.05) ..
             (3.30,1.85)
    .. controls (3.65,1.65) and (4.00,1.70) ..
             (4.45,1.95);


\fill[
    pattern=north east lines,
    pattern color=gray!60
]
    (1.15,1.40)
    .. controls (1.60,1.65) and (1.95,2.00) ..
             (2.30,2.10)
    .. controls (2.65,2.25) and (2.95,2.05) ..
             (3.30,1.85)
    .. controls (3.65,1.65) and (4.00,1.70) ..
             (4.45,1.95)
    -- (4.45,3.45)
    -- (1.15,3.45)
    -- cycle;

\fill (2.65,2.25) circle (2pt);
\node[above right] at (2.50,2.25) {$x_0$};

\fill (2.30,2.50) circle (2pt);
\node[right] at (1.80,2.50) {$x$};

\node at (3.65,2.15) {$\partial\Omega$};

\node at (2.65,3.00) {$\Omega\cap B_r(x_0)$};

\node[font=\bfseries] at (2.65,3.85)
    {(a) Original coordinates};


\draw[very thick,gray!60]
    (5.25,-0.1) -- (5.25,3.75);


\draw[->] (5.75,0) -- (10.25,0) node[right] {$y'$};
\draw[->] (5.75,-0.15) -- (5.75,4.0) node[above] {$y_N$};

\fill (5.75,0) circle (1.5pt);
\node[below left] at (5.75,0) {$0$};


\fill[
    pattern=north east lines,
    pattern color=gray!60
]
    (6.40,0)
    .. controls (6.53,0.35) and (6.60,0.65) ..
             (6.85,0.82)
    .. controls (7.10,1.00) and (7.20,0.78) ..
             (7.45,1.05)
    .. controls (7.70,1.32) and (7.85,1.10) ..
             (8.10,1.28)
    .. controls (8.35,1.48) and (8.50,1.25) ..
             (8.75,1.40)
    .. controls (9.00,1.55) and (9.15,1.28) ..
             (9.38,1.08)
    .. controls (9.63,0.85) and (9.73,0.50) ..
             (9.80,0)
    -- cycle;


\draw[thick]
    (6.40,0)
    .. controls (6.53,0.35) and (6.60,0.65) ..
             (6.85,0.82)
    .. controls (7.10,1.00) and (7.20,0.78) ..
             (7.45,1.05)
    .. controls (7.70,1.32) and (7.85,1.10) ..
             (8.10,1.28)
    .. controls (8.35,1.48) and (8.50,1.25) ..
             (8.75,1.40)
    .. controls (9.00,1.55) and (9.15,1.28) ..
             (9.38,1.08)
    .. controls (9.63,0.85) and (9.73,0.50) ..
             (9.80,0);


\draw[thick]
    (6.40,0) -- (9.80,0);

\node[below] at (8.10,-0.04)
    {$y_N=0$};


\node at (8.50,0.58)
    {$T(\Omega\cap B_r(x_0))$};


\fill (8.05,1.02) circle (2pt);

\node[right] at (8.05,1.02)
    {$T(x)$};


\draw[<->]
    (5.75,0) -- (8.05,1.02);

\node[below right] at (7.00,0.5)
    {$|T(x)|>0$};


\draw[->,very thick]
    (4.55,3.15)
    .. controls (5.05,3.65) and (5.55,3.65) ..
    (6.10,3.15);


\node[align=center] at (8.5,1.9)
    {$\partial\Omega
    \displaystyle\longmapsto
    y_N=0$};


\node[align=center] at (8.10,-0.70)
    {$|T(x)|\sim |x|$};

\node[font=\bfseries] at (8.5,3.85)
    {(b) Flattened coordinates};

\end{tikzpicture}

\caption{
Local flattening of a Lipschitz boundary away from the origin.
}
\label{fig:boundary-flattening-away-origin}

\end{figure}
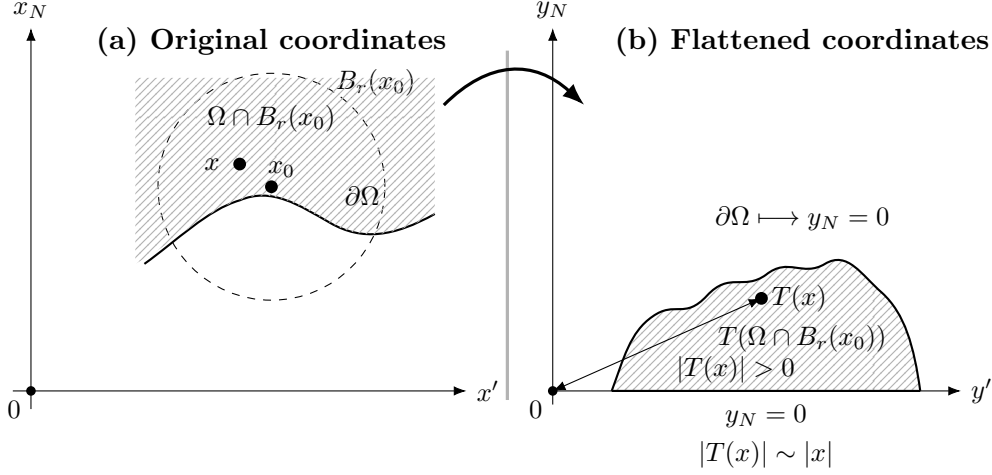

Define the flattening map
\begin{equation*}
   T:B_{r}(x_{0}) \longrightarrow Q
\end{equation*}
by
\begin{equation*}
T(x',x_{N})
=
\bigl(x',x_{N}-\gamma(x')\bigr).
\end{equation*}
Its inverse is given by
\begin{equation*}
T^{-1}(y',y_{N})
=
\bigl(y',y_{N}+\gamma(y')\bigr).
\end{equation*}
Since $\gamma$ is Lipschitz, $T$ is bi-Lipschitz. Also, $T(  \Omega \cap B_{r}(x_{0})) = Q_{+} := \{ x \in Q : x_{N}>0 \}$.

Indeed, if $x\in B_{r}(x_{0}) \cap \partial\Omega $, then
$x_N=\gamma(x')$, and therefore
\begin{equation*}
  T(x',x_N)
=
(x',0).  
\end{equation*}

We now verify that the flattening map does not map any point of the
domain to the origin. Let $x=(x',x_N)\in B_r(x_0)\cap\Omega$. By \eqref{eq:xprime-nonzero}, we have $x'\neq0$.

Since the first $N-1$ components of $T(x)$ are exactly
$x'$, we obtain
\begin{equation*}
  T(x)
=
\bigl(x',x_N-\gamma(x')\bigr)
\neq0.  
\end{equation*}
Thus
\begin{equation}\label{eq:Ttilde-nonzero}
T(x)\neq0
\qquad
\text{for every }x\in B_r(x_0)\cap\Omega.
\end{equation}

In particular, this conclusion holds pointwise for every point in the
whole neighborhood \(B_r(x_0)\cap\Omega\).

Finally, since $0\notin\overline{B_r(x_0)}$,
there exist constants $c_0, \, C_0>0$ depending on $r$ such that
\begin{equation*}
    0<c_0\leq |x|\leq C_0
\qquad
\text{for every }x\in B_r(x_0)\cap\Omega.
\end{equation*}
On the other hand, combining \eqref{eq:Ttilde-nonzero} with the identity $T(x',\gamma(x'))=(x',0)$ and \eqref{eq:xprime-nonzero}, we see that $T(x)\neq0$ for every $x\in\overline{B_r(x_0)}\cap\overline{\Omega}$. Since $\overline{B_r(x_0)}\cap\overline{\Omega}$ is compact and $T$ is continuous, the function $x\mapsto|T(x)|$ attains a positive minimum and a finite maximum on this set. Hence, there exist constants $c_1, C_1>0$ depending on $r$ such that
\begin{equation*}
    0<c_1\leq |T(x)|\leq C_1
\qquad
\text{for every }x\in B_r(x_0)\cap\Omega.
\end{equation*}
Consequently,
\begin{equation*}
   \frac{c_1}{C_0}|x|
\leq |T(x)|
\leq \frac{C_1}{c_0}|x|, 
\end{equation*}
and hence
\begin{equation}\label{eq:T-comparable}
|T(x)|\sim |x|
\qquad
\text{for every }x\in B_r(x_0)\cap\Omega.
\end{equation}

\bigskip

\textbf{Acknowledgement:} The author gratefully acknowledges the financial support of the Anusandhan National Research Foundation (ANRF) through the National Postdoctoral Fellowship (PDF/2025/004611). The author also thanks the Theoretical Statistics and Mathematics Unit, Indian Statistical Institute, Delhi Centre, India, for providing a supportive and stimulating research environment.


\end{document}